\pdfoutput=1
\documentclass[10pt,a4]{article}
\usepackage{amsthm,amsmath,amssymb}
\theoremstyle{definition}
\newtheorem{definition}{Definition}[section]
\newtheorem{lemdef}[definition]{Lemma \& Definition}
\newtheorem{construction}[definition]{Construction}

\theoremstyle{theorem}
\newtheorem{theorem}[definition]{Theorem}
\newtheorem{corollary}[definition]{Corollary}
\newtheorem{lemma}[definition]{Lemma}
\newtheorem{problem}[definition]{Problem}
\newtheorem{proposition}[definition]{Proposition}
\newtheorem{example}[definition]{Example}
\newtheorem{Remark}[definition]{Remark}

\newtheorem*{remark}{Remark}

\theoremstyle{theorem}
\usepackage{fullpage}
\usepackage{mathbbol}
\usepackage{calc}
\usepackage{bm}
\usepackage{mathdots}
\usepackage{mathrsfs}
\usepackage{tikz}
\usetikzlibrary{intersections, calc, math}
\usetikzlibrary{positioning}
\usetikzlibrary{decorations.pathreplacing}
\usepackage{tikz-cd}
\usepackage{amsmath}
\usepackage{mathtools}
\usepackage[enableskew]{youngtab}
\usepackage{ifthen}
\DeclareMathOperator{\id}{id}

\newcommand{\sijto}{\mathrel{\circ \hspace{-1.8mm} = \hspace{-2.2mm} > \hspace{-2.2mm} = \hspace{-1.8mm} \circ}}
\newcommand{\subsijto}{\mathrel{\circ \hspace{-0.8mm} = \hspace{-1.6mm} > \hspace{-1.6mm} = \hspace{-0.8mm} \circ}}

\newcommand{\sint}[2]{\underline{\left[{#1},{#2}\right)}}
\newcommand{\SPP}[2]{\operatorname{SPP}\left({#1};{#2}\right)}
\newcommand{\QTCPP}[2]{\operatorname{QTCPP}\left({#1};{#2}\right)}
\newcommand{\eSPP}[2]{\operatorname{e-SPP}\left({#1};{#2}\right)}
\newcommand{\stairPP}[2]{\operatorname{stair-PP}\left({#1};{#2}\right)}
\newcommand{\pstairPP}[2]{\operatorname{p-stair-PP}\left({#1};{#2}\right)}
\newcommand{\SSYT}[2]{\operatorname{SSYT}\left({#1};{#2}\right)}
\newcommand{\sh}[1]{\operatorname{sh}\left({#1}\right)}
\newcommand{\sgn}[1]{\operatorname{sgn}\left({#1}\right)}
\newcommand{\supp}[1]{\operatorname{supp}\left({#1}\right)}

\newcommand{\rw}[1]{\textcolor{red}{#1}}

\newcommand{\EllipseABb}[5]{
  \pgfmathsetmacro{\Cx}{(#1 + #3)/2}%
  \pgfmathsetmacro{\Cy}{(#2 + #4)/2}%
  \pgfmathsetmacro{\a}{veclen(#3-#1, #4-#2)/2}%
  \pgfmathsetmacro{\ang}{atan2(#4-#2, #3-#1)}%
  \begin{scope}[shift={(\Cx,\Cy)}, rotate=\ang]
    \draw[cyan, thick] ellipse [x radius=\a, y radius={#5}];
  \end{scope}%
}

\begin{document}

\title{
A bijection between symmetric plane partitions and quasi transpose complementary plane partitions
}
\author{
Takuya Inoue
\footnote{
Graduate School of Mathematical Sciences, the University of Tokyo,
3-8-1 Komaba, Meguro-ku, 153-8914 Tokyo, Japan.
Email: inoue@ms.u-tokyo.ac.jp
}
}
\date{}
\maketitle

\begin{abstract}
We resolve the explicit bijection problem between symmetric plane partitions (SPPs) and quasi transpose complementary plane partitions (QTCPPs), introduced by Schreier--Aigner, who proved their equinumerosity. 
First, we relate this problem to Proctor's parallel equinumerosities for SPPs, even SPPs, staircase plane partitions, and parity staircase plane partitions, by constructing several bijections. 
As a result, we reduce our task to constructing a \emph{compatible} bijection between even SPPs and staircase plane partitions. 
We then provide non--intersecting lattice path configurations for these objects, apply the LGV lemma, and transform the resulting path configurations. 
This process leads us to new combinatorial objects, \(I_m\) and \(J_m\), and the task is further reduced to constructing a compatible \emph{sijection} (signed bijection) \(I_m\sijto J_m\), which is carried out in the final part of this paper. 
Our construction also answers the 35--year--old open problem posed by Proctor: constructing an explicit bijection between even SPPs and staircase plane partitions.
\end{abstract}

\tableofcontents

\section{Introduction} \label{sec:intro}
The aim of bijective combinatorics is to construct explicitly described families of bijections, using combinatorial methods, between certain combinatorial objects whose cardinalities are already known to be equal, often by computational or algebraic techniques.
Such bijections can reveal additional and sometimes unexpected mathematical structures underlying known enumerative results, thereby providing insights that are not apparent from counting formulas alone.
In particular, in the context of enumeration problems related to plane partitions, several well-known open problems in bijective combinatorics remain.

This paper settles the explicit bijection problem between symmetric plane partitions (SPPs) and quasi transpose complementary plane partitions (QTCPPs).
In \cite{SAF24}, Schreier-Aigner proves the equinumerosity
\[
  \#\SPP{n}{M}=\#\QTCPP{n}{M},
\]
via determinant methods, but leaves the construction of a bijection open.
We construct an explicit bijection between \(\SPP{n}{M}\) and \(\QTCPP{n}{M}\) by considering a connection with a result of Proctor.
In \cite{Pro90}, Proctor (1990) establishes two parallel equinumerosities among four classes of plane partitions---SPPs, even symmetric plane partitions (even SPPs), staircase plane partitions (stair-PPs), and parity staircase plane partitions (p-stair-PPs):
\[
  \#\SPP{n}{M}=\#\pstairPP{n}{M}, \qquad
  \#\eSPP{n}{m}=\#\stairPP{n}{m},
\]
using representation–theoretic arguments.
The definitions of these four classes will be recalled in Section~2.
For the case of e-SPPs and stair-PPs, Sheats constructed a bijection in 1999~\cite{She99}. 
On the other hand, the other case still remains open.
In this paper we resolve this open problem, which has remained unsolved for 35 years, as well as the problem between SPPs and QTCPPs.

Let us outline our construction.
As we explain in Section~3, there is a simple bijection
\(\QTCPP{n}{M} \to \pstairPP{n}{M}\),
as well as bijections providing certain relations ---which we call \(1:2^{\#S}\) correspondences--- between
\(\pstairPP{n}{M}\) and \(\stairPP{n}{m}\), and between
\(\SPP{n}{M}\) and \(\eSPP{n}{m}\):
\begin{align*}
\SPP{n}{2m+1} &\to \eSPP{n}{m} \times \{0,1\}^n,&
\SPP{n}{2m} &\to \bigsqcup_{\pi' \in \eSPP{n}{m}} \{0,1\}^{S(\pi')}, \\
\pstairPP{n}{2m+1} &\to \stairPP{n}{m} \times \{0,1\}^n,&
\pstairPP{n}{2m} &\to \bigsqcup_{\pi' \in \stairPP{n}{m}} \{0,1\}^{S(\pi')}.
\end{align*}
As we will show, by these relations, the bijection problem between QTCPPs and SPPs is reduced to the problem of constructing a bijection
\[
  f:\ \eSPP{n}{m}\longrightarrow \stairPP{n}{m},
\]
together with bijections \(g_\pi:S(\pi)\to S\bigl(f(\pi)\bigr)\),
where \(S\) denotes a set--valued function on certain classes of plane partitions introduced in Section~3.
%where $S$ is a set-valued function on some classes of plane partitions in Section~3.
Such $f$ must be \textit{compatible} with the statistic $\# S$, according to the notion of compatibility introduced by the author in \cite{Inoue}.
It is important to note that Sheats's construction does not satisfy this compatibility condition.
We will provide a compatible construction by means of \textit{sijections} (signed bijections) in Sections~4--7.

\begin{figure}
\begin{center}
\begin{tikzpicture}[
  box/.style={draw=none, align=center},
  every node/.style={font=\small}
]

% カテゴリラベル
%\node at (-3,3.5) {\textbf{shifted staircase}};
%\node at (3.5,3.5) {\textbf{staircase}};

% 左側ノード
\node[box] (spp) at (-4,1) {\( \SPP{n}{M} \)};
\node[box] (espp) at (-4,-1) {\( \eSPP{n}{m} \)};

% 右側ノード（QTCPPとstairPPをx=2.5に揃える）
\node[box] (qtcpp) at (2.5,1) {\( \QTCPP{n}{M} \)};
\node[box] (pstair) at (5.4,1) {\( \pstairPP{n}{M} \)};
\node[box] (stairpp) at (2.5,-1) {\( \stairPP{n}{m} \)};

% equinumerous 横線
\draw[-] (spp.east) -- (qtcpp.west) node[midway, above] {equinumerosity by Aigner};
\draw[-] (espp) -- (-4,-2);
\draw[-] (-4,-2) -- (2.5,-2) node[midway, below] {equinumerosity by Proctor};
\draw[-] (2.5,-2) -- (stairpp);
\draw[<->,red] (qtcpp.east) -- (pstair.west) node[midway, above] {bij.};
\draw[<->,red] (espp.east) -- (stairpp.west) node[midway, above] {a compatible bij. w/$S$};

% 縦破線
%\draw[dashed] (0.5,-1.8) -- (0.5,3.3);

% 赤文字ラベルを矢印と重ならないようにずらす
\node[box, text=red] at (-2.3,0) {$1:2^S$ correspondence};
\node[box, text=red] at (3.7,0) {$1:2^S$ correspondence};

% 上下の赤い両向き矢印（中央）
\draw[<->, red, thick] (spp.south) -- (espp.north);
%\draw[<->, red, thick] (qtcpp.south) -- (stairpp.north);

% p-stair と stairPP を ┐ ┘型の折れ線矢印で結ぶ
\draw[<->, red, thick] (pstair.south) |- (stairpp.east);

%SPP-p-stairPP
\draw[-] (spp) -- (-4,2);
\draw[-] (-4,2) -- (5.4, 2) node[midway, above] {equinumerosity by Proctor};
\draw[-] (5.4, 2) -- (pstair);

% 外枠
\draw[thick] (-6,-3.2) rectangle (7,2.9);

\end{tikzpicture}
\caption{An outline of our construction. The red arrows indicate the bijections and relations established in this paper, which together form the desired bijection $\SPP{n}{M} \to \QTCPP{n}{M}$.}
\end{center}
\end{figure}

In Sections~4 and~5 we describe a non–intersecting path configuration for \(\stairPP{n}{m}\), based on the piles–of–cubes picture. 
Applying the LGV lemma \cite{GV,Lind} to such a configulation, we obtain a sijection
\[
  \stairPP{n}{m}
    \;\to\; \mathcal{P}_{\Gamma'}^{\mathrm{(NI)}}(\mathbf{a},\mathbf{b})
    \;\overset{\mathrm{LGV}}{\sijto}\;
    \mathcal{P}_{\Gamma'}(\mathbf{a},\mathbf{b}),
\]
where the non–intersection condition is relaxed.
After organizing this path configuration, we introduce a signed index set \(I_m\) together with statistics \(\eta_i\) on \(I_m\),
which allows us to obtain the following decomposition
\[
  \stairPP{n}{m}
  \;\sijto\;
  \bigsqcup_{\alpha\in I_m}
  \prod_{i=1}^m \mathcal{C}\bigl(2n,\; n+\lvert\eta_i\rvert(\alpha)\bigr).
\]

In Section~6 we then carry out a parallel analysis for \(\eSPP{n}{m}\). 
We begin by giving a path model with a movable boundary that encodes the condition defining \(\eSPP{n}{m}\). 
As in the staircase case, we then apply the LGV lemma \cite{GV,Lind} to pass from non–intersecting paths to general paths. 
By reorganizing the resulting path configuration and introducing another signed index set \(J_m\) together with statistics \(\eta_i\) on \(J_m\), we obtain a decomposition
\[
  \eSPP{n}{m}
  \;\sijto\;
  \bigsqcup_{\sigma\in J_m}
  \prod_{i=1}^m \mathcal{C}\bigl(2n,\; n+\lvert\eta_i\rvert(\sigma)\bigr).
\]
This expression is directly comparable with the decomposition for \(\stairPP{n}{m}\) obtained in Section~5.

Finally, in Section~7 we build a good sijection \(I_m \sijto J_m\) that matches a certain condition to provide a sijection
\[
  \bigsqcup_{\alpha\in I_m}
  \prod_{i=1}^m \mathcal{C}\bigl(2n,\; n+\lvert\eta_i\rvert(\alpha)\bigr)
  \sijto 
  \bigsqcup_{\sigma\in J_m}
  \prod_{i=1}^m \mathcal{C}\bigl(2n,\; n+\lvert\eta_i\rvert(\sigma)\bigr),
\]
which then completes the entire construction.
\medskip

\noindent\textit{Organization of this paper.}
Section~2 recalls the classes SPP, QTCPP, even–SPP, stair–PP, and p–stair–PP, and the relations proved in \cite{SAF24,Pro90}, and formulates the bijection problem. 
Section~3 explains the \(1:2^{\#S}\) correspondences and the reduction to the problem of constructing a compatible bijection between even SPPs and stair-PPs. 
Section~4 fixes notation on signed sets, sijections, and compatibility, and reviews the LGV framework. 
Section~5 develops the path model and transformations on the staircase side and introduces \(I_m\).
Section~6 does the same on the even–SPP side and introduces \(J_m\).
Section~7 relates \(I_m\) and \(J_m\) and assembles the constructions into the bijection between \(\SPP{n}{M}\) and \(\QTCPP{n}{M}\). 
Section~8 demonstrates an example of the full sijection.
Section~9 gives a short conclusion.

\section{Definitions}

\subsection{The SPP-QTCPP bijection problem}

In this section we describe the bijection problem that is the main subject of this paper. We begin by giving some basic definitions.

\begin{definition}
Let $\lambda = (\lambda_1 \geq \lambda_2 \geq \cdots \geq \lambda_r)$ be a partition.
A \textit{plane partition} of shape $\lambda$ is a $2$-dimensional array $\pi \colon \{(i,j) \mid 1 \leq i \leq r, 1 \leq j \leq \lambda_i \} \to \mathbb{Z}_{\geq 0}$ such that
\begin{itemize}
\item $\pi_{i,j} \geq \pi_{i,j+1}$ when $1 \leq i \leq r$ and $1 \leq j \leq \lambda_i-1$,
\item $\pi_{i,j} \geq \pi_{i+1,j}$ when $1 \leq i \leq r-1$ and $1 \leq j \leq \lambda_{i+1}$.
\end{itemize}
\end{definition}

\begin{definition}
Let $\lambda = (\lambda_1 > \lambda_2 > \cdots > \lambda_r)$ be a strict partition.
A \textit{shifted plane partition} of shape $\lambda$ is a $2$-dimensional array $\pi \colon \{(i,j) \mid 1 \leq i \leq r, i \leq j \leq i+\lambda_i-1 \} \to \mathbb{Z}_{\geq 0}$ such that
\begin{itemize}
\item $\pi_{i,j} \geq \pi_{i,j+1}$ when $1 \leq i \leq r$ and $i \leq j \leq i+\lambda_i-2$,
\item $\pi_{i,j} \geq \pi_{i+1,j}$ when $1 \leq i \leq r-1$ and $i+1 \leq j \leq i+\lambda_{i+1}$.
\end{itemize}
\end{definition}

For a positive integer $n$, a staircase plane partition of size $n$ is a plane partition of shape $\lambda = (n,n-1,\ldots,1)$.
Similarly for a shifted staircase plane partition of size $n$. For example, the following is a shifted staircase plane partition of size $4$:
\[
	\young(7533,:432,::21,:::0).
\]

\begin{definition}
Let $n \in \mathbb{Z}_{>0}$ and $M \in \mathbb{Z}_{\geq 0}$.
A \textit{symmetric plane partition (SPP)} of size $n$ with upper bound $M$ is a plane partition of shape $\lambda = (n)^n = (n,n,\ldots,n)$ such that $\pi_{1,1} \leq M$ and
\[
	\pi_{i,j} = \pi_{j,i} \quad\text{ for all }\quad 1 \leq i, j \leq n.
\]
\end{definition}

Considering an SPP of size $n$ is equivalent to considering a shifted staircase plane partition of size $n$.
Therefore, we define $\SPP{n}{M}$ as follows:
\[
	\SPP{n}{M} := \{ \pi \colon \text{a shifted staircase plane partition of size $n$ such that $\pi_{1,1} \leq M$} \}.
\]

Next, we define QTCPP, which is a class of plane partitions introduced by Schreier-Aigner \cite{SAF24}.

\begin{definition}
Let $n \in \mathbb{Z}_{>0}$ and $M \in \mathbb{Z}_{\geq 0}$.
A \textit{quasi transpose complementary plane partition (QTCPP)} of size $n$ with upper bound $M$ is a plane partition of shape $\lambda = (n)^n = (n,n,\ldots,n)$ such that 
\[
	\pi_{i,j} = M - \pi_{n+1-j,n+1-i} \quad\text{ if }\quad i+j \ne n+1.
\]
\end{definition}

For a QTCPP, we can recover the entries under the anti-diagonal line from the remaining entries.
We therefore identify it with a staircase plane partition of size $n$, subject to certain conditions.
First, we define the set of staircase plane partitions as follows:
\begin{equation} \label{def::stairPP}
	\stairPP{n}{m} := \{ \pi \colon \text{a staircase plane partition of size $n$ such that $\pi_{1,1} \leq m$} \}.
\end{equation}
Then, the set $\QTCPP{n}{M}$ is defined:
\begin{equation} \label{def::QTCPP}
	\QTCPP{n}{M} := \left\{ \pi \in \stairPP{n}{M} \,\middle|\, 
	\begin{array}{l}
	%\text{a staircase plane partition of size $n$ such that:} \\
	\text{$\pi$ satisfies the conditions:} \\
	%\text{(1) $\pi_{1,1} \leq M,$} \\
	\text{(1) $\pi_{i,j} \geq M-\pi_{i-1,j}$ when $i+j=n+1$ and $2 \leq i \leq n,$} \\
	\text{(2) $\pi_{i,j} \geq M-\pi_{i,j-1}$ when $i+j=n+1$ and $2 \leq j \leq n.$} 
	\end{array}
	\right\}
\end{equation}

For $\pi \in \QTCPP{n}{M}$, when $i+j < n+1$, we have
\begin{equation} \label{eq::QTCPP-M/2}
	\pi_{i,j} \leq \pi_{i,n+1-i} \leq M-\pi_{i,j}, 
\end{equation}
and thus $\pi_{i,j} \geq \frac{M}{2}$.
Schreier-Aigner shows the following equinumerosity in \cite{SAF24}.

\begin{theorem} \label{thm::SPP-QTCPP}
Let $n \in \mathbb{Z}_{>0}$ and $M \in \mathbb{Z}_{\geq 0}$. Then, we have
\[
	\# \SPP{n}{M} = \# \QTCPP{n}{M}.
\]
\end{theorem}

In \cite{SAF24}, this is proven by deriving a determinant formula for the number of QTCPPs and showing that it coincides with the well-known count of SPPs \cite{Stan86}.
In particular, the corresponding bijective problem was left open.

\begin{problem}
Let $n \in \mathbb{Z}_{>0}$ and $M \in \mathbb{Z}_{\geq 0}$.
Construct a bijection between $\SPP{n}{M}$ and $\QTCPP{n}{M}$.
\end{problem}

\subsection{Four classes of plane partitions considered by Proctor}
Proctor considers four classes of plane partitions in \cite{Pro90}, and showed two equinumerosities that are parallel in a certain sense.
In Proctor's paper, the proofs are given using methods from representation theory and by direct computations, and no explicit bijection is provided there.
However, one of the four classes he considers is SPP and some observations from Proctor's results play an important role in this paper.
We begin by introducing the four classes.

The first class consists of SPP, as mentioned above.
The second class consists of even symmetric plane partitions, which are defined as follows.

\begin{definition}
Let $n \in \mathbb{Z}_{>0}$ and $m \in \mathbb{Z}_{\geq 0}$.
An \textit{even symmetric plane partition (even SPP)} of size $n$, with upper bound $2m$, is an SPP $\pi$ of size $n$ with upper bound $2m$ such that
\[
	\pi_{i,i} \equiv 0 \pmod{2} \quad \text{for} \quad 1 \leq i \leq n.
\]
\end{definition}
The set of even SPPs is defined as follows:
\[
	\eSPP{n}{m} := \{ \pi \in \SPP{n}{2m} \mid \pi_{i,i} \equiv 0 \pmod{2} \,\text{ for all }\, 1 \leq i \leq n \}.
\]

The third class consists of staircase plane partitions, and the set of them is denoted by $\stairPP{n}{m}$ as defined in (\ref{def::stairPP}).
The fourth class consists of parity staircase plane partitions, which are staircase plane partitions satisfying a certain parity condition,
in analogy with even SPPs, the symmetric plane partitions satisfying a similar condition.

\begin{definition}
Let $n \in \mathbb{Z}_{>0}$ and $M \in \mathbb{Z}_{\geq 0}$.
A \textit{parity staircase plane partition} of size $n$ with upper bound $M$ is a staircase plane partition $\pi$ of size $n$ with upper bound $M$ such that 
\[
	\pi_{i,j} \equiv M \pmod{2} \quad \text{when} \quad i+j<n+1.
\]
\end{definition}

The set of parity staircase plane partitions is defined as:
\[
	\pstairPP{n}{M} := \{ \pi \in \stairPP{n}{M} \mid \pi_{i,j} \equiv M \pmod{2} \quad \text{when} \quad i+j<n+1 \}.
\]

Proctor shows the following two equinumerosities in \cite{Pro90}.
\begin{theorem}[Proctor 1990]
Let $n \in \mathbb{Z}_{>0}$ and $m, M \in \mathbb{Z}_{\geq 0}$. Then, it holds that
\begin{align*}
\# \SPP{n}{M} = \# \pstairPP{n}{M}, \\
\# \eSPP{n}{m} = \# \stairPP{n}{m}.
\end{align*}
\end{theorem}

In fact, Proctor gives a proof of a more refined version of this enumeration.
The definitions of the five classes of plane partitions introduced so far, along with Proctor’s result, suggest that there is a meaningful relationship between 
$\SPP{n}{M}$ and $\eSPP{n}{m}$, as well as a similar relationship between $\pstairPP{n}{M}$ and $\stairPP{n}{m}$.
Furthermore, in combination with Theorem \ref{thm::SPP-QTCPP}, this suggests the existence of a bijection between $\pstairPP{n}{M}$ and $\QTCPP{n}{M}$.
These relationships are summarized in the following diagram.

\begin{center}
\begin{tikzpicture}[
  box/.style={draw=none, align=center},
  every node/.style={font=\small}
]

% カテゴリラベル
\node at (-3,2) {\textbf{shifted staircase}};
\node at (3.5,2) {\textbf{staircase}};

% 左側ノード
\node[box] (spp) at (-4,1) {\( \SPP{n}{M} \)};
\node[box] (espp) at (-4,-1) {\( \eSPP{n}{m} \)};

% 右側ノード（QTCPPとstairPPをx=2.5に揃える）
\node[box] (qtcpp) at (2.5,1) {\( \QTCPP{n}{M} \)};
\node[box] (pstair) at (5.4,1) {\( \pstairPP{n}{M} \)};
\node[box] (stairpp) at (2.5,-1) {\( \stairPP{n}{m} \)};

% equinumerous 横線
\draw[-] (spp.east) -- (qtcpp.west) node[midway, above] {equinumerosity};
\draw[-] (espp.east) -- (stairpp.west) node[midway, below] {equinumerosity};
\draw[-] (qtcpp.east) -- (pstair.west) node[midway, above] {eq.};

% 縦破線
\draw[dashed] (0.5,-1.8) -- (0.5,1.8);

% 赤文字ラベルを矢印と重ならないようにずらす
\node[box, text=red] at (-2.8,0) {\textit{some relation?}};
\node[box, text=red] at (3.7,0) {\textit{some relation?}};

% 上下の赤い両向き矢印（中央）
\draw[<->, red, thick] (spp.south) -- (espp.north);
\draw[<->, red, thick] (qtcpp.south) -- (stairpp.north);

% p-stair と stairPP を ┐ ┘型の折れ線矢印で結ぶ
\draw[<->, red, thick] (pstair.south) |- (stairpp.east);

% 外枠
\draw[thick] (-6,-2.2) rectangle (7,2.4);

\end{tikzpicture}
\end{center}

We briefly review the previous attempts to construct the bijections implied by Proctor's result.
An explicit bijection between $\eSPP{n}{m}$ and $\stairPP{n}{m}$ is constructed by Sheats in 1999  \cite{She99}, using a \textit{symplectic jeu de taquin algorithm},
which is a complicated algorithm analogous to jeu de taquin \cite{Schu} (as cited in \cite{HM}).
On the other hand, for the equinumerosity between $\SPP{n}{M}$ and $\pstairPP{n}{M}$, no explicit bijection is known.
In \cite{Pech}, Pechenik conjectured that some algorithm related to a \text{K-theoretic jeu de taquin} would yield a bijection between $\SPP{n}{M}$ and $\pstairPP{n}{M}$;
however, this is not the case.
%We will give a counterexample in Appendix {\color{red}B}.

\section{$1 : 2^{\# S}$ correspondences} \label{Sec::1-2s}

In this section, we explain the relation between $\SPP{n}{M}$ and $\eSPP{n}{m}$, and between $\QTCPP{n}{M}$ and $\stairPP{n}{m}$.
In subsequent sections we construct a bijection between $\SPP{n}{M}$ and $\QTCPP{n}{M}$ by lifting a bijection between $\eSPP{n}{m}$ and $\stairPP{n}{m}$, using these relations.
For the lifting to work, we impose a certain condition on a bijection between $\eSPP{n}{m}$ and $\stairPP{n}{m}$.
As Sheats's construction does not satisfy this condition, we construct a new suitable bijection in this paper.

\subsection{A bijection between p-stair-PPs and QTCPPs}
First we construct a bijection between $\pstairPP{n}{M}$ and $\QTCPP{n}{M}$.
Although this bijection is not directly used in the construction of the bijection between SPPs and QTCPPs,
it is necessary for connecting our problem with Proctor’s results and for demonstrating that our construction also resolves the bijection problem arising from Proctor’s equinumerousities.
\begin{construction} \label{bij::pst-qtc}
Let $n \in \mathbb{Z}_{>0}$ and $m, M \in \mathbb{Z}_{\geq 0}$.
We define $f_{\ref{bij::pst-qtc}} \colon \pstairPP{n}{M} \to \QTCPP{n}{M}$ as follows.
If $n=1$, then we have $\pstairPP{n}{M} = \QTCPP{n}{M}$, and therefore we define $f_{\ref{bij::pst-qtc}}$ as the identity map. Otherwise, we define
\begin{equation} \label{construction::pst-qtc}
	f_{\ref{bij::pst-qtc}}(\pi)_{i,j} = \begin{cases}
		\frac{\pi_{i,j}+M}{2} & (i+j<n+1), \\
		\pi_{i,j}+\frac{M-m_{i,j}}{2} & (i+j=n+1),
	\end{cases}
\end{equation}
where 
\[
m_{i,j} = \begin{cases}
	\min(\pi_{i-1,j},\pi_{i,j-1}) & (i,j \ne 1), \\
	\pi_{i-1,j} & (j=1), \\
	\pi_{i,j-1} & (i=1).
\end{cases}
\]
Then, $f_{\ref{bij::pst-qtc}}$ is bijective and the inverse function is defined by
\begin{equation} \label{construction::qtc-pst}
	f_{\ref{bij::pst-qtc}}^{-1}(\pi')_{(i,j)} = \begin{cases}
		2 \pi'_{i,j} - M & (i+j<n+1), \\
		\pi'_{i,j}-M+m'_{i,j} & (i+j=n+1),
	\end{cases}
\end{equation}
where 
\[
m'_{i,j} = \begin{cases}
	\min(\pi'_{i-1,j},\pi'_{i,j-1}) & (i,j \ne 1), \\
	\pi'_{i-1,j} & (j=1), \\
	\pi'_{i,j-1} & (i=1).
\end{cases}
\]
\end{construction}

\begin{proof}
First we verify that $f_{\ref{bij::pst-qtc}}$ is well-defined.
Since $\pi_{i,j} \equiv M \pmod{2}$ for all $(i,j) \in \mathbb{Z}_{>0}^2$ such that $i+j<n+1$ and $0 \leq \pi_{i,j} \leq M$, it follows that $f_{\ref{bij::pst-qtc}}(\pi)_{i,j} \in \{0,1,\ldots,M\}$.
% and thus $f_{\ref{bij::pst-qtc}}(\pi) \in \stairPP{n}{M}$.
For all $(i,j) \in \mathbb{Z}_{>0}^2$ such that $i+j=n+1$ and $i \geq 2$, we have
\begin{align*}
	{f_{\ref{bij::pst-qtc}}(\pi)}_{i-1,j} - {f_{\ref{bij::pst-qtc}}(\pi)}_{i,j} &= \frac{\pi_{i-1,j}+M}{2} - \left( \pi_{i,j} + \frac{M-m_{i,j}}{2} \right) \\
	&= \frac{1}{2} \left( \left( \pi_{i-1,j} - \pi_{i,j} \right) + \left( m_{i,j} - \pi_{i,j} \right) \right) \\
	&\geq 0,
\end{align*}
and
\begin{align*}
	{f_{\ref{bij::pst-qtc}}(\pi)}_{i,j} - \left( M - {f_{\ref{bij::pst-qtc}}(\pi)}_{i-1,j} \right) &= \left( \pi_{i,j} + \frac{M-m_{i,j}}{2} \right) - \left( M - \frac{\pi_{i-1,j}+M}{2} \right) \\
	&= \pi_{i,j} + \frac{1}{2} \left( \pi_{i-1,j} - m_{i,j} \right) \\
	&\geq 0.
\end{align*}
Thus we have ${f_{\ref{bij::pst-qtc}}(\pi)}_{i-1,j} \leq {f_{\ref{bij::pst-qtc}}(\pi)}_{i,j} \leq M - {f_{\ref{bij::pst-qtc}}(\pi)}_{i-1,j}$.
Similarly, we can verify that ${f_{\ref{bij::pst-qtc}}(\pi)}_{i,j-1} \leq {f_{\ref{bij::pst-qtc}}(\pi)}_{i,j} \leq M - {f_{\ref{bij::pst-qtc}}(\pi)}_{i,j-1}$, for all $(i,j) \in \mathbb{Z}_{>0}^2$ such that $i+j=n+1$ and $j \geq 2$.
The other required inequalities are easy to verify, so we can conclude that ${f_{\ref{bij::pst-qtc}}(\pi)} \in \QTCPP{n}{M}$.

Next, we verify that $f^{-1}_{\ref{bij::pst-qtc}}$ defined as (\ref{construction::qtc-pst}) is well-defined.
Let $\pi' \in \QTCPP{n}{M}$ and $\pi = f^{-1}_{\ref{bij::pst-qtc}}(\pi')$. Then, for all $(i,j) \in \mathbb{Z}_{>0}^2$ such that $i+j<n+1$, we have $\pi_{i,j} = 2 \pi'_{i,j} - M$.
Since $\pi'_{i,j} \geq \frac{M}{2}$ from (\ref{eq::QTCPP-M/2}), it follows that $0 \leq \pi_{i,j} \leq M$ and $\pi_{i,j} \equiv M \pmod{2}$.
For all $(i,j) \in \mathbb{Z}_{>0}^2$ such that $i+j=n+1$ and $i \geq 2$, we have
\begin{align*}
	\pi_{i-1,j} - \pi_{i,j} &= (2\pi'_{i-1,j} - M) - (\pi'_{i,j} - M + m'_{i,j}) \\
	&= (\pi'_{i-1,j}-m'_{i,j}) + (\pi'_{i-1,j} - \pi'_{i,j}) \\
	&\geq 0.
\end{align*}
The other required inequalities can be verified similarly, so we can conclude that $\pi = f^{-1}_{\ref{bij::pst-qtc}}(\pi') \in \pstairPP{n}{M}$.

Last, we verify that $f^{-1}_{\ref{bij::pst-qtc}}(f_{\ref{bij::pst-qtc}}(\pi)) = \pi$ and $f_{\ref{bij::pst-qtc}}(f^{-1}_{\ref{bij::pst-qtc}}(\pi')) = \pi'$, for $\pi \in \pstairPP{n}{M}$ and $\pi' \in \QTCPP{n}{M}$.
Let $f:=f_{\ref{bij::pst-qtc}}$ and $g:=f^{-1}_{\ref{bij::pst-qtc}}$.
For all $(i,j) \in \mathbb{Z}_{>0}^2$ such that $i+j<n+1$, we have
\[
	(g(f(\pi)))_{i,j} = 2( f(\pi))_{i,j} - M = \pi_{i,j},
\]
and
\[
	(f(g(\pi'))_{i,j} = \frac{g(\pi')_{i,j} + M}{2} = \pi'_{i,j}.
\]
For all $(i,j) \in \mathbb{Z}_{>0}^2$ such that $i+j=n+1$ and $i,j \geq 2$, we have
\begin{align*}
	(g(f(\pi)))_{i,j} &= (f(\pi))_{i,j} - M + \min(f(\pi)_{i-1,j}, f(\pi)_{i,j-1}) \\
	&= \pi_{i,j} + \frac{M-m_{i,j}}{2} - M + \min\left( \frac{\pi_{i-1,j} + M}{2}, \frac{\pi_{i,j-1}+M}{2} \right) \\
	&= \pi_{i,j} + \frac{M-m_{i,j}}{2} - M + \frac{m_{i,j}+M}{2} = \pi_{i,j},
\end{align*}
and
\begin{align*}
	(f(g(\pi')))_{i,j} &= g(\pi')_{i,j} + \frac{M-\min(g(\pi')_{i-1,j}, g(\pi')_{i,j-1})}{2} \\
	&= \pi'_{i,j} - M + m'_{i,j} + \frac{M-\min(2 \pi'_{i-1,j} - M, 2 \pi'_{i,j-1} - M)}{2} \\
	&= \pi'_{i,j} - M + m'_{i,j} + \frac{2M-2m'_{i,j}}{2} = \pi'_{i,j}.
\end{align*}
The cases $(i,j)=(1,n)$ and $(n,1)$ can be verified in a similar manner. Therefore, we conclude that $f_{\ref{bij::pst-qtc}}$ is a bijection between $\pstairPP{n}{M}$ and $\QTCPP{n}{M}$.
\end{proof}
This bijection can also be understood in the following way.
First, we can verify, just as in the above proof, that it defines a bijection satisfying the required inequalities in the range $i+j<n+1$.
Once the values in this range are fixed, the values on the diagonal $i+j=n+1$ are constrained by
\[
	0 \leq \pi_{i,j} \leq m_{i,j},\qquad\text{and}\qquad M-m'_{i,j} \leq \pi'_{i,j} \leq m'_{i,j}.
\]
By definition, we have $m_{i,j} = 2m'_{i,j} - M$, so the number of possible values for $\pi_{i,j}$ and $\pi'_{i,j}$ are the same.
Hence, by shifting the values of $\pi'_{i,j}$ by $M-m'_{i,j}$, we obtain a bijection.

\begin{example}
The tableau $\,\vcenter{\hbox{\young(8663,440,42,1)}} \in \pstairPP{4}{8}$ corresponds to $\,\vcenter{\hbox{\young(8774,662,64,3)}} \in \QTCPP{4}{8}.$
\end{example}

\subsection{$1 : 2^{\# S}$ correspondences between QTCPPs and stair-PPs}
In this subsection, we explain what we shall call $1 : 2^{\# S}$ correspondences between QTCPPs and stair-PPs, including the definition of $S$.
First, we treat the case where $M$ is odd.

\begin{construction} \label{1-2s::qtc-sta}
Let $n \in \mathbb{Z}_{>0}$, $m \in \mathbb{Z}_{\geq 0}$ and $M=2m+1$.
We define the map
\[
\begin{tikzcd}
f_{\ref{1-2s::qtc-sta}} \arrow[r, phantom, "\colon"]&
\QTCPP{n}{M} \arrow[r] \arrow[d, phantom, "\rotatebox{90}{$\in$}"] 
  & \stairPP{n}{m} \times \{0,1\}^n \arrow[d, phantom, "\rotatebox{90}{$\in$}"]  \\ &
\pi \arrow[r, mapsto] & (\pi', t)
\end{tikzcd}
\]
by
\begin{align*}
\pi'_{i,j} &= \begin{cases}
	\pi_{i,j} - m - 1 & (i+j<n+1), \\
	\max(\pi_{i,j} - m - 1, m - \pi_{i,j} ) & (i+j=n+1),
\end{cases} \\
t_i &= \begin{cases}
	0 & (\pi_{i,n+1-i} \geq m+1), \\
	1 & (\pi_{i,n+1-i} < m+1).
\end{cases}	
\end{align*}
The map $f_{\ref{1-2s::qtc-sta}}$ is bijective, and its inverse map is given by
\begin{equation} \label{2s-1::sta-qtc}
	\pi_{i,j} = \begin{cases}
		m-\pi'_{i,j} & (i+j=n+1 \text{ and } t_i = 1), \\
		m+1+\pi'_{i,j} & (\text{otherwise}).
	\end{cases}
\end{equation}
\end{construction}

\begin{proof}
In the range $i+j<n+1$, since $\pi_{i,j} \geq \frac{M}{2}$ and $\pi_{i,j} \in \mathbb{Z}$, we have $\pi'_{i,j} = \pi_{i,j}-m-1 \geq 0$.
For all $(i,j) \in \mathbb{Z}_{>0}^2$ such that $i+j=n+1$ and $i \geq 2$, we have 
\begin{align*}
	\pi'_{i-1,j}-\pi'_{i,j} &= (\pi_{i-1,j}-m-1) - \max(\pi_{i,j}-m-1,m-\pi_{i,j}) \\
	&= \min(\pi_{i-1,j} - \pi_{i,j}, \pi_{i-1,j} + \pi_{i,j}-2m-1) \\
	&= \min(\pi_{i-1,j}-\pi_{i,j}, \pi_{i,j} - (M-\pi_{i-1,j})) \\
	&\geq 0.
\end{align*}
The other required inequalities can be verified similarly, so the map $f_{\ref{1-2s::qtc-sta}}$ is well-defined.
It can be easily verified that when $(\pi',t) \in \stairPP{n}{m} \times \{0,1\}^n$, $\pi$ defined by (\ref{2s-1::sta-qtc}) is in $\QTCPP{n}{M}$.
It can also be verified, element-wise, that the inversion of $f_{\ref{1-2s::qtc-sta}}$ is given by (\ref{2s-1::sta-qtc}).
\end{proof}

\begin{example}
The tableau $\,\vcenter{\hbox{\young(7643,553,52,4)}} \in \QTCPP{4}{7}$ corresponds to $$\,\left( \,\vcenter{\hbox{\young(3200,110,11,0)}}\,, (1,1,1,0) \right) \in \stairPP{4}{3} \times \{0,1\}^4.$$
Furthermore, when $M=7$, for example, the values of the entries in $\QTCPP{n}{M}$ and their corresponding values in $\stairPP{n}{m} \times \{0,1\}^n$ are organized as follows
(where $0$ and $1$ are replaced by $+$ and $-$, respectively):
\[
\begin{tikzpicture}[every node/.style={font=\small}, x=1cm, y=0.7cm]

% bracket for i+j < n+1
%\draw[decorate,decoration={brace,mirror,amplitude=4pt}] (0.5,3.8) -- (0.5,0.2) node[midway,left=3pt] {\( i + j < n + 1 \)};
% 上端
\draw[thick] (0.3,2.4) -- (0.5,2.4);
% 縦線
\draw[thick] (0.3,2.4) -- (0.3,0.0);
% 下端
\draw[thick] (0.3,0.0) -- (0.5,0.0);

% ラベル（必要なら少し下寄りに）
\node[left] at (-0, 1.2) {\( i + j < n + 1 \)};

% bracket for i+j = n+1
%\draw[decorate,decoration={brace,mirror,amplitude=4pt}] (0.7,3.8) -- (0.7,-3.8) node[pos=0.6,left=3pt] {\( i + j = n + 1 \)};
% 上端
\draw[thick] (0.7,2.4) -- (0.9,2.4);
% 縦線
\draw[thick] (0.7,2.4) -- (0.7,-2.7);
% 下端
\draw[thick] (0.7,-2.7) -- (0.9,-2.7);

% ラベル（必要なら少し下寄りに）
\node[left] at (0.5, -1.5) {\( i + j = n + 1 \)};

% QTCPP label
%\node at (1.6,3) {values in $\QTCPP{n}{M}$};

% right-hand side label
%\node at (6,3) {$\{0,1\}^n \simeq \{+,-\}^n$};
%\node at (10, 3) {values in $\stairPP{n}{M}$};

% arrows and entries
\foreach \val/\sign/\pi/\y in {
7/{$+$}/3/2.1,
6/{$+$}/2/1.5,
5/{$+$}/1/0.9,
4/{$+$}/0/0.3,
3/{$-$}/0/-0.5,
2/{$-$}/1/-1.1,
1/{$-$}/2/-1.7,
0/{$-$}/3/-2.3} {
  % Left values
  \node at (1.2,\y) {$\val$};
  % Arrow
  \draw[|->, thick] (2,\y) -- (3.5,\y);
  % t_i
  \node[text=blue] at (4,\y) {\sign};
  % pi'_{i,j}
  \node[text=red] at (4.5,\y) {\(\pi\)};
}

% t_i and pi'_{i,j} labels
\node at(1.2, -3.2) {\(\pi_{i,j}\)};
\node[text=blue] at (4,-3.2) {\( t_i \)};
\node[text=red] at (4.7,-3.2) {\( \pi'_{i,j}. \)};

\end{tikzpicture}
\]
\end{example}

Next, we consider the case where $M$ is even.
To begin, we transform the co-domain $\stairPP{n}{m} \times \{0,1\}^n$ (from the case where $M$ is odd) into the form:
\[
	\bigsqcup_{\pi' \in \stairPP{n}{m}} \{0,1\}^n.
\]
In the case $M$ is even, let $S(\pi') = \{ 1 \leq i \leq n \mid \pi'_{i,n+1-i} \ne 0 \}$. Then, we have a bijection
\[
	\QTCPP{n}{M} \to \bigsqcup_{\pi' \in \stairPP{n}{\frac{M}{2}}} \{0,1\}^{S(\pi')}.
\]

\begin{construction} \label{1-2s::qtc-sta-even}
Let $n \in \mathbb{Z}_{>0}$, $m \in \mathbb{Z}_{\geq 0}$ and $M=2m$.
We define the map
\[
\begin{tikzcd}
f_{\ref{1-2s::qtc-sta-even}} \arrow[r, phantom, "\colon"]&
\QTCPP{n}{M} \arrow[r] \arrow[d, phantom, "\rotatebox{90}{$\in$}"] 
  & {\displaystyle \bigsqcup_{\pi' \in \stairPP{n}{m}} \{0,1\}^{S(\pi')}} \arrow[d, phantom, "\rotatebox{90}{$\in$}"]  \\ &
\pi \arrow[r, mapsto] & (\pi', t \colon S(\pi') \to \{0,1\})
\end{tikzcd}
\]
by
\begin{align*}
\pi'_{i,j} &= \begin{cases}
	\pi_{i,j} - m & (i+j<n+1), \\
	\max(\pi_{i,j} - m, m - \pi_{i,j} ) & (i+j=n+1),
\end{cases} \\
t_i &= \begin{cases}
	0 & (i \in S(\pi') \text{ and } \pi_{i,n+1-i} > m), \\
	1 & (i \in S(\pi') \text{ and } \pi_{i,n+1-i} < m).
\end{cases}	
\end{align*}
Then, the map $f_{\ref{1-2s::qtc-sta-even}}$ is bijective, and the inversion is given by
\begin{equation}
	\pi_{i,j} = \begin{cases}
		m+\pi'_{i,j} & (i+j<n+1) \\
		m+(-1)^{t_i}\pi'_{i,j} & (i+j=n+1 \text{ and } i \in S(\pi')) \\
		m & (i+j=n+1 \text{ and } i \not\in S(\pi')).
	\end{cases}
\end{equation}
\end{construction}
We omit the proof, as it is obtained in the same way as in the case where $M$ is odd.

\begin{example}
The tableau $\,\vcenter{\hbox{\young(8753,664,62,5)}} \in \QTCPP{4}{8}$ corresponds to $$(\pi',t) = \left( \,\vcenter{\hbox{\young(4311,220,22,1)}}, (1,\cdot,1,0) \right) \in \stairPP{4}{4} \times \{0,1\}^{\{1,3,4\}}.$$
Here, since $\pi'_{2,3} = 0$ and the other entries on the anti-diagonal line of $\pi'$ are not $0$, we have $S(\pi') = \{1,3,4\}$, and $t=(1,\cdot,1,0)$ denotes the function $t \colon S(\pi') = \{1,3,4\} \to \{0,1\}$ such that
$t(1)=1$, $t(3)=1$ and $t(4)=0$.
In the case where $M=8$, the values of the entries in $\QTCPP{n}{M}$ and their corresponding values on the other side are organized as follows (where $0$ and $1$ are replaced by $+$ and $-$, respectively):
\[
\begin{tikzpicture}[every node/.style={font=\small}, x=1cm, y=0.7cm]

% bracket for i+j < n+1
%\draw[decorate,decoration={brace,mirror,amplitude=4pt}] (0.5,3.8) -- (0.5,0.2) node[midway,left=3pt] {\( i + j < n + 1 \)};
% 上端
\draw[thick] (0.3,3) -- (0.5,3);
% 縦線
\draw[thick] (0.3,3) -- (0.3,0.0);
% 下端
\draw[thick] (0.3,0.0) -- (0.5,0.0);

% ラベル（必要なら少し下寄りに）
\node[left] at (-0, 1.5) {\( i + j < n + 1 \)};

% bracket for i+j = n+1
%\draw[decorate,decoration={brace,mirror,amplitude=4pt}] (0.7,3.8) -- (0.7,-3.8) node[pos=0.6,left=3pt] {\( i + j = n + 1 \)};
% 上端
\draw[thick] (0.7,3) -- (0.9,3);
% 縦線
\draw[thick] (0.7,3) -- (0.7,-2.7);
% 下端
\draw[thick] (0.7,-2.7) -- (0.9,-2.7);

% ラベル（必要なら少し下寄りに）
\node[left] at (0.5, -1.5) {\( i + j = n + 1 \)};

% QTCPP label
%\node at (1.6,3) {values in $\QTCPP{n}{M}$};

% right-hand side label
%\node at (6,3) {$\{0,1\}^n \simeq \{+,-\}^n$};
%\node at (10, 3) {values in $\stairPP{n}{M}$};

% arrows and entries
\foreach \val/\sign/\pi/\y in {
8/{$+$}/4/2.7,
7/{$+$}/3/2.1,
6/{$+$}/2/1.5,
5/{$+$}/1/0.9,
4/{$\cdot$}/0/0.3,
3/{$-$}/0/-0.5,
2/{$-$}/1/-1.1,
1/{$-$}/2/-1.7,
0/{$-$}/3/-2.3} {
  % Left values
  \node at (1.2,\y) {$\val$};
  % Arrow
  \draw[|->, thick] (2,\y) -- (3.5,\y);
  % t_i
  \node[text=blue] at (4,\y) {\sign};
  % pi'_{i,j}
  \node[text=red] at (4.5,\y) {\(\pi\)};
}

% t_i and pi'_{i,j} labels
\node at(1.2, -3.2) {\(\pi_{i,j}\)};
\node[text=blue] at (4,-3.2) {\( t_i \)};
\node[text=red] at (4.7,-3.2) {\( \pi'_{i,j}. \)};

\end{tikzpicture}
\]
\end{example}

\subsection{$1 : 2^{\# S}$ correspondences between SPPs and even SPPs}
Next, we explain that the same relation exists between SPPs and even SPPs as the relation between QTCPPs and stair-PPs described in the previous subsection.
Namely, we construct the following bijections:
\begin{itemize}
\item $\SPP{n}{2m+1} \to \eSPP{n}{m} \times \{0,1\}^n$,
\item For any $\pi' \in \eSPP{n}{m}$, we define
\[
	S(\pi') := \{1 \leq j \leq n \mid \pi'_{1,j} \ne 2m \},
\]
and we construct a bijection
\[
	\SPP{n}{2m} \to \bigsqcup_{\pi' \in \eSPP{n}{m}} \{0,1\}^{S(\pi')}.
\]
\end{itemize}

First, we describe a bijection known as conjugation. Let $\SSYT{\lambda}{n}$ be the set of semi-standard Young tableaux of shape $\lambda$ with entries between $1$ and $n$.
We adopt the convention that entries are decreasing, rather than increasing.
We denote the shape of a semi-standard Young tableau $T$ by $\sh{T}$.

\begin{definition}
Let $n \in \mathbb{Z}_{>0}$. For partitions $\lambda = (\lambda_1 \geq \lambda_2 \geq \cdots \geq \lambda_n)$ and $\mu = (\mu_1 \geq \mu_2 \geq \cdots \geq \mu_n)$,
we write $\lambda \leq \mu$ to mean that $\lambda_i \leq \mu_i$ for all $i \in \{1,2,\ldots,n\}$.
\end{definition}

\begin{construction}
Let $n \in \mathbb{Z}_{>0}$, $m \in \mathbb{Z}_{\geq 0}$.
The bijection $\displaystyle \phi_\text{conj} \colon \SPP{n}{m} \to \bigsqcup_{\lambda \leq (m)^n} \SSYT{\lambda}{n}$ is defined as follows:
\begin{gather*}
\sh{\phi_\text{conj}(\pi)} = (\pi_{1,1}, \pi_{2,2}, \ldots, \pi_{n,n}), \\
\phi_\text{conj}(\pi)_{i,j} = \# \{ i \leq k \leq n \mid \pi_{i,k} \geq j \}.
\end{gather*}
Furthermore, the inverse map is described as follows:
\begin{quote}
To any $(T,\lambda) \in \bigsqcup_{\lambda \leq (m)^n} \SSYT{\lambda}{n}$, there corresponds a shifted staircase plane partition $\phi_\text{conj}^{-1}\bigl( (T,\lambda) \bigr)$ of size $n$ such that
\[
	\phi_\text{conj}^{-1}\bigl( (T,\lambda) \bigr)_{i,j} = \# \{ 1 \leq k \leq \lambda_i \mid T_{i,k} \geq j-i+1 \}.
\]
\end{quote}
\end{construction}

\begin{proof}
First we verify that $\phi_\text{conj}(\pi) \in \bigsqcup_{\lambda \leq (m)^n} \SSYT{\lambda}{n}$.
Since $\pi_{1,1} \leq m$, it holds that $(\pi_{1,1}, \pi_{2,2}, \ldots, \pi_{n,n}) \leq (m)^n$.
Let $S_{i,j} =  \{ i \leq k \leq n \mid \pi_{i,k} \geq j \}$.
For all applicable $(i,j)$s, we have $1 \leq \#S_{i,j} \leq n$, since $j \leq \pi_{i,i}$.
For all applicable $(i,j)$s, $\phi_\text{conj}(\pi)_{i,j} \geq \phi_\text{conj}(\pi)_{i,j+1}$ holds, because $S_{i,j} \subseteq S_{i,j+1}$ by definition.

We now check that $\phi_\text{conj}(\pi)_{i-1,j} > \phi_\text{conj}(\pi)_{i,j}$.
Let $k \in S_{i,j}$. Then, we have $\pi_{i-1,k} \geq \pi_{i,k} \geq j$, so $k \in S_{i-1,j}$.
Furthermore, we have $i-1 \in S_{i-1,j}$. Thus, we have $S_{i,j} \cup \{ i-1 \} \subseteq S_{i-1,j}$, which means that $\phi_\text{conj}(\pi)_{i-1,j} > \phi_\text{conj}(\pi)_{i,j}$.

Next, we verify that $\phi_\text{conj}^{-1}\bigl( (T,\lambda) \bigr) \in \SPP{n}{m}$.
Let $S'_{i,j} = \{ 1 \leq k \leq \lambda_i \mid T_{i,k} \geq j-i+1 \}$.
Since $S'_{i,j} \supseteq S'_{i,j+1}$ by definition, we have $\phi_\text{conj}^{-1}\bigl( (T,\lambda) \bigr)_{i,j} \geq \phi_\text{conj}^{-1}\bigl( (T,\lambda) \bigr)_{i,j+1}$.
We check that $\phi_\text{conj}^{-1}\bigl( (T,\lambda) \bigr)_{i-1,j} \geq \phi_\text{conj}^{-1}\bigl( (T,\lambda) \bigr)_{i,j}$.
Let $k \in S'_{i,j}$. Then, we have $T_{i,k} \geq j-i+1$. Since $T_{i-1,k} \geq T_{i,k}+1 = j-i+2 = j-(i-1)+1$, it holds that $k \in S'_{i-1,j}$.
Therefore, we have $S'_{i,j} \subseteq S'_{i-1,j}$, which means that $\phi_\text{conj}^{-1}\bigl( (T,\lambda) \bigr)_{i-1,j} \geq \phi_\text{conj}^{-1}\bigl( (T,\lambda) \bigr)_{i,j}$.

Last, since both $\phi_\text{conj}$ and $\phi_\text{conj}^{-1}$ act by taking the conjugation of a partition in each row, we have
$\phi_\text{conj} \circ \phi_\text{conj}^{-1} = \id$ and $\phi_\text{conj}^{-1} \circ \phi_\text{conj} = \id$.
\end{proof}

\begin{remark}
Since $\phi_\text{conj}(\pi)$ is actually independent of $n$ and $m$, we do not use notations such as  $\phi_\text{conj}^{(n,m)}$.
In addition, we call $\phi_\text{conj}(\pi)$ the conjugate of $\pi$.
\end{remark}

\begin{example} \label{ex::3-2-1}
The conjugate of $\,\vcenter{\hbox{\young(4433,:333,::31,:::0)}}$ is $\,\vcenter{\hbox{\young(4442,333,211)}}$.
\end{example}

The restriction of $\phi_\text{conj} \colon \SPP{n}{2m} \to \bigsqcup_{\lambda \leq (2m)^n} \SSYT{\lambda}{n}$ to $\eSPP{n}{m}$ maps into 
\[
	\bigsqcup_{\lambda \leq (2m)^n \text{ s.t. } \lambda_i \equiv 0 \pmod{2} \text{ for all $1 \leq i \leq n$ } } \SSYT{\lambda}{n}.
\]
Therefore, in the case $M$ is odd, it is sufficient to construct a bijection
\begin{equation} \label{bij::SSYT-odd}
\bigsqcup_{\lambda \leq (2m+1)^n} \SSYT{\lambda}{n} \to
\bigsqcup_{\lambda \leq (2m)^n \colon \text{even} } \SSYT{\lambda}{n} \times \{0,1\}^n,
\end{equation}
where we say a partition $\lambda$ is even if all elements of $\lambda$ are even.

We will give a construction of this bijection using the algorithm known as \emph{row insertion} or \emph{row bumping}.  
First, we review some properties of bumping following Fulton's book \cite{Fulton}.  
Let \( T \) be a semi-standard Young tableau and \( x \) a positive integer.  
We write $ T \leftarrow x $ to denote the tableau obtained by inserting \( x \) into \( T \) using the row insertion algorithm, which is defined as follows:
\begin{quote}
Starting from the first row, insert \( x \) into the row by replacing (or ``bumping'') the leftmost entry \( y \) that is strictly less than \( x \). Then insert \( y \) into the next row in the same manner. If no such entry exists in a row, place \( x \) at the end of that row and terminate the process.
\end{quote}
The result is a new tableau \( T \leftarrow x \), which remains semi-standard and has one more box than \( T \).  
We refer to the final position where the new box is added as the \emph{insertion position}.

We can also define the inverse of the row insertion algorithm.  
Given a tableau \( T' (= T \leftarrow x) \) and $\sh{T}$, we can recover \( T \) and the inserted integer \( x \) by reversing the bumping process.
Therefore, the row insertion algorithm defines a bijection 
\[ \SSYT{\lambda}{n} \times \{1,2,\ldots,n\} \to \bigsqcup_{\lambda' \geq \lambda \text{ s.t. }\lvert \lambda' \rvert = \lvert \lambda \rvert+1} \SSYT{\lambda'}{n}. \]

We state the row bumping lemma (see \S 1.1 of  \cite{Fulton} for a proof).
\begin{lemma}
Let $T$ be a semi-standard Young tableau and let $x$ and $x'$ be positive integers with $x<x'$.
Let $B$ be the insertion position of $T \leftarrow x$ and $B'$ the insertion position of $(T \leftarrow x) \leftarrow x'$.
Then, $B'$ is weakly to the left of and strictly below $B$.
\end{lemma}
The following proposition is a consequence of this lemma (see \S 1.1 of  \cite{Fulton} for a proof).
\begin{proposition} \label{prop::row-insertion}
Let $T$ be a semi-standard Young tableau of shape $\lambda$ and $x_1 < x_2 < \cdots < x_p$ positive integers.
Let $\mu = \sh{(((T \leftarrow x_1) \leftarrow x_2) \leftarrow \cdots) \leftarrow x_p}$.
Then, $\mu \setminus \lambda$ has at most one box in each row.

Conversely, let $U$ be a semi-standard Young tableau of shape $\mu$, $\lambda$ a partition such that $\lambda \leq \mu$ and $\mu \setminus \lambda$ has at most one box in each row,
and $p$ the number of boxes in $\mu \setminus \lambda$.
Then, there exists a unique tableau $T$ of shape $\lambda$ as well as positive integers $x_1 < x_2 < \cdots < x_p$ such that $U = (((T \leftarrow x_1) \leftarrow x_2) \leftarrow \cdots) \leftarrow x_p$.
\end{proposition}

\begin{remark}
In the context of bijective combinatorics, the unique existence of $T$ and $x$’s is not sufficient.
However, we can recover the values of $T$ and $x$’s from $U$ and $\lambda$, by repeatedly applying the inverse of the row insertion algorithm with boxes in $\mu \setminus \lambda$ from bottom to top.
\end{remark}

Using this proposition, we can construct the desired bijection (\ref{bij::SSYT-odd}) as follows.
\begin{construction} \label{construction::SSYT-odd}
Let $n \in \mathbb{Z}_{>0}$ and $m \in \mathbb{Z}_{\geq 0}$.
We construct a bijection
\[
f_\text{\ref{construction::SSYT-odd}} \colon 
\bigsqcup_{\mu \leq (2m+1)^n} \SSYT{\lambda}{n} \to
\bigsqcup_{\lambda \leq (2m)^n \colon \text{even} } \SSYT{\lambda}{n} \times \{0,1\}^n.
\]
Let $(U,\mu) \in \bigsqcup_{\mu \leq (2m+1)^n} \SSYT{\mu}{n}$.
Let $\lambda$ be an even partition such that $\lambda_i \in \{ \mu_i, \mu_i-1 \}$. Then, $\lambda \leq (2m)^n$.
Furthermore, applying Proposition \ref{prop::row-insertion}, we have a unique tableau $T \in \SSYT{\lambda}{n}$ and positive integers $x_1 < x_2 < \cdots <x_p \leq n$
such that $U = (((T \leftarrow x_1) \leftarrow x_2) \leftarrow \cdots) \leftarrow x_p$.

We define a function $t \colon \{1,2,\ldots,n\} \to \{0,1\}$ by
\[
	t(i) = \begin{cases} 1 & \text{ if $i \in \{ x_1, x_2, \ldots, x_p \}$, } \\ 0 & \text{ otherwise. } \end{cases}
\]
Then, we define $f_\text{\ref{construction::SSYT-odd}}\bigl((U,\mu)\bigr) := ((T,\lambda),t)$.

The inversion $f_\text{\ref{construction::SSYT-odd}}^{-1}$ is given as follows.
Let $((T,\lambda),t) \in \bigsqcup_{\lambda \leq (2m)^n \colon \text{even} } \SSYT{\lambda}{n} \times \{0,1\}^n$.
Let $\{ x_1 < x_2 < \cdots < x_p \} = \{ 1 \leq i \leq n \mid t(i) = 1 \}$ and $U = (((T \leftarrow x_1) \leftarrow x_2) \leftarrow \cdots) \leftarrow x_p$.
Then, $f_\text{\ref{construction::SSYT-odd}}^{-1}\bigl( ((T,\lambda),t) \bigr) = (U, \sh{U})$.
Note that Proposition \ref{prop::row-insertion} implies that $\sh{U} \setminus \lambda$ has at most one box in each row. Therefore, $\sh{U} \leq (2m+1)^n$.
\end{construction}

\begin{proof}
According to the remark of Proposition \ref{prop::row-insertion}, the maps $f_\text{\ref{construction::SSYT-odd}}$ and $f_\text{\ref{construction::SSYT-odd}}^{-1}$ are well-defined and explicitly constructed.
Apparently, $f_\text{\ref{construction::SSYT-odd}}^{-1} \circ f_\text{\ref{construction::SSYT-odd}} = \id$.
By the uniqueness of $T$ and $x$’s, we also have $f_\text{\ref{construction::SSYT-odd}} \circ f_\text{\ref{construction::SSYT-odd}}^{-1} = \id$.
\end{proof}

\begin{example} \label{ex::3-2-2}
Let $U = \,\vcenter{\hbox{\young(4442,333,211)}}$ and $\mu = \sh{U} = (4,3,3)$.
Then, $\lambda = (4,2,2)$ and $\mu \setminus \lambda = \{ (2,3), (3,3) \}$.
First, we apply the inverse of the row insertion with the box $(3,3)$.
Then, we have $$U = \,\vcenter{\hbox{\young(4432,331,21)}} \leftarrow 4.$$
Next, we apply the inverse of the row insertion with the box $(2,3)$, and we have $$U = \left( \,\vcenter{\hbox{\young(4431,33,21)}} \leftarrow 4 \right) \leftarrow 2.$$
Therefore, $T = \,\vcenter{\hbox{\young(4431,33,21)}}$, $t(2)=t(4)=1$ and $t(1)=t(3)=0$.
\end{example}

We now proceed to the case where $M$ is even.
This is achieved by restricting $f_\text{\ref{construction::SSYT-odd}}$ to $\bigsqcup_{\mu \leq (2m)^n} \SSYT{\lambda}{n}$.
First, we extend the definition of $S(\pi')$ for $\pi' \in \eSPP{n}{m}$ to $S(T)$ for $T \in \SSYT{\lambda}{n}$, where $\lambda$ is even, through the bijection $\phi_\text{conj}$. %\footnote{Apparently, $S$ can be extended to $T \in \SSYT{\lambda}{n}$, where $\lambda$ is not necessarily even, and $S(\pi')$ to $\pi' \in \SPP{n}{M}$. However, they are useless at least in this paper.}
Namely, we define $S(T) := S(\phi_\text{conj}^{-1}(T))$, which is identical with $\{T_{1,2m}+1, T_{1,2m}+2, \ldots, n\}$, where we consider $T_{1,2m}=0$ if $\lambda_1 < 2m$ and $T_{1,0} = n$.
Then, the following proposition provides the desired bijection.

\begin{proposition} In the same setting as Construction \ref{construction::SSYT-odd}, it holds that
\[
	f_\text{\ref{construction::SSYT-odd}}\left(\bigsqcup_{\mu \leq (2m)^n} \SSYT{\lambda}{n}\right)
	= \left\{ ((T,\lambda),t) \in \bigsqcup_{\lambda \leq (2m)^n \colon \text{even} } \SSYT{\lambda}{n} \times \{0,1\}^n \mid t^{-1}(1) \subseteq S(T) \right\},
\]
where the RHS is in bijection to $\displaystyle \bigsqcup_{(T,\lambda) \in \bigsqcup_{\lambda \leq (2m)^n \colon \text{even} } \SSYT{\lambda}{n}} \{0,1\}^{S(T)}$.
\end{proposition}
\begin{proof}
Let $(U,\mu) \in \bigsqcup_{\mu \leq (2m+1)^n} \SSYT{\lambda}{n}$ and $f_\text{\ref{construction::SSYT-odd}}\left((U,\mu)\right) = ((T,\lambda),t)$.
Furthermore, let $U = (((T \leftarrow x_1) \leftarrow x_2) \leftarrow \cdots) \leftarrow x_p$, where $x_1<x_2<\cdots<x_p$. Then, $t^{-1}(1) \subseteq S(T)$ is equivalent to $T_{1,2m}+1 \leq x_1$ or $t \equiv 0$ (namely, $U=T$ and $x_1$ does not exist).

($\subseteq$)
If $\mu$ is even, $U=T$ and $t \equiv 0$. Otherwise, $\mu \setminus \lambda$ is not empty and $x_1$ is well-defined.
Therefore, it is sufficient to show that if $\mu_1 \leq 2m$ and $\mu$ is not even, then $T_{1,2m}+1 \leq x_1$.
Then, the first row of $T$ is the result of the inverse of the row insertion algorithm, where $x_1$ is bumped out.\footnote{Note that when $\mu_1 \leq 2m+1$, this is not the case. Then, according to the remark of Proposition \ref{prop::row-insertion}, the first row of $T$ is the result of just removing the last element, since the box $(1,2m+1)$ is included in $\mu \setminus \lambda$.}
Let $(T \leftarrow x_1)_{1,b}$ be the position of the bumped element $x_1$ before the action. Then, we have $T_{1,2m} \leq T_{1,b} < x_1$.

($\supseteq$)
We assume that $t^{-1}(1) \subseteq S(T)$.
If $t \equiv 0$, then we have $\mu = \lambda \leq (2m)^n$.
Otherwise, we have $T_{1,2m}+1 \leq x_1$. Since the first row of $T$ has an element less than $x_1$, $\sh{T \leftarrow x_1}_1 = \sh{T}_1$.
Similarly, since the first row of $(((T \leftarrow x_1) \leftarrow x_2) \leftarrow \cdots ) \leftarrow x_k$ has an element at most $x_k$, which is less than $x_{k+1}$,
$\sh{((((T \leftarrow x_1) \leftarrow x_2) \leftarrow \cdots ) \leftarrow x_k) \leftarrow x_{k+1}}_1 = \sh{(((T \leftarrow x_1) \leftarrow x_2) \leftarrow \cdots ) \leftarrow x_k}_1$.
By induction, we have $\sh{U}_1 = \sh{T}_1$, which means $\mu_1 = \lambda_1 \leq 2m$.
\end{proof}

\begin{construction} \label{construction::spp-espp}
Let $n \in \mathbb{Z}_{>0}$ and $m \in \mathbb{Z}_{\geq 0}$. We have bijections 
\begin{itemize}
\item $(\phi_\text{conj}^{-1} \times \id_{\{0,1\}^n} )\circ f_\text{\ref{construction::SSYT-odd}} \circ \phi_\text{conj} \colon \SPP{n}{2m+1} \to \eSPP{n}{m} \times \{0,1\}^n = \bigsqcup_{\pi' \in \eSPP{n}{m}} \{0,1\}^{n}$,
\item the restriction of the above bijection to $\SPP{n}{2m} \colon \SPP{n}{2m} \to \bigsqcup_{\pi' \in \eSPP{n}{m}} \{0,1\}^{S(\pi')}$.
\end{itemize}
\[
\begin{tikzcd}[row sep=large, column sep=huge]
  % Row 1
  \SPP{n}{2m}
    \arrow[d, "{\phi_{\mathrm{conj}}}"']
    \arrow[r, hookrightarrow]
  & \SPP{n}{2m+1}
    \arrow[d, "{\phi_{\mathrm{conj}}}"]
  \\[2pt]
  % Row 2
  \displaystyle \bigsqcup_{\mu \le (2m)^n} \SSYT{\mu}{n}
    \arrow[d, "{f_{\ref{construction::SSYT-odd}}}"']
    \arrow[r, hookrightarrow]
  & \displaystyle \bigsqcup_{\mu \le (2m+1)^n} \SSYT{\mu}{n}
    \arrow[d, "{f_{\ref{construction::SSYT-odd}}}"]
  \\[2pt]
  % Row 3
  \displaystyle \bigsqcup_{\substack{\lambda \le (2m)^n \colon \text{even}}}
    \SSYT{\lambda}{n} \times \{0,1\}^{S(T)}
    \arrow[d, "{\phi_{\mathrm{conj}}^{-1} \times \id}"']
    \arrow[r, hookrightarrow]
  & \displaystyle \bigsqcup_{\substack{\lambda \le (2m)^n \colon \text{even}}}
    \SSYT{\lambda}{n} \times \{0,1\}^{n}
    \arrow[d, "{\phi_{\mathrm{conj}}^{-1} \times \id}"]
  \\[2pt]
  % Row 4
  \displaystyle \bigsqcup_{\pi' \in \eSPP{n}{m}} \{0,1\}^{S(\pi')}
    \arrow[r, hookrightarrow]
  & \eSPP{n}{m} \times \{0,1\}^{n}
\end{tikzcd}
\]

\end{construction}

\begin{example} \label{ex::3-2-3}
Let $\pi = \,\vcenter{\hbox{\young(4433,:333,::31,:::0)}} \in \SPP{4}{4}$. According to Examples \ref{ex::3-2-1} and \ref{ex::3-2-2},
$$(f_\text{\ref{construction::SSYT-odd}} \circ \phi_\text{conj})(\pi) = \left(\left(\,\vcenter{\hbox{\young(4431,33,21)}}, (4,2,2) \right), t\right),$$ where $t$ is the same as in Example \ref{ex::3-2-2}.
Since $$\pi' := \phi_\text{conj}^{-1}\left(\left(\,\vcenter{\hbox{\young(4431,33,21)}},(4,2,2) \right) \right) = \,\vcenter{\hbox{\young(4332,:222,::21,:::0)}} \in \eSPP{4}{2},$$
$S(\pi') = \{2,3,4\} \supseteq t^{-1}(1)$. Therefore, $\pi$ corresponds to $(t \mid_{S(\pi')}, \pi')$.
\end{example}

\begin{remark}
In fact, this bijection can also be described by means of a certain algorithm on SPPs, rather than through $\phi_{\mathrm{conj}}$.
In this algorithm, we define disjoint sequences of boxes, each starting from a box in 
$S(\pi)$ and ending at a cell on the diagonal line, and decrement the entries of all boxes in each sequence.
%{\color{red}See Appendix~C for further details.}
\end{remark}

In conclusion, we summarize the construction of the desired bijections.
By combining Construction \ref{construction::spp-espp} with Construction \ref{1-2s::qtc-sta} and Construction \ref{1-2s::qtc-sta-even}, it is sufficient to construct bijections
\[
	f \colon \eSPP{n}{m} \to \stairPP{n}{m},
\]
and
\[
	g_{\pi} \colon S(\pi) \to S(f(\pi)), \quad \text{ for all } \pi \in \eSPP{n}{m},
\]
in order to obtain the desired bijections between $\SPP{n}{2m}$ and $\QTCPP{n}{2m}$, and between $\SPP{n}{2m+1}$ and $\QTCPP{n}{2m+1}$.
For example, $\pi$ in Example \ref{ex::3-2-3} corresponds to $f_\text{\ref{1-2s::qtc-sta-even}}^{-1}\left( (t \mid_{S(\pi')} \circ g_{\pi'}^{-1}, f(\pi')) \right)$.

\hspace{.5cm}
$
\begin{tikzcd}[row sep=large]
  % Row 1
  \SPP{n}{2m}
    \arrow[d, "{\text{Construction~\ref{construction::spp-espp}}}"']
   & \ni
  & \pi
    \arrow[d, mapsto]
  \\[2pt]
  % Row 2
  \displaystyle \bigsqcup_{\pi' \in \eSPP{n}{m}} \{0,1\}^{S(\pi')}
    \arrow[d, "\begin{matrix}\rw{\text{applying }f \text{ and } g,} \\ \text{\rw{ which we shall construct }} \\ \text{\rw{ in the rest of this paper }} \end{matrix}"']
   & \ni 
   & (t,\pi'), \rlap{\quad $t \colon S(\pi') \to \{0,1\}$}
    \arrow[d, mapsto]
  \\[15pt]
  % Row 3
  \displaystyle \bigsqcup_{\pi'' \in \stairPP{n}{m}} \{0,1\}^{S(\pi'')}
    \arrow[d, "f_{\ref{1-2s::qtc-sta-even}}^{-1}"']
   & \ni
  & (t \circ g_{\pi'}^{-1}, f(\pi')), \rlap{ \quad $t \circ g_{\pi'}^{-1} \colon S(f(\pi')) \to \{0,1\}$}
    \arrow[d, mapsto]
  \\[2pt]
  % Row 4
  \QTCPP{n}{2m}
   & \ni
  & f_{\ref{1-2s::qtc-sta-even}}^{-1}\!\left( (t \circ g_{\pi'}^{-1}, f(\pi')) \right)
\end{tikzcd}
$

\section{Preliminaries}\label{sec::prel}
Hereafter we shall construct the bijections $f$ and $g_\pi$ described at the end of Section 2.
This construction is based on knowledge concerning signed sets, sijections, lattice paths, and the so-called LGV lemma.
Let us therefore first review the definitions and properties of signed sets, sijections and compatibilities of sijections, which were established by Fischer, Konvalinka \cite{FK1}, and the author \cite{Inoue}.
\subsection{Signed Sets}
A \textit{signed set} is a pair of disjoint finite sets.
For a signed set $S$, we call its first (resp. second) element the \textit{plus part} (resp. \textit{minus part}) of $S$ and denote it by $S^{+}$ (resp. $S^{-}$).
Namely, $S=(S^{+},S^{-})$.
The support of $S$ is $S^{+} \sqcup S^{-}$, and we denote it by $\supp{S}$.
% For a signed set $S$, we call the set $S^{+} \sqcup S^{-}$ the support of $S$ and denote it by $\supp{S}$.
The \textit{size} of a signed set $S$ is defined by $ \#S= \#S^{+} - \#S^{-}$, where we denote the size of a set $X$ by $\#X$.

Next, we define some basic notions relevant to signed sets. (See also \cite{FK1}.)
\begin{itemize}
\item The \textit{opposite} of a signed set $S$ is $-S := (S^-,S^+)$.
\item The \textit{disjoint union} of two signed sets $S$ and $T$ is $S \sqcup T := (S^{+} \sqcup T^{+},S^{-} \sqcup T^{-})$.
\item The \textit{Cartesian product} of two signed sets $S$ and $T$ is \[S \times T := (S^{+} \times T^{+} \sqcup S^{-} \times T^{-},S^{+} \times T^{-} \sqcup S^{-} \times T^{+}).\]
\end{itemize}
We define the disjoint union and the Cartesian product of a finite number of signed sets in the same manner.
In addition, we can define a \textit{disjoint union with signed index} of a family of signed sets indexed by a signed set (to be rigorous, indexed by the support of the signed set) in the following way.
\begin{definition}
	Let $T$ be a signed set and $\{S_t\}_{t \in \supp{T}}$ a family of signed sets.
	The disjoint union with signed index in $T$ of $\{S_t\}_{t \in \supp{T}}$ is
	\[
		\bigsqcup_{t \in T} S_t := \left( \bigsqcup_{t \in T^{+}}(S_t^+) \sqcup \bigsqcup_{t \in T^{-}}(S_t^-), \bigsqcup_{t \in T^{+}}(S_t^-) \sqcup \bigsqcup_{t \in T^{-}}(S_t^+) \right).
	\]
\end{definition}
We denote an element of $\bigsqcup_{t \in T} S_t$ as $(s_t,t)$ for $s_t \in \supp{S_t}$.
%For simplicity, we will sometimes denote an element of $\bigsqcup_{u \in U} \bigsqcup_{t \in T} S_{t,u}$ as $(s_{t,u},t,u)$ instead of $((s_{t,u},t),u)$.

Let us define a \textit{signed interval}, which is the most basic example of a signed set.
For any two integers $a$ and $b$, a signed interval $\sint{a}{b}$ is defined by,
\[
	\sint{a}{b}=\begin{cases}
		([a,b)\cap\mathbb{Z},\emptyset) & (a < b), \\
		(\emptyset,\emptyset) & (a=b), \\
		(\emptyset,[b,a)\cap\mathbb{Z}) & (a>b).
	\end{cases}
\]
%It is noteworthy that we use half-open intervals to describe signed intervals instead of closed intervals as in \cite{FK1,FK2}.
%Because of this, we have $\sint{b}{a}=-\sint{a}{b}$.
%This relation is crucial and many properties of Gelfand-Tsetlin patterns are easier to establish using this notation.
%For more information, see \S \ref{secofGT}.

\subsection{Sijections}
A \textit{sijection} $ \varphi $ from a signed set $S$ to a signed set $T$ is an involution on $\supp{S} \sqcup \supp{T}$ such that $ \varphi(S^{+} \sqcup T^{-}) = S^{-} \sqcup T^{+}$,
namely a sijection $ \varphi $ is a bijection between $S^{+} \sqcup T^{-}$ and $S^{-} \sqcup T^{+}$.
We denote it by $S \sijto T$.
If there is a sijection from $S$ to $T$, then $\#S = \#T$ holds.
This relation is an analogy of the relation between ordinary sets and a bijection.
In particular, if $ S^{-}=T^{-}=\emptyset $, then we can interpret $S$ and $T$ as ordinary sets and a sijection between them as an ordinary bijection.

A sijection $ \varphi $ can be also recognized as a triplet of a sign-preserving bijection from a subset of $S$ to a subset of $T$, a sign-reversing involution on the remaining part of $S$, and a sign-reversing involution on the remaining part of $T$.
Figure \ref{illust_sij} below illustrates this interpretation of a sijection.
In the figure, the upper-left square represents the set $S^{+}$, and similarly for the other three squares.
The symbol $\sijto$ was inspired by Figure \ref{illust_sij}.

\begin{figure}[h]
\centering
\begin{tikzpicture}
\draw[thick] (0,0) -- (4,0) -- (4,4) -- (0,4) -- (0,0);
\draw[thick] (0,2) -- (4,2);
\draw[thick] (2,0) -- (2,4);
\draw (0,0) -- (4,4);
\draw (4,0) -- (0,4);
\node (S) at (1,4.5) {$S$};
\node (T) at (3,4.5) {$T$};
\node (+) at (-0.5,3) {$+$};
\node (-) at (-0.5,1) {$-$};
\draw[<->] (1.3,3.3) -- (2.7,3.3);
\draw[<->] (1.3,0.7) -- (2.7,0.7);
\draw[->] (0.75,2.7) .. controls (1.15,2.3) and (1.15,1.7) .. (0.75,1.3);
\draw[->] (0.65,1.3) .. controls (0.25,1.7) and (0.25,2.3) .. (0.65,2.7);
\draw[->] (3.35,2.7) .. controls (3.75,2.3) and (3.75,1.7) .. (3.35,1.3);
\draw[->] (3.25,1.3) .. controls (2.85,1.7) and (2.85,2.3) .. (3.25,2.7);
\end{tikzpicture}
\caption{An illustration of a sijection.}
\label{illust_sij}
\end{figure}

Since a sijection $\varphi \colon S \sijto T$ is an involution, we can use the notion ``$\varphi(s)$'' for an element $s$ not only of $S$ but also of $T$.
On the other hand, it is sometimes important to distinguish between the domain and the codomain, especially when we consider a composition of sijections.
Thus, we define an inversion $\varphi^{-1} \colon T \sijto S$ of $\varphi$ by
\[
	\varphi^{-1} = \varphi \quad (\textrm{as an involution}),
\]
and distinguish $\varphi^{-1}$ from $\varphi$ as a sijection.

The simplest example of a sijection is the \textit{identity sijection} $\textrm{id}_S$, defined by
\[\begin{tikzcd}[row sep=tiny]
S \arrow[r,phantom, "\sijto"] 	& S \\
s \arrow[r,leftrightarrow]		& s
\end{tikzcd}\]
for any signed set $S$. Let us introduce a few more examples.% which we shall use in later sections.
%\begin{example}\label{ex_sij}
%For any integers $a$, $b$ and $c$, we can construct a sijection $\varphi \colon \sint{a}{b} \sijto \sint{a}{c} \sqcup \sint{c}{b}$ as follows:
%\begin{quote}
%	In case $a<b<c$, for $x \in \sint{a}{b}$, $\varphi(x)=x \in \sint{a}{c}$ and for $x \in \sint{a}{b} \subset \sint{a}{c}$ (i.e. as an element of the codomain), $\varphi(x)=x \in \sint{a}{b}$. The remaining part of $\sint{a}{c}$ can be related to $\sint{c}{b}$.
%	
%	A similar construction can be given in the other cases.
%\end{quote}
%This sijection is simple but plays an important role in Fischer and Konvalinka's constructions \cite{FK1,FK2}.
%\end{example}
\begin{example}\label{ex_sij}
For any integers $a$, $b$ and $c$, there exists a sijection $\varphi \colon \sint{a}{b} \sijto \sint{a}{c} \sqcup \sint{c}{b}$.
In fact, each integer in $[\min \{a,b,c\}, \max \{a,b,c\}) \cup \mathbb{Z}$ appears twice as an element, and when we let one correspond to the other, we obtain the sijection.
\end{example}

\begin{example}
Let $A$, $B$ be ordinary sets and $f \colon A \to B$ a bijection.
In this situation, we can interpret $f$ as a sijection between $S=(A,B)$ and $(\emptyset,\emptyset)$.
\end{example}

\begin{example}\label{sij_oppo}
Let $S$, $T$ be signed sets and $\varphi \colon S \sijto T$ a sijection.
Since $\varphi$ is also a bijection between $S^+ \sqcup T^-$ and $S^- \sqcup T^+$, by applying the previous example to this bijection,
we have a sijection $S \sqcup -T \sijto (\emptyset,\emptyset)$.
%Then, we have a sijection $S \sqcup -T \sijto (\emptyset,\emptyset)$.
%In fact, $\varphi$ is also a bijection between $S^+ \sqcup T^-$ and $S^- \sqcup T^+$.
%Therefore, this is also a sijection between $(S^+ \sqcup T^-, S^- \sqcup T^+)=S \sqcup -T$ and $(\emptyset,\emptyset)$ by the previous example.
Conversely, a sijection between $S \sqcup -T$ can be interpreted as a sijection between $S$ and $T$.
%In particular, we have $S \sqcup -S \sijto (\emptyset,\emptyset)$ derived from the identity sijection on $S$.
\end{example}
In particular, we have $$S \sqcup -S \sijto (\emptyset,\emptyset)$$ derived from the identity sijection on $S$ for any signed set $S$.
\subsubsection{Composition of sijections}
A composition of sijections is defined as follows \cite{FK1}.
\begin{lemdef}[{\cite[Proposition 2 (1)]{FK1}}]\label{lem:comp}
Let $S$, $T$, $U$ be signed sets and $ \varphi \colon S \sijto T $, $ \psi \colon T \sijto U $ sijections.
For $ s \in \supp{S} $, we define $ \psi \circ \varphi(s) $ as the last well-defined element in the sequence
\[
	s,\, \varphi(s),\, \psi(\varphi(s)),\, \varphi(\psi(\varphi(s))),\, \ldots.
\]
Similarly, for $ u \in \supp{U} $, we define $ \psi \circ \varphi(u) $ as the last well-defined element in the sequence
\[
	u,\, \psi(u),\, \varphi(\psi(u)),\, \psi(\varphi(\psi(u))),\, \ldots.
\]
Then, $ \psi \circ \varphi $ is a sijection from $S$ to $U$, and we call it the composition of $ \varphi $ and $ \psi $.
\end{lemdef}
First we should clarify what we mean by \textit{``the last well-defined element in the sequence \\
$ s,\, \varphi(s),\, \psi(\varphi(s)),\, \varphi(\psi(\varphi(s))),\, \ldots $''}.
For example if $\varphi(s)$ belongs to $\supp{S}$, then $\psi(\varphi(s))$ is not defined because the domain of $\psi$ is $\supp{T} \sqcup \supp{U}$.
Therefore, $ (\psi \circ \varphi)(s) $ is $\varphi(s)$ in this case.
For the original proof of the lemma, see \cite{Doyle}.
Here, we give an alternative proof of this fact using the language of graphs.
A graph in the proof might have multiple edges, so we fix notations relevant to multisets.
We use double braces to describe a multiset, and use the symbol ``$+$'' to describe a sum of multisets.
\begin{proof}[Proof of Lemma \& Definition \ref{lem:comp}]
Consider a graph with vertices $V=\supp{S} \sqcup \supp{T} \sqcup \supp{U}$ and edges $E=R+B$, where
\begin{align*}
	R&=\{\!\{\, \{v,\varphi(v)\} \mid v \in S^{+} \sqcup T^{-}\,\}\!\}, \\
	B&=\{\!\{\, \{v,\psi(v)\} \mid v \in T^{+} \sqcup U^{-}\,\}\!\}.
\end{align*}
Note that the graph is an undirected bipartite finite graph, where one part is $S^{+} \sqcup T^{-} \sqcup U^{+}$ and the other part is $S^{-} \sqcup T^{+} \sqcup U^{-}$.
We consider painting edges in R with red and those in B with blue.
Then, by the definitions, each vertex belongs to at most one red edge and at most one blue edge.
Because of the degrees of the vertices and the finiteness of the graph, the graph consists of finitely many line graphs and cycles.
Moreover, all the endpoints of these line graphs belong to $\supp{S} \sqcup \supp{U}$ and for each line graph or cycle the color of the edges alternate.
Please refer to Figure \ref{fig::comp}  below for an example:
dashed edges represent those in $B$ and the others those in $R$.
In addition, the shape of a node corresponds to which part it belongs to as a vertex of the bipartite graph.
\begin{figure}[h]
\centering
\begin{tikzpicture}[main/.style = {draw, circle}]
\node (S) at (0,7) {$S$};
\node (phi) at (1.5,7.1) {$\overset{\varphi}{\sijto}$};
\node (T) at (3,7) {$T$};
\node (phi) at (4.5,7.1) {$\overset{\psi}{\sijto}$};
\node (U) at (6,7) {$U$};
\node (+) at (-2,4.5) {$+$};
\node (-) at (-2,1) {$-$};
\node[main] (1) at (0,5.5) {$+$};
\node[main] (2) at (0,3.5) {$+$};
\node[draw,rectangle] (3) at (0,1) {$-$};
\node[draw,rectangle] (4) at (3,6) {$+$};
\node[draw,rectangle] (5) at (3,5) {$+$};
\node[draw,rectangle] (6) at (3,4) {$+$};
\node[draw,rectangle] (7) at (3,3) {$+$};
\node[main] (8) at (3,2) {$-$};
\node[main] (9) at (3,1) {$-$};
\node[main] (10) at (3,0) {$-$};
\node[main] (11) at (6,5.5) {$+$};
\node[main] (12) at (6,3.5) {$+$};
\node[draw,rectangle] (13) at (6,1) {$-$};
\draw[-,thick,red] (1) -- (4);
\draw[-,thick,red] (2) -- (5);
\draw[-,thick,red] (3) -- (10);
\draw[-,thick,red] (6) to [out=180,in=180] (9);
\draw[-,thick,red] (7) to [out=180,in=180] (8);
\draw[-,dashed,thick,blue] (4) -- (11);
\draw[-,dashed,thick,blue] (5) to [out=0,in=0] (9);
\draw[-,dashed,thick,blue] (6) to [out=0,in=0] (10);
\draw[-,dashed,thick,blue] (7) to [out=0,in=0] (8);
\draw[-,dashed,thick,blue] (12) to [out=180,in=180] (13);
\draw (-2.5,-0.5) to (7.5,-0.5);
\draw (-2.5,2.5) to (7.5,2.5);
\draw (-2.5,6.5) to (7.5,6.5);
\end{tikzpicture}
\caption{A composition of sijections.}
\label{fig::comp}
\end{figure}

Consider $s \in \supp{S}$.
The degree of $s$ is $1$, so it is an endpoint of a line graph.
Therefore, $ \psi \circ \varphi(s) $ is the other endpoint of the line graph.
If $ \psi \circ \varphi(s) \in \supp{S} $, the first edge and the last edge in the line graph are both red, so its length is odd.
Thus, $s$ and $ \psi \circ \varphi(s) $ belong to different parts of $G$ as a bipartite graph.
Therefore, the sign of $\psi \circ \varphi(s)$ is different from that of $s$.
In the other case, namely when $ \psi \circ \varphi(s) \in \supp{U} $, the length of the line graph is even and $ \psi \circ \varphi(s) $ belongs to the same part as $s$,
so the sign of $\psi \circ \varphi(s)$ is equal to that of $s$.

For $ u \in \supp{U} $, we can give a similar proof. Therefore, $\psi \circ \varphi$ is a well-defined sijection from $S$ to $U$.
\end{proof}

Relating to this proof we define a graph $G=(V,E)$ of a sijection $\varphi \colon S \sijto T$ as follows:
\begin{align*}
	V &:= \supp{S} \sqcup \supp{T}, \\
	E &:= \{\!\{\, \{v,\varphi(v)\} \mid v \in S^{+} \sqcup T^{-}\,\}\!\}.
\end{align*}
This notion is useful to prove some lemmas and to understand properties of sijections.
For example, with this language, the sijection $\varphi \colon \sint{1}{3} \sqcup \sint{3}{2} \sijto \sint{1}{2}$ in Example \ref{ex_sij} is described as:
\[
\begin{tikzpicture}[main/.style = {draw, circle}]
\node (S) at (0,3) {$\sint{1}{3} \sqcup \sint{3}{2}$};
\node (T) at (3,3) {$\sint{1}{2}$};
\node (phi) at (-1.9,3) {$\varphi \colon$};
\node (sij) at (1.8,3) {$\sijto$};
\node (+) at (-2,1.5) {$+$};
\node (-) at (-2,0) {$-$};
\draw (-2.5,-0.5) to (3.5,-0.5);
\draw (-2.5,0.5) to (3.5,0.5);
\draw (-2.5,2.5) to (3.5,2.5);
\node[main] (1) at (0,2) {$1$};
\node[main] (2) at (0,1) {$2$};
\node[main] (3) at (0,0) {$2$};
\node[main] (4) at (3,2) {$1$};
\draw[-] (1) -- (4);
\draw[-] (2) to [out=0,in=0] (3);
\end{tikzpicture}
\]

%\begin{remark}
%There are two reasons why we impose finiteness on signed sets.
%One of them is because, without the finiteness it could simply be meaningless since $(\mathbb{N},\mathbb{N})$ is in sijection to arbitrary signed sets (and any infinite ``signed set'').
%The other reason is because it is needed for the well-definedness of compositions of sijections.
%For example, let $S=(\{1\},\emptyset)$, $T=(2\mathbb{N}, 2\mathbb{N}+1)$, $U=(\emptyset,\emptyset)$, $\varphi(2x-1)=2x$ and $\psi(2x)=2x+1$.
%If we define a graph in the same way as in finite cases, the connected component that $1 \in S^+$ belongs to is still a line graph but not finite since we have
%\[
%	\varphi(1) = 2,\quad \psi(\varphi(1))=3,\quad \varphi(\psi(\varphi(1)))=4,\quad \ldots.
%\]
%Thus, the proof and the definition are broken in infinite cases.
%\end{remark}

Last, we prove that compositions of sijections have the associative property.
\begin{proposition}[\cite{Doyle}, Corollary 3]
Let $S$, $T$, $U$ and $V$ be signed sets and $ \varphi \colon S \sijto T $, $ \psi \colon T \sijto U $, and $ \xi \colon U \sijto V $ sijections. Then,
\[
	\xi \circ ( \psi \circ \varphi ) = ( \xi \circ \psi ) \circ \varphi .
\]
\end{proposition}
\begin{proof}
Consider a graph with vertices $ \supp{S} \,\sqcup\, \supp{T} \,\sqcup\, \supp{U} \,\sqcup\, \supp{V} $ and edges
$ \{\!\{\, \{v,\varphi(v)\} \mid v \in S^{+} \sqcup T^{-}\,\}\!\} + \{\!\{\, \{v,\psi(v)\} \mid v \in T^{+} \sqcup U^{-}\,\}\!\} + \{\!\{\, \{v,\xi(v)\} \mid v \in U^{+} \sqcup V^{-}\,\}\!\} $.
This graph consists of line graphs and cycles.
Moreover, for any endpoint $v$ of a line graph the other endpoint of the line graph is $(\xi \circ ( \psi \circ \varphi ))(v) = (( \xi \circ \psi ) \circ \varphi)(v)$.
\end{proof}

\subsubsection{Cartesian products of sijections} \label{sec::Cart_sij}
The definition of the Cartesian product of sijections is a bit more complicated.
For sijections $\varphi_i \colon S_i \sijto T_i$ ($i=1,2$),
we can construct a sijection $\varphi_1 \times \varphi_2$ between $S_1 \times S_2$ and $T_1 \times T_2$ as follows:
\begin{quote}
	For $s=(s_1,s_2) \in \supp{S_1 \times S_2} = \supp{S_1} \times \supp{S_2}$,
	\[
		(\varphi_1 \times \varphi_2)(s)
		= \begin{cases}
		(\varphi_1(s_1),s_2) &\qquad \varphi_1(s_1) \in \supp{S_1}, \\
		(s_1,\varphi_2(s_2)) &\qquad \varphi_1(s_1) \in \supp{T_1} \text{ and } \varphi_2(s_2) \in \supp{S_2}, \\
		(\varphi_1(s_1),\varphi_2(s_2)) &\qquad \varphi_1(s_1) \in \supp{T_1} \text{ and } \varphi_2(s_2) \in \supp{T_2},
		\end{cases}
	\]
	and for $t=(t_1,t_2) \in \supp{T_1 \times T_2} = \supp{T_1} \times \supp{T_2}$,
	\[
		(\varphi_1 \times \varphi_2)(t)
		= \begin{cases}
		(\varphi_1(t_1),t_2) &\qquad \varphi_1(t_1) \in \supp{T_1}, \\
		(t_1,\varphi_2(t_2)) &\qquad \varphi_1(t_1) \in \supp{S_1} \text{ and } \varphi_2(t_2) \in \supp{T_2}, \\
		(\varphi_1(t_1),\varphi_2(t_2)) &\qquad \varphi_1(t_1) \in \supp{S_1} \text{ and } \varphi_2(t_2) \in \supp{S_2}.
		\end{cases}
	\]
\end{quote}
It is easy to check that this sijection is indeed an involution on $\supp{S_1 \times S_2} \sqcup \supp{T_1 \times T_2}$ and that it meets the sign conditions:
\begin{itemize}
\item if $(\varphi_1 \times \varphi_2)(s) \in \supp{S_1 \times S_2}$, the sign of $(\varphi_1 \times \varphi_2)(s)$ is different from the sign of $s$,
\item if $(\varphi_1 \times \varphi_2)(s) \in \supp{T_1 \times T_2}$, the sign of $(\varphi_1 \times \varphi_2)(s)$ is equal to the sign of $s$,
\item and similar results are true for $t \in \supp{T_1 \times T_2}$.
\end{itemize}
Therefore, it is indeed a sijection.
Additionally, this definition has the associative property.
\begin{proposition}\label{prop_asso_Cart}
Let $ \varphi_i \colon S_i \sijto T_i $ be a sijection for $i=1,2,3$. Then, we have
\[
	(\varphi_1 \times \varphi_2) \times \varphi_3 = \varphi_1 \times (\varphi_2 \times \varphi_3).
\]
\end{proposition}
\begin{proof}
Regardless of how one puts the parentheses, if all $\varphi_i$ send $s_i$ to $T_i$, then $s=(s_1,s_2,s_3)$ corresponds to $(\varphi_1(s_1),\varphi_2(s_2),\varphi_3(s_3))$.
Otherwise, only the leftmost element $s_i$ of $s$ such that $\varphi_i(s_i) \in S_i$ is replaced with $\varphi_i(s_i)$.
\end{proof}

Thus, we can define the Cartesian product of a finite number of sijections as follows. (See also Proposition 2(2) of \cite{FK1}.)
\begin{definition}
Let $S_1,S_2,\ldots,S_k$, $T_1,T_2,\ldots,T_k$ be signed sets and let $ \varphi_i \colon S_i \sijto T_i $ be a sijection for $i=1,2,\ldots,k$.
We define
\[
	\varphi_1 \times \varphi_2 \times \cdots \times \varphi_k = (\cdots((\varphi_1 \times \varphi_2) \times \varphi_3) \times \cdots ) \times \varphi_k.
\]
\end{definition}
Let $s=(s_1,s_2,\ldots,s_k) \in \supp{S_1 \times S_2 \times \cdots \times S_k} = \supp{S_1} \times \supp{S_2} \times \cdots \times \supp{S_k}$.
Then, according to the proof of Proposition \ref{prop_asso_Cart},
when $\varphi_i(s_i) \in \supp{T_i}$ for all $i = 1,2,\ldots,k$, we have
\[
	(\varphi_1 \times \varphi_2 \times \cdots \times \varphi_k)(s) = (\varphi_1(s_1),\varphi_2(s_2),\ldots,\varphi_k(s_k)).
\]
Otherwise, we have
\[
	(\varphi_1 \times \varphi_2 \times \cdots \times \varphi_k)(s) = (s_1,s_2,\ldots,s_{j-1},\varphi_j(s_j),s_{j+1},\ldots,s_k),
\]
where $j$ is the minimum index such that $\varphi_j(s_j) \in \supp{S_j}$.
For an element of $\supp{T_1 \times T_2 \times \cdots \times T_k}$, a similar expression can be given.

\begin{Remark} \label{rem-2.2.2}
We remark on a disadvantageous feature of this Cartesian product of sijections.
Let $\varphi_i \colon S_i \sijto T_i$ and $\psi_i \colon T_i \sijto U_i$ be sijections for $i=1,2$.
In general, it does not hold that
\begin{equation} \label{cartesian-counterex}
	(\psi_1 \times \psi_2) \circ (\varphi_1 \times \varphi_2) = (\psi_1 \circ \varphi_1) \times (\psi_2 \circ \varphi_2).
\end{equation}
For example, let $S=(\{A\},\emptyset)$, $T=(\{A,B\},\{B^\dag\})$ and define a sijection $\varphi$ as
\[
\begin{tikzpicture}[main/.style = {draw, circle}]
\node (S) at (0,3) {$S$};
\node (T) at (3,3) {$T$};
\node (phi) at (-1.9,3) {$\varphi \colon$};
\node (sij) at (1.5,3) {$\sijto$};
\node (+) at (-2,1.5) {$+$};
\node (-) at (-2,0) {$-$};
\draw (-2.5,-0.5) to (3.5,-0.5);
\draw (-2.5,0.5) to (3.5,0.5);
\draw (-2.5,2.5) to (3.5,2.5);
\node[main] (1) at (0,2) {$A$};
\node[main] (2) at (3,1) {$B$};
\node[scale=0.84,main] (3) at (3,0) {$B^\dag$};
\node[main] (4) at (3,2) {$A$};
\draw[-] (1) -- (4);
\draw[-] (2) to [out=180,in=180] (3);
\end{tikzpicture}.
\]

If the relation (\ref{cartesian-counterex}) is always true, it must hold that
\[
	(\text{id}_T \times \varphi) \circ (\varphi \times \text{id}_S) = \varphi \times \varphi =  (\varphi \times \text{id}_T) \circ (\text{id}_S \times \varphi).
\]
However, we have by a simple calculation that
\[
	((\text{id}_T \times \varphi) \circ (\varphi \times \text{id}_S))((B,B)) = (B,B^\dag) \ne (B^\dag,B) = ((\varphi \times \text{id}_T) \circ (\text{id}_S \times \varphi))((B,B)),
\]
so we have a contradiction. Thus, the relation (\ref{cartesian-counterex}) is not always true.

This feature is inevitable, namely it is not due to a defect in the definition.
For instance, in the above situation, we cannot determine which of $(B,B^\dag)$ and $(B^\dag,B)$ should be $(\varphi \times \varphi)((B,B))$ a priori because of symmetry.
\end{Remark}

\subsubsection{Disjoint unions of sijections}
For sijections $\varphi_i \colon S_i \sijto T_i$ ($i=1,2$), the disjoint union $\varphi_1 \sqcup \varphi_2 \colon S_1 \sqcup S_2 \sijto T_1 \sqcup T_2$ is defined by
\[ (\varphi_1 \sqcup \varphi_2)(s) = \begin{cases} \varphi_1(s) & s \in \supp{S_1} \sqcup \supp{T_1}, \\ \varphi_2(s) & s \in \supp{S_2} \sqcup \supp{T_2}. \end{cases}\]
We define the disjoint union of a finite number of sijections in the same manner.
Considering the graph of sijections, the disjoint union is just a juxtaposition, so we have $(\varphi_1 \sqcup \varphi_2) \sqcup \varphi_3 = \varphi_1 \sqcup (\varphi_2 \sqcup \varphi_3)$.

Next, we define the disjoint union with signed index of a family of sijections.
\begin{definition}[Proposition 2(3) in \cite{FK1}] \label{def_sij_DIS}
Let $T$, $\widetilde{T}$ be signed sets and $\psi \colon T \sijto \widetilde{T}$ a sijection.
Assume that we have signed sets $\{ S_t \}_{t \in \supp{T} \sqcup \supp{\widetilde{T}}}$ and sijections
$ \{ \varphi_t \colon S_t \sijto S_{\psi(t)} \}_{ t \in T^+ \sqcup \widetilde{T}^-}  $.
Let $\varphi_t = \varphi_{(\psi(t))}^{-1}$ for $t \in T^- \sqcup \widetilde{T}^+$.
Then, we have a sijection $\varphi$ from $\bigsqcup_{t \in T} S_t$ to $\bigsqcup_{t \in \widetilde{T}} S_t$ defined by
\[
	\varphi\left((s_t,t)\right)=\begin{cases} (\varphi_t(s_t),t) & \text{if } \varphi_t(s_t) \in S_t, \\ (\varphi_t(s_t),\psi(t)) & \text{if } \varphi_t(s_t) \in S_{\psi(t)}. \end{cases}
\]
We denote it by $$\bigsqcup_{t \,;\, \psi \colon T \subsijto \widetilde{T}} \varphi_t \qquad \text{ or } \qquad \bigsqcup_{t \,;\, \psi} \varphi_t. $$
\end{definition}

Considering the graphs of sijections, the disjoint union of this type is also a juxtaposition. Therefore, it is indeed a sijection.

When $T=\widetilde{T}$ and $\psi=\mathrm{id}_T$ hold, the situation becomes simple.
Let $\{ \varphi_t \colon S_t \sijto \widetilde{S}_t \}_{t \in \supp{T}}$ be sijections, then we have
\[
	\bigsqcup_{t \in T} \varphi_t \left( := \bigsqcup_{t \,;\, \id_T} \varphi_t \right)\colon \bigsqcup_{t \in T} S_t \sijto \bigsqcup_{t \in T} \widetilde{S}_t.
\]

%Then, when we have sijections $\{ \varphi_t \colon S_t \sijto \widetilde{S}_t \}_{t \in \supp{T}}$, by applying the definition we have
%\[
%	\bigsqcup_{t \in T} \varphi_t \colon \bigsqcup_{t \in T} S_t \sijto \bigsqcup_{t \in T} \widetilde{S}_t.
%\]

%The constructions of sijections in the following Examples (\ref{sij_dist} - \ref{du_dist}) become trivial by considering their graphs.
%In particular, they can be decomposed to a bijection between plus parts and one between minus parts.
%Therefore, we shall use ``$=$'' to describe these sijections in later constructions of much more complicated sijections.
We shall construct another four canonical sijections.
\begin{construction}\label{sij_dist}
Let $S$, $T_1$ and $T_2$ be signed sets. Then we have a canonical sijection
\[
	S \times (T_1 \sqcup T_2) \sijto (S \times T_1) \sqcup (S \times T_2),
\]
which is derived from the identities
\begin{align*}
	\left( S \times (T_1 \sqcup T_2) \right)^+ &= \left( (S \times T_1) \sqcup (S \times T_2) \right)^+ \\
	&= (S^+ \times T_1^+) \sqcup (S^+ \times T_2^+) \sqcup (S^- \times T_1^-) \sqcup (S^- \times T_2^-),\\
	\left( S \times (T_1 \sqcup T_2) \right)^- &= \left( (S \times T_1) \sqcup (S \times T_2) \right)^- \\
	&= (S^+ \times T_1^-) \sqcup (S^+ \times T_2^-) \sqcup (S^- \times T_1^+) \sqcup (S^- \times T_2^+).
\end{align*}
\end{construction}
\begin{construction}\label{sij_dist_family}
Let $S$ and $T$ be signed sets and $\{S_t\}_{t \in \supp{T}}$ a family of signed sets. Then, we have a canonical sijection
\[
	S \times \bigsqcup_{t \in T} S_t \sijto \bigsqcup_{t \in T} \left( S \times S_t \right)
\]
which is derived from the identities.
\end{construction}
\begin{construction}\label{du_dist}
Let $T$, $U$ be signed sets and $\{S_v\}_{v \in \supp{T} \sqcup \supp{U}}$ a family of signed sets.
Then, we have a canonical sijection
\[
	\bigsqcup_{t \in T} S_t \sqcup \bigsqcup_{u \in U} S_u \sijto \bigsqcup_{v \in T \sqcup U} S_v,
\]
which is derived from the identities.
\end{construction}
\begin{construction}\label{du_ord}
Let $S_1$ and $S_2$ be signed sets and $\{ T_{s_1,s_2} \}_{s_1 \in \supp{S_1}, s_2 \in \supp{S_2}}$ a family of signed sets.
Then, we have a canonical sijection
\[
	\bigsqcup_{s_1 \in S_1} \bigsqcup_{s_2 \in S_2} T_{s_1,s_2} \sijto \bigsqcup_{s_2 \in S_2} \bigsqcup_{s_1 \in S_1} T_{s_1,s_2},
\]
which is derived from bijections of the form
\[
	\bigsqcup_{a \in A} \bigsqcup_{b \in B} C_{a,b} \to \bigsqcup_{b \in B} \bigsqcup_{a \in A} C_{a,b} \colon ((c,b),a) \mapsto ((c,a),b),
\]
where $A$, $B$ and $C_{a,b}$ (for each $a \in A$ and $b \in B$) are ordinary sets.
\end{construction}
\begin{construction}\label{du_nest}
Let $U$ be a signed set and $\{T_u\}_{u\in\supp{U}}$, $\{S_{t,u}\}_{u \in\supp{U},t \in\supp{T_u}}$ families of signed sets.
Then, we have a canonical sijection
\[
	\bigsqcup_{u \in U} \bigsqcup_{t \in T_u} S_{t,u} \sijto \bigsqcup_{t \in \bigsqcup_{u \in U} T_u} S_{t,u},
\]
which is derived from bijections of the form
\[
	\bigsqcup_{a \in A} \bigsqcup_{b \in B_a} C_{a,b} \to \bigsqcup_{b \in \bigsqcup_{a \in A} B_a} C_{a,b} \colon ((c,b),a) \mapsto (c,(b,a)).
\]
\end{construction}

These constructions can be decomposed to bijections, which are almost identities.
In these cases, we shall identify the domain with the codomain and  use ``$=$'' to describe these sijections.

%As a last remark in this section, in the following we fix operator precedence for signed sets as:
Lastly, we explain the operator precedence for signed sets.
The operators should be evaluated in the order:
\begin{enumerate}
\item taking the opposite ($-$),
\item the Cartesian products ($\times$),
\item disjoint unions ($\sqcup$).
\end{enumerate}
For example, we can write $S \times T_1 \sqcup S \times T_2$ instead of $(S \times T_1) \sqcup (S \times T_2)$,
but not $S \times T_1 \sqcup T_2$ instead of $S \times (T_1 \sqcup T_2)$.

\subsection{Compatibility}
\begin{definition}
A \textit{statistic} of a signed set is a function on (at least) its support.
Let $S$, $T$ be signed sets and $\eta$ a statistic of $S \sqcup T$. 
A sijection $\phi \colon S \sijto T$ is \textit{compatible with $\eta$} if
\[
	\forall s \in \supp{S \sqcup T},\,\eta(\phi(s))=\eta(s).
\]
\end{definition}
The co-domain of statistics does not matter, so we will not pay attention to it.
%For simplicity, we will sometimes write $ s \insupp  S $ instead of $s \in \supp{S}$.
%Compatibility is a generalization of the notion of \textit{normality} which is used in \cite{FK1}.
%For example, let us consider the sijection $\varphi$ in Example \ref{ex_sij}, which is normal in the language of \cite{FK1}.
%We can recognize each element of the support of  $\sint{a}{b} \sqcup (\sint{a}{c} \sqcup \sint{c}{b})$ as an integer in the canonical way,
%so we can define a canonical $\mathbb{Z}$-valued statistic of $\sint{a}{b} \sqcup (\sint{a}{c} \sqcup \sint{c}{b})$.
%We can define a canonical $\mathbb{Z}$-valued statistic of $\sint{a}{b} \sqcup (\sint{a}{c} \sqcup \sint{c}{b})$,
%because we can recognize each element of the support of  $\sint{a}{b} \sqcup (\sint{a}{c} \sqcup \sint{c}{b})$ as an integer in the canonical way.
%The sijection $\varphi$ is compatible with this statistic.
%We call statistics of this type \textit{normal} statistics.
%The precise definition is given in Definition \ref{defofnormal}.
It is important to note that compatibility is preserved under compositions. % and so is normality.
\begin{lemma}\label{lem:comp_comp}
Let $\phi \colon S \sijto T$, $\psi \colon T \sijto U$ be sijections and $\eta$ a statistic of $S \sqcup T \sqcup U$.
If $\phi$ and $\psi$ are compatible with $\eta$, then $\psi \circ \phi$ is compatible with $\eta$.
\end{lemma}
\begin{proof}
For any $s \in \supp{S}$, $\eta(s)=\eta(\phi(s))=\eta(\psi(\phi(s)))=\cdots$ while they are well-defined.
Therefore, $\eta(s) = \eta(\psi \circ \phi(s))$.
The same is true for an element of $\supp{U}$.
\end{proof}

Similarly, compatibility is preserved under disjoint unions and Cartesian products.
\begin{lemma}\label{lemofsqcup}
Let $ \varphi_i \colon S_i \sijto T_i $ be a sijection compatible with a statistic $\eta_i$ for $i=1,2$.
When we define a statistic $\eta=\eta_1\sqcup\eta_2$ by
\begin{align*}
	\eta(u)=(\eta_1\sqcup\eta_2)(u)=\begin{cases}
		\eta_1(u) & (u \in \supp{S_1 \sqcup T_1}) \\
		\eta_2(u) & (u \in \supp{S_2 \sqcup T_2}), \end{cases}
\end{align*}
$\varphi_1 \sqcup \varphi_2$ is compatible with $\eta_1\sqcup\eta_2$.
\end{lemma}
\begin{lemma}\label{lemoftimes}
Let $ \varphi_i \colon S_i \sijto T_i $ be a sijection compatible with a statistic $\eta_i$ for $i=1,2$.
When we define a statistic $\eta=\eta_1\times\eta_2$ by
\begin{align*}
	\eta((u_1,u_2))=(\eta_1\times\eta_2)(u)=(\eta_1(u_1),\eta_2(u_2)),
\end{align*}
$\varphi_1 \times \varphi_2$ is compatible with $\eta_1\times \eta_2$.
\end{lemma}
By these definitions, a disjoint union and the Cartesian product of more than two sijections can be treated in the same manner.
Furthermore, in the situation of Definition \ref{def_sij_DIS}, if all the sijections $\{\varphi_t\}$ are compatible with some statistic $\eta$, %if for all $t \in T^+ \sqcup \widetilde{T}^-$, the sijection $\varphi_t$ is %$\{\varphi_t\}_{t \in T^+ \sqcup \widetilde{T}^-}$
% compatible with some statistic $\eta$,
then $\bigsqcup_{t \,;\, \psi} \varphi_t$ is also compatible with $\eta$.
Proofs can be given either by considering the graphs of sijections or by using the definitions.

\subsection{A Refinement of a Compatibility of Sijections}
We must slightly generalize the notion of compatibility to fit the context of the present research.
Let $A$ and $B$ be signed sets and $S_A$ and $S_B$ set-valued statistics on $A$ and $B$, respectively.
Note that $\# S_A$ and $\# S_B$ are ordinal statistics on $A$ and $B$, respectively.
Let $\phi$ be a compatible sijection $\left(A, \# S_A\right) \sijto \left(B, \# S_B\right)$.
We call a family $\{ g_a\}_{a \in \supp{A} \sqcup \supp{B}}$ of bijections a refinement of the compatible sijection $\phi$ if both of the following conditions are met for all $a \in \supp{A} \sqcup \supp{B}$:
\begin{itemize}
\item $g_a \colon S(a) \to S(\phi(a))$, where $S$ is a set-valued function such that $S(a) = \begin{cases}S_A(a) & (a \in \supp{A}) \\ S_B(a) & (b \in \supp{B}) \end{cases}$.
\item $g_{\phi(a)}^{-1} = g_a$.
\end{itemize}
In many cases, a refinement can be recovered from an explicit construction of the compatible sijection itself.
Note that the composition of sijections does not diminish the possibility of recovering a refinement, as compatibility is preserved under compositions.
In such cases, we simply say that the sijection is compatible with $S = S_A \sqcup S_B$ and denote it by $\left(A, S_A\right) \sijto \left(B, S_B\right)$.

In the context of the present research, we will consider \( A = \eSPP{n}{M} \) and \( B = \stairPP{n}{m} \),
and we will need to construct a compatible sijection $(A, \#S) \sijto (B, \#S)$,
where \( S \) is defined in Section~\ref{Sec::1-2s}, together with a refinement of this compatible sijection.
In other words, it will be sufficient to construct a compatible sijection $(A, S) \sijto (B, S)$.

\subsection{the LGV lemma}
Non-intersecting lattice paths and the LGV lemma are well-known methods to count plane partitions \cite{GV, Lind}.
We review the statement of the LGV lemma and illustrate this counting method on the example of $\stairPP{n}{m}$ which will appear in the next subsection.

\begin{theorem}[Lindstr\"{o}m--Gessel--Viennot lemma]
Let $\Gamma$ be a finite\footnote{In general, $\Gamma$ can be locally finite rather than finite, considering formal variables.} directed acyclic graph and $a$ and $b$ vertices of $\Gamma$.
We define $\mathcal{P}_{\Gamma}(a \to b)$ as the set of paths from $a$ to $b$.
Let $\mathbf{a}=(a_1,a_2,\ldots,a_n)$ and $\mathbf{b}=(b_1,b_2,\ldots,b_n)$ sequences of vertices of $\Gamma$.
We define the signed set of paths from $\mathbf{a}$ to $\mathbf{b}$, $\mathcal{P}_\Gamma(\mathbf{a},\mathbf{b})$, by
\[
	\mathcal{P}_\Gamma(\mathbf{a},\mathbf{b}) = \bigsqcup_{\sigma \in S_n} \sgn{\sigma} \prod_{i=1}^n \mathcal{P}_{\Gamma}(a_i \to b_{\sigma_i}).
\]
In addition, we denote the set of non-intersecting paths from $\mathbf{a}$ to $\mathbf{b}$ by $\mathcal{P}_\Gamma^{\text{NI}}(\mathbf{a},\mathbf{b})$.
Here, non-intersecting means that they do not have a common vertex.

Then, the number of non-intersecting paths from $\mathbf{a}$ to $\mathbf{b}$ is given by the following determinant:
\[
	\begin{pmatrix}
	\# \mathcal{P}_{\Gamma}(a_1 \to b_1) & \# \mathcal{P}_{\Gamma}(a_1 \to b_2) & \cdots & \# \mathcal{P}_{\Gamma}(a_1 \to b_n) \\
	\# \mathcal{P}_{\Gamma}(a_2 \to b_1) & \# \mathcal{P}_{\Gamma}(a_2 \to b_2) & \cdots & \# \mathcal{P}_{\Gamma}(a_2 \to b_n) \\
	\vdots & \vdots & \ddots & \vdots \\
	\# \mathcal{P}_{\Gamma}(a_n \to b_1) & \# \mathcal{P}_{\Gamma}(a_n \to b_2) & \cdots & \# \mathcal{P}_{\Gamma}(a_n \to b_n)
	\end{pmatrix}.
\]
In particular, we have
\[
	\# \mathcal{P}_\Gamma^{\text{NI}}(\mathbf{a},\mathbf{b}) = \# \mathcal{P}_\Gamma(\mathbf{a},\mathbf{b}).
\]
\end{theorem}

A bijective proof of the LGV lemma can be obtained by constructing a sign-reversing involution on
$\mathcal{P}_{\Gamma}(\mathbf{a},\mathbf{b}) \setminus \mathcal{P}_{\Gamma}^{\mathrm{NI}}(\mathbf{a},\mathbf{b})$ as follows:

\begin{quote}
Fix a topological order of the directed acyclic graph $\Gamma$.
Let
$P=(P_1,P_2,\ldots,P_n)\in 
\mathcal{P}_{\Gamma}(\mathbf{a},\mathbf{b}) \setminus 
\mathcal{P}_{\Gamma}^{\mathrm{NI}}(\mathbf{a},\mathbf{b})$,
and let $\sigma$ be the unique permutation satisfying
$P_i\colon a_i\to b_{\sigma_i}$ for all $i$.

Since $P$ is not non-intersecting, there exists a vertex $v\in V(\Gamma)$ 
that is visited by two distinct paths in the tuple.
Choose $v$ to be the \emph{earliest} such vertex in the fixed topological order, 
and among the paths passing through $v$, let $P_{i_0}$ and $P_{i_1}$ 
(with $i_0<i_1$) be those corresponding to the two smallest indices.

Decompose these two paths at the vertex $v$ as
\[
P_{i_0}=A_0 \cdot B_0,\qquad
P_{i_1}=A_1 \cdot B_1,
\]
where $A_j$ is the segment from $a_{i_j}$ to $v$ 
and $B_j$ is the segment from $v$ to $b_{\sigma_{i_j}}$ 
for $j\in\{0,1\}$.
Then, define a new $n$-tuple of paths 
$Q=(Q_1,Q_2,\ldots,Q_n)$ by swapping the tails $B_0$ and $B_1$,
\[
Q_{i_0}=A_0\cdot B_1,\qquad
Q_{i_1}=A_1\cdot B_0,
\]
and let $Q_i=P_i$ for all $i\neq i_0,i_1$.

The permutation associated with $Q$ is thus obtained from $\sigma$
by composing with the transposition $(i_0\ i_1)$.
Hence the corresponding sign is reversed, 
and the map $P\mapsto Q$ is an involution.
\end{quote}

Obviously, this involution and the identity map on $\mathcal{P}_{\Gamma}^{\mathrm{NI}}(\mathbf{a},\mathbf{b})$ form a sijection
$\mathcal{P}_{\Gamma}(\mathbf{a},\mathbf{b}) \sijto \mathcal{P}_{\Gamma}^{\mathrm{NI}}(\mathbf{a},\mathbf{b})$.

\subsection{Lattice paths pertaining to $\stairPP{n}{m}$}
In the case of $\stairPP{n}{m}$, we set $\Gamma = (\mathbb{Z}^2, E)$, where $E := \{ (i,j) \to (i,j-1) \} \cup \{ (i,j) \to (i-1,j) \}$ and
let $a_i = (i,-i)$ and $b_j=(2j-n-1,-m-j)$ for $1 \leq i,j \leq n$. Then, $\stairPP{n}{m}$ is in bijection to $\mathcal{P}_{\Gamma}^{\text{NI}}(\mathbf{a},\mathbf{b})$.
For example, we consider 
$$\pi = \,\vcenter{\hbox{\young(6433,421,30,1)}} \in \stairPP{4}{6}.$$
This can be described by means of piles of cubes and is associated with the following non-intersecting paths:
\input{a.tikz}
Furthermore, due to the LGV lemma, $\stairPP{n}{m}$ is in sijection to $\mathcal{P}_{\Gamma}(\mathbf{a},\mathbf{b})$.

Next, we explain what we call \textit{the double LGV lemma technique}. 
Let $\mathbf{a}=(a_1,a_2,\ldots,a_n)$ and $\mathbf{b}=(b_1,b_2,\ldots,b_n)$ be vertices on a directed acyclic graph $\Gamma_1$, and $\mathbf{c}$ and $\mathbf{d}$ vertices on a directed acyclic graph $\Gamma_2$.
According to the bijective proof of the LGV lemma, we have sijections $\mathcal{P}^{\text{NI}}_{\Gamma_1}(\mathbf{a},\mathbf{b}) \sijto \mathcal{P}_{\Gamma_1}(\mathbf{a},\mathbf{b})$ and
$\mathcal{P}^{\text{NI}}_{\Gamma_2}(\mathbf{c},\mathbf{d}) \sijto \mathcal{P}_{\Gamma_2}(\mathbf{c},\mathbf{d})$. Furthermore, these sijections are compatible in many situations.
Therefore, by constructing a (good) sijection
$\mathcal{P}_{\Gamma_1}(\mathbf{a},\mathbf{b}) \sijto \mathcal{P}_{\Gamma_2}(\mathbf{c},\mathbf{d})$,
we can establish a property between combinatorial objects (e.g., plane partitions) associated with non-intersecting paths,
without explicitly handling the non-intersecting conditions.
This method is not entirely new; for example, it has been used in the bijective proof of the hook-content formula for the number of plane partitions \cite{RW}.
In the following this method will be used several times.

\section{Lattice Paths and a Further Transformation on the stair-PPs side}

\subsection{A First Lattice Path Configuration for stair-PPs and Transformation of $S$} \label{ssec::1-LP-stairPP}
Let $n \in \mathbb{Z}_{>0}$ and $m \in \mathbb{Z}_{\geq 0}$.
We define a set-valued statistic $\widetilde{S}$ on $\stairPP{n}{m}$ by $$\widetilde{S}(\pi) = \{ 1 \leq j \leq n \mid \pi_{1,j} \ne m \}.$$
The goal of Section~\ref{ssec::1-LP-stairPP} is to construct a compatible bijection $\bigl(\stairPP{n}{m}, S\bigr) \sijto \bigl(\stairPP{n}{m}, \widetilde{S}\bigr)$. %, together with a refinement.
Here, it is important that the definition of $\widetilde{S}$ is close to the definition of $S$ on $\eSPP{n}{m}$: $$S(\pi') = \{1 \leq j \leq n \mid \pi'_{1,j} \ne 2m \}.$$

We set $\Gamma = (\mathbb{Z}^2, E)$, where $E := \{ (i,j) \to (i,j-1) \} \cup \{ (i,j) \to (i-1,j) \}$, $a_i = (-i,-i)$ and $b_j=(n+1-2j,-m-j)$ for $1 \leq i,j \leq n$.
Then, as stated in the previous section, $\stairPP{n}{m}$ is in bijection to $\mathcal{P}_{\Gamma}^{\text{(NI)}}(\mathbf{a},\mathbf{b})$ and in sijection to $\mathcal{P}_{\Gamma}(\mathbf{a},\mathbf{b})$.
By the construction of the bijection $\stairPP{n}{m} \to \mathcal{P}_{\Gamma}^{\text{(NI)}}(\mathbf{a},\mathbf{b})$, 
the set-valued statistic $S$ on $\stairPP{n}{m}$ naturally extends to $\mathcal{P}_{\Gamma}^{\text{(NI)}}(\mathbf{a},\mathbf{b})$ and further to $\mathcal{P}_{\Gamma}(\mathbf{a},\mathbf{b})$ as follows:
\begin{itemize}
\item let $E_1 := \{ a_i-(0,1) \to a_i \mid 1 \leq i \leq n \}$,
\item $S(P) = E(P) \cap E_1$ for $P \in \supp{\mathcal{P}_{\Gamma}(\mathbf{a},\mathbf{b})}$, where $E(P)$ denotes the set of edges contained in $P$.
\end{itemize}
Similarly, $\widetilde{S}$ naturally extends to $\mathcal{P}_{\Gamma}(\mathbf{a},\mathbf{b})$ as follows:
\begin{itemize}
\item let $E_2 := \{ b_j \to b_j+(1,0) \mid 1 \leq j \leq n \}$,
\item $\widetilde{S}(P) = E(P) \cap E_2$ for $P \in \supp{\mathcal{P}_{\Gamma}(\mathbf{a},\mathbf{b})}$.
\end{itemize}
By definition, the bijection $\stairPP{n}{m} \to \mathcal{P}_{\Gamma}^{\text{(NI)}}(\mathbf{a},\mathbf{b})$ is compatible with $S$ and $\widetilde{S}$.
Furthermore, the sijection $\mathcal{P}_{\Gamma}^{\text{NI}}(\mathbf{a},\mathbf{b}) \sijto  \mathcal{P}_{\Gamma}(\mathbf{a},\mathbf{b})$ given by the LGV lemma is compatible with $S$ and $\widetilde{S}$,
since the sijection does not involve the edge set.
By the double LGV lemma method, what we have to do is construct a compatible sijection $(\mathcal{P}_{\Gamma}(\mathbf{a},\mathbf{b}), S) \sijto (\mathcal{P}_{\Gamma}(\mathbf{a},\mathbf{b}), \widetilde{S})$.
In fact, such a sijection is obtained by rotating each path by 180 degrees around the midpoint of its two endpoints.
This action does not change the endpoints of each path, and it associates edges in $E_1$ with edges in $E_2$.
For example,
\[
    \begin{tikzpicture}[scale=0.6, baseline=(current bounding box.center)]
        \fill[red] (1,-1) circle (2pt);
        \node[above] at (1,-1) {$a_1$};
        \fill[red] (2,-2) circle (2pt);
        \node[above] at (2,-2) {$a_2$};
        \fill[red] (3,-3) circle (2pt);
        \node[above] at (3,-3) {$a_3$};
        \fill[red] (4,-4) circle (2pt);
        \node[above] at (4,-4) {$a_4$};
        \fill[red] (-3,-7) circle (2pt);
        \node[below] at (-3,-7) {$b_1$};
        \draw[cyan, thick](-3.000000,-6.500000) circle (0.2 and 0.5) ;
        \fill[red] (-1,-8) circle (2pt);
        \node[below] at (-1,-8) {$b_2$};
        \draw[cyan, thick](-1.000000,-7.500000) circle (0.2 and 0.5) ;
        \fill[red] (1,-9) circle (2pt);
        \node[below] at (1,-9) {$b_3$};
        \draw[cyan, thick](1.000000,-8.500000) circle (0.2 and 0.5) ;
        \fill[red] (3,-10) circle (2pt);
        \node[below] at (3,-10) {$b_4$};
        \draw[cyan, thick](3.000000,-9.500000) circle (0.2 and 0.5) ;
        \node[above, cyan] at (4,-10) {$E_1$};
        \draw[blue, thick](1,-1)--(1,-2)--(0,-2)--(0,-4)--(-1,-4)--(-1,-5)--(-2,-5)--(-2,-7)--(-3,-7);
        \draw[orange, thick](2,-2)--(1,-2)--(1,-4)--(0,-4)--(0,-6)--(-1,-6)--(-1,-8);
        \draw[blue, thick](3,-3)--(3,-4)--(2,-4)--(2,-6)--(1,-6)--(1,-9);
        \draw[orange, thick](4,-4)--(4,-7)--(3,-7)--(3,-10);
    \end{tikzpicture}
    \hspace{1cm}
    \longrightarrow
    \hspace{1cm}
    \begin{tikzpicture}[scale=0.6, baseline=(current bounding box.center)]
        \fill[red] (1,-1) circle (2pt);
        \node[above] at (1,-1) {$a_1$};
        \draw[cyan, thick](1.000000,-1.500000) circle (0.2 and 0.5) ;
        \fill[red] (2,-2) circle (2pt);
        \node[above] at (2,-2) {$a_2$};
        \draw[cyan, thick](2.000000,-2.500000) circle (0.2 and 0.5) ;
        \fill[red] (3,-3) circle (2pt);
        \node[above] at (3,-3) {$a_3$};
        \draw[cyan, thick](3.000000,-3.500000) circle (0.2 and 0.5) ;
        \fill[red] (4,-4) circle (2pt);
        \node[above] at (4,-4) {$a_4$};
        \draw[cyan, thick](4.000000,-4.500000) circle (0.2 and 0.5) ;
        \node[above, cyan] at (5,-5) {$E_2$};
        \fill[red] (-3,-7) circle (2pt);
        \node[below] at (-3,-7) {$b_1$};
        \fill[red] (-1,-8) circle (2pt);
        \node[below] at (-1,-8) {$b_2$};
        \fill[red] (1,-9) circle (2pt);
        \node[below] at (1,-9) {$b_3$};
        \fill[red] (3,-10) circle (2pt);
        \node[below] at (3,-10) {$b_4$};
        \draw[blue, thick](1,-1)--(0,-1)--(0,-3)--(-1,-3)--(-1,-4)--(-2,-4)--(-2,-6)--(-3,-6)--(-3,-7);
        \draw[orange, thick](2,-2)--(2,-4)--(1,-4)--(1,-6)--(0,-6)--(0,-8)--(-1,-8);
        \draw[blue, thick](3,-3)--(3,-6)--(2,-6)--(2,-8)--(1,-8)--(1,-9);
        \draw[orange, thick](4,-4)--(4,-7)--(3,-7)--(3,-10);
    \end{tikzpicture}
\]

For a further example, please see Figure \ref{fig::4} in Section~8.

\subsection{A second Lattice Path Configuration for stair-PPs} \label{ssec::2-LP-stairPP}
Next, we explain one more lattice path configuration of $\stairPP{n}{m}$.
In this construction, we cut the cubes along the $xy$-plane, instead of the $xz$-plane as in the previous construction.

Let $E := \{ (i,j) \to (i,j+1) \} \cup \{ (i,j) \to (i+1,j) \}$ and set $\Gamma = (\mathbb{Z}^2,E)$.
Define $E' := \{ e \in E \mid e \subset \{ (x,y) \in \mathbb{R}^2 \mid x-y \geq 2m+1 \} \}$ and $\Gamma' = (\mathbb{Z}^2, E')$.
%We set $\Gamma' = (\mathbb{Z}^2, E')$, where $E' := \{ e \in E \mid e \subset \{(x,y) \in \mathbb{R}^2 \mid x-y \geq 2m+1 \} \}$ and $E := \{ (i,j) \to (i,j+1) \} \cup \{ (i,j) \to (i+1,j) \}$.
Let $a_i = (i,-i)$,  $b_i = (i+n,-i+n)$, $\mathbf{a} = (a_1,a_2,\ldots,a_m)$ and $\mathbf{b} = (b_1,b_2,\ldots,b_m)$,
then $\stairPP{n}{m}$ is in bijection to $\mathcal{P}_{\Gamma'}^{\text{NI}}(\mathbf{a},\mathbf{b})$.
For example, we consider 
$$\pi = \,\vcenter{\hbox{\young(6433,421,30,1)}} \in \stairPP{4}{6}.$$
This is described as piles of cubes as follows, and is associated with the following non-intersecting paths:
\input{c.tikz}
By applying the LGV lemma, this set of non-intersecting paths is in sijection to $\mathcal{P}_{\Gamma'}(\mathbf{a},\mathbf{b})$.
Under this sijection, the set-valued statistic $\widetilde{S}$ naturally extends to $\mathcal{P}_{\Gamma'}(\mathbf{a},\mathbf{b}) (\supseteq \mathcal{P}^\text{(NI)}_{\Gamma'}(\mathbf{a},\mathbf{b}))$ and further to $\mathcal{P}_{\Gamma}(\mathbf{a},\mathbf{b})$ as:
\begin{itemize}
\item let $E_3 := \{ e \in E \mid e \subset \{ (x,y) \in \mathbb{R}^2 \mid y \geq n-1 \} \}$,
\item $\widetilde{S}(P) = E(P) \cap E_3$ for $P \in \supp{\mathcal{P}_{\Gamma}(\mathbf{a},\mathbf{b})}$.
%\item {\color{red} transpose shi nai to Construction 5.2. de kowareru <- $a$ wo reflect shita hou ga yoi kamo?}
\end{itemize}
To better organize the signed set $\mathcal{P}_{\Gamma'}(\mathbf{a},\mathbf{b})$, we prepare the following construction as a lemma.

\begin{construction} \label{construction::mirror}
%Let $\Gamma = (\mathbb{Z}^2, E)$. %, where
%\[
%  E := \{ (i,j) \to (i,j-1) \} \cup \{ (i,j) \to (i+1,j) \}.
%\]
%and let $\Gamma' = (\mathbb{Z}^2, E')$, where
%\[
%  E' := \{ e \in E \mid e \subset \{(x,y) \in \mathbb{R}^2 \mid x-y \geq 2m+1 \} \}.
%\]
Let $a=(a_x,a_y), b=(b_x,b_y) \in \mathbb{Z}^2$ satisfy $a_x-a_y \geq 2m+1$ and $b_x-b_y \geq 2m+1$.
Then we construct a sijection
\[
	\mathcal{P}_{\Gamma'}(a \to b) \sijto \mathcal{P}_{\Gamma}(a \to b) \setminus \mathcal{P}_{\Gamma}(a' \to b),
\]
where $a' = (a_y+2m+2,\, a_x-2m-2)$ is the reflection of $a$ across the line $x-y = 2m+2$.

Constructing this sijection is equivalent to constructing a bijection
\[
	\mathcal{P}_{\Gamma}(a \to b) \setminus \mathcal{P}_{\Gamma'}(a \to b)
	\;\to\; \mathcal{P}_{\Gamma}(a' \to b).
\]
Let $P \in \mathcal{P}_{\Gamma}(a \to b)^+ \setminus \mathcal{P}_{\Gamma'}(a \to b)^+$.  
Then $P$ intersects the line $x-y=2m+2$.  
By reflecting the part of $P$ before the last such intersection—that is, the intersection with maximum $x+y$—we obtain the desired bijection.
For example,
\[
	\begin{tikzpicture}[scale=0.6, baseline=(current bounding box.center)]
        \draw[orange, dashed](6,-7)--(11,-2) node[pos=0, left] {$x-y=2m+1$};
        \draw[orange, very thick](6,-8)--(12,-2) node[pos=0, left] {$x-y=2m+2$};
        \draw[blue, thick](6,-6)--(9,-6)--(9,-3)--(10,-3)--(10,-2);
        \fill[red] (6,-6) circle (2pt);
        \node[left] at (6,-6) {$a$};
        \fill[red] (10,-2) circle (2pt);
        \node[above] at (10,-2) {$b$};
        \fill[green] (9,-5) circle (2pt);
	\end{tikzpicture}
    \hspace{1cm}
    \longrightarrow
    \hspace{1cm}
	\begin{tikzpicture}[scale=0.6, baseline=(current bounding box.center)]
        \draw[orange, dashed](6,-7)--(11,-2) node[pos=0, left] {$x-y=2m+1$};
        \draw[orange, very thick](6,-8)--(12,-2) node[pos=0, left] {$x-y=2m+2$};
        \draw[blue, dashed](6,-6)--(9,-6)--(9,-3)--(10,-3)--(10,-2);
        \draw[blue, thick](8,-8)--(8,-5)--(9,-5)--(9,-3)--(10,-3)--(10,-2);
        \fill[red] (6,-6) circle (2pt);
        \node[left] at (6,-6) {$a$};
        \fill[red] (10,-2) circle (2pt);
        \node[above] at (10,-2) {$b$};
        \fill[green] (9,-5) circle (2pt);
        \fill[red] (8,-8) circle (2pt);
        \node[below] at (8,-8) {$a'$};
	\end{tikzpicture}
\]
\end{construction}
\begin{corollary}
In the same setting as Construction~\ref{construction::mirror}, we set $a = a_i$ and $b=b_j$. Then, the bijection 
\[
	\mathcal{P}_{\Gamma'}(a_i \to b_j) \sijto \mathcal{P}_{\Gamma}(a_i \to b_j) \setminus \mathcal{P}_{\Gamma}(a_i' \to b_j),
\]
constructed in Construction~\ref{construction::mirror} is compatible with the set-valued statistic $\widetilde{S}$
because any path from $a_i$ to $b_j$ cannot intersect the line $x-y=2m+2$ once it has entered the region $\{ (x,y) \mid y \ge n-1 \}$.
%because a path from $a_i$ to $b_j$ does not intersect the line $x-y=2m+2$ after the area $\{ (x,y) \in \mathbb{R}^2 \mid y \geq n-1 \}$.
\end{corollary}
Due to this construction we have
\begin{align*}
	\mathcal{P}_{\Gamma'}(\mathbf{a},\mathbf{b}) =& \bigsqcup_{\sigma \in S_m} \sgn{\sigma} \prod_{i=1}^m \mathcal{P}_{\Gamma'}(a_{\sigma_i} \to b_{i}) \\
	\sijto& \bigsqcup_{\sigma \in S_m} \sgn{\sigma} \prod_{i=1}^m \left\{ \mathcal{P}_{\Gamma}(a_{\sigma_i} \to b_{i}) \setminus \mathcal{P}_{\Gamma}(a_{\sigma_i}' \to b_{i}) \right\} \\
	=& \bigsqcup_{\sigma \in S_m} \sgn{\sigma} \bigsqcup_{\mathbf{e} \in \{0,1\}^m} (-1)^{\lvert \mathbf{e} \rvert} \prod_{i=1}^m \left\{ \mathcal{P}_{\Gamma}(a_{\sigma_i}^{(e_i)} \to b_{i}) \right\},
\end{align*}
where
\begin{itemize}
\item $a_i' = (-i+2m+2, i-2m-2)$ is the point symmetric to $a_i$ with respect to the line $x-y=2m+2$,
\item $a_i^{(0)} = a_i$, $a_i^{(1)} = a_i'$,
\item $\lvert \mathbf{e} \rvert = \sum_{i=1}^m e_i$ for $\mathbf{e} = (e_1,e_2,\ldots,e_m) \in \{0,1\}^m$.
\end{itemize}
Let $\mathcal{C}(u,v) := \mathcal{P}_{\Gamma}\bigl((0,0) \to (u-v,v)\bigr)$ for integers $u,v$.
It is well known that $\# \mathcal{C}(u,v) = \binom{u}{v}$,
where we set $\binom{u}{v} = 0$ if the condition $0 \le v \le u$ is not satisfied.
We write $\mathcal{P}_{\Gamma}( (a_x,a_y) \to (b_x,b_y) ) = \mathcal{C}(b_x-a_x+b_y-a_y, b_y-a_y)$, since we have a canonical bijection between their positive parts and their negative parts are empty.
Summarizing Section~\ref{ssec::2-LP-stairPP}: we have constructed a sijection,
\begin{multline} \label{sij::stair-I}
	\stairPP{n}{m} \to \mathcal{P}_{\Gamma'}^\text{(NI)}(\mathbf{a},\mathbf{b}) \sijto \mathcal{P}_{\Gamma'}(\mathbf{a},\mathbf{b}) \\
	\sijto \bigsqcup_{\sigma \in S_m} \sgn{\sigma} \bigsqcup_{\mathbf{e} \in \{0,1\}^m} (-1)^{\lvert \mathbf{e} \rvert} \prod_{i=1}^m \left [ \begin{cases}\mathcal{C}(2n,n-i+\sigma_i) & e_i = 0 \\ \mathcal{C}(2n,n-i-\sigma_i+2m+2)& e_i = 1 \end{cases} \right ].
\end{multline}
In order to ensure the desired compatibility, we extend the set-valued statistic $\widetilde{S}$ as follows:
\begin{quote}
%For $\displaystyle (((P_1,P_2,\ldots,P_m),\mathbf{e}),\sigma) \in \supp{\bigsqcup_{\sigma \in S_m} \sgn{\sigma} \bigsqcup_{\mathbf{e} \in \{0,1\}^m} (-1)^{\lvert \mathbf{e} \rvert} \prod_{i=1}^m \left [ \begin{cases}\mathcal{C}(2n,n-i+\sigma_i) & e_i = 0 \\ \mathcal{C}(2n,n-i-\sigma_i+2m+2)& e_i = 1 \end{cases} \right ]}$,
For $\displaystyle (((P_1,P_2,\ldots,P_m),\mathbf{e}),\sigma) \in \supp{\text{the RHS of (\ref{sij::stair-I})}}$,
the value of $\widetilde{S}$ is $$\{ e \in E(P_1) \mid  e \subset \{ (x,y) \in \mathbb{R}^2 \mid y \geq n-1+\sigma_1 \} \}. $$
In particular, $\#\widetilde{S}$ equals the number of consecutive horizontal edges at the end of the path $P_1$.
\end{quote}

\subsection{The Definition of $I_m$}
We introduce a new signed set of combinatorial objects $I_m \left[ \begin{gathered}a_1,a_2, \ldots, a_m \\ b_1,b_2,\ldots,b_m \end{gathered} \right]$.
\begin{definition}
Let $m \in \mathbb{Z}_{>0}$ and $a_1,a_2,\ldots,a_m, b_1,b_2, \ldots, b_m \in \mathbb{Z}$.
Then, $I_m \left[ \begin{gathered}a_1,a_2, \ldots, a_m \\ b_1,b_2,\ldots,b_m \end{gathered} \right]$ is defined as follows.
First, the support is
\[
	\supp{I_m \left[ \begin{gathered}a_1,a_2, \ldots, a_m \\ b_1,b_2,\ldots,b_m \end{gathered} \right]} = S_m \times \{0,1\}^m.
\]
Then, a sign is defined by
\[
\begin{tikzcd}
\textrm{sgn} \arrow[r, phantom, "\colon"]&
S_m \times \{0,1\}^m \arrow[r] \arrow[d, phantom, "\rotatebox{90}{$\in$}"] 
  & \{+1,-1\} \arrow[d, phantom, "\rotatebox{90}{$\in$}"]  \\ &
(\sigma, (t_1,t_2,\ldots,t_m)) \arrow[r, mapsto] & \sgn{\sigma} \cdot \prod_{i=1}^m (-1)^{t_i},
\end{tikzcd}
\]
where $\sgn{\sigma}$ for $\sigma \in S_m$ is the sign of the permutation involved.
\end{definition}

In addition, we define a statistic $\eta_i$, for $1 \leq i \leq m$, on $I_m$ as follow:
\[
	\eta_i \colon (\sigma, (t_1,t_2,\ldots,t_m)) \mapsto a_i - (-1)^{t_i}b_{\sigma_i}.
\]
Then $\lvert \eta_i \rvert \colon (\sigma, (t_1,t_2,\ldots,t_m)) \mapsto \lvert a_i - (-1)^{t_i}b_{\sigma_i} \rvert$ is also a statistic on $I_m$.
Using these statistics, the sijection (\ref{sij::stair-I}) can be rewritten as follows:
\begin{align} \label{sij::stair-I-2-1}
  \stairPP{n}{m}
  \sijto&
  \bigsqcup_{\alpha \in I_m\Bigl[\begin{smallmatrix}
    m, m-1, \ldots, 1 \\
    m, m-1, \ldots, 1
  \end{smallmatrix}\Bigr]}
  \prod_{i=1}^m \mathcal{C}(2n,\, n+\eta_i(\alpha)) \\
  =&  \label{sij::stair-I-2-2}
  \bigsqcup_{\alpha \in I_m\Bigl[\begin{smallmatrix}
    m, m-1, \ldots, 1 \\
    m, m-1, \ldots, 1
  \end{smallmatrix}\Bigr]}
  \prod_{i=1}^m \mathcal{C}(2n,\, n+\lvert \eta_i \rvert(\alpha)),
\end{align}
where the second equality means a disjoint union of identities on $\mathcal{C}(2n,\, n+\eta_i(\alpha))$ and/or canonical sijections
\[
	\mathcal{C}(2n, n+w) = \mathcal{C}(2n, n-w),
\]
obtained by transposing the paths, i.e., reflecting the paths with respect to the line $x=y$.
We now extend the set-valued statistic $\widetilde{S}$ as follows:
\begin{quote}
For $((P_1,P_2,\ldots,P_m),\alpha) \in \supp{\text{the RHS of (\ref{sij::stair-I-2-1}) or (\ref{sij::stair-I-2-2})}}$, 
the value of $\widetilde{S}$ is $$ \{ e \in E(P_1) \mid e \subset \{ (x,y) \in \mathbb{R}^2 \mid y \geq n+\lvert\eta_1\rvert(\alpha) \} \}.$$
  \end{quote}
Then, since $\eta_1(\alpha) \geq 0$ for $\alpha \in I \left[ \begin{gathered} m,m-1,\ldots,1 \\ m,m-1,\ldots,1 \end{gathered} \right]$,
the constructed sijection is compatible with $\widetilde{S}$. Combining it with the sijection constructed in Section~\ref{ssec::1-LP-stairPP}, we obtain a compatible sijection
\[
	(\stairPP{n}{m}, S) \sijto \left(
  \bigsqcup_{\alpha \in I_m\Bigl[\begin{smallmatrix}
    m, m-1, \ldots, 1 \\
    m, m-1, \ldots, 1
  \end{smallmatrix}\Bigr]}
  \prod_{i=1}^m \mathcal{C}(2n,\, n+\lvert \eta_i \rvert(\alpha)), \,\,\widetilde{S}  \right).
\]

\section{Lattice Paths and a Further Transformation on the even SPPs side}
In this section, we organize a lattice paths configuration for $\eSPP{n}{m}$.
\subsection{A Lattice Path Configuration for even SPPs}
First, we explain a lattice path configuration for $\SPP{n}{2m}$, which includes $\eSPP{n}{m}$.
Let $a_i = (i,-i+n)$, $b_i = (i+n, -i+n)$. Then, $\SPP{n}{2m}$ is in bijection to
\[
	\bigsqcup_{\tau \in \binom{n+2m}{2m}} \mathcal{P}^{\text{(NI)}}_{\Gamma}(\mathbf{a}_\tau,\mathbf{b}),
\]
where $\tau \in \binom{n+2m}{2m}$ means that $1 \leq \tau_1 < \tau_2 < \cdots < \tau_{2m} \leq n+2m$, and let $\mathbf{a}_\tau = (a_{\tau_1}, a_{\tau_2}, \ldots, a_{\tau_{2m}})$ and $\mathbf{b} = (b_1,b_2,\ldots, b_{2m})$.
For example, let us consider
$$\pi = \,\vcenter{\hbox{\young(6433,:421,::20,:::0)}} \in \eSPP{4}{3} \subset \stairPP{4}{6}.$$
This is described as piles of cubes as below and is associated with the following non-intersecting paths:
\input{d.tikz}
Under this bijection, $T \in \SPP{n}{2m}$ is even if and only if the corresponding $\tau$ satisfies $\tau_{2i} = \tau_{2i-1}+1$, for all $i=1,2,\ldots,m$. Therefore, $\eSPP{n}{m}$ is in bijection to
\[
	\bigsqcup_{\tau \in \binom{n+2m}{2m} \colon \tau_{2i} = \tau_{2i-1}+1 \text{ for $i=1,2,\ldots,m$ }} \mathcal{P}^{\text{(NI)}}_{\Gamma}(\mathbf{a}_\tau,\mathbf{b}).
\]
In fact, we can extend the range of $\tau$ as follows:
\[
	\bigsqcup_{\tau \in \{1,2,\ldots,n+2m\}^{2m} \colon \tau_1 < \tau_3 < \cdots < \tau_{2m-1} \text{ and }\tau_{2i} = \tau_{2i-1}+1 \text{ for $i=1,2,\ldots,m$ }} \mathcal{P}^{\text{(NI)}}_{\Gamma}(\mathbf{a}_\tau,\mathbf{b}),
\]
since if $\tau \not\in \binom{n+2m}{2m}$ then $\mathcal{P}^{\text{(NI)}}_{\Gamma}(\mathbf{a}_\tau,\mathbf{b})$ is empty.
In addition, by the LGV lemma, this is in sijection to 
\begin{equation}
	\bigsqcup_{\tau \in \{1,2,\ldots,n+2m\}^{2m} \colon \tau_1 < \tau_3 < \cdots < \tau_{2m-1} \text{ and }\tau_{2i} = \tau_{2i-1}+1 \text{ for $i=1,2,\ldots,m$ }} \mathcal{P}_{\Gamma}(\mathbf{a}_\tau,\mathbf{b}). \label{sset::6-1-1}
\end{equation}
We now prepare some notations. Let
\begin{gather*}
S_{2m}' = \{\sigma \in S_{2m} \colon \min(\sigma_1, \sigma_2) < \min(\sigma_3,\sigma_4) < \cdots < \min(\sigma_{2m-1}, \sigma_{2m}) \},\\
T_{n,m} = \{\widetilde\tau \in \{1,2,\ldots,n+2m\}^m \colon \text{ elements of $\widetilde\tau$ are pairwise distinct.} \}
\end{gather*}
Then we have the following bijection.
\begin{construction} \label{bij::TS-ST}
We construct a bijection $\varphi_\text{\ref{bij::TS-ST}} \colon \tbinom{n+2m}{m} \times S_{2m} \to T_{n,m} \times S_{2m}'$.
In fact, both sides are in bijection to $(T_{n,m} \times S_{2m}) / S_m$, where $S_m$ acts as follows:
\begin{quote}
	Let $\rho \in S_m$ and $(\widetilde\tau, \sigma) \in (T_{n,m} \times S_{2m})$. Then,
	\begin{equation} \label{def::action}
		\rho \cdot (\widetilde\tau, \sigma) = ((\widetilde\tau_{\rho_1},\widetilde\tau_{\rho_2},\ldots,\widetilde\tau_{\rho_m}), (\sigma_{2\rho_1-1}, \sigma_{2\rho_1}, \sigma_{2\rho_2-1}, \sigma_{2\rho_2}, \ldots, \sigma_{2\rho_m-1}, \sigma_{2\rho_m})).
	\end{equation}
	In particular, under this action $\sgn{\sigma}$ is preserved.
\end{quote}
Note that in the LHS, we choose representatives such that $\widetilde\tau_1 < \widetilde\tau_2 < \cdots < \widetilde\tau_m$,
and that in the RHS we choose representatives such that $\min(\sigma_1, \sigma_2) < \min(\sigma_3,\sigma_4) < \cdots < \min(\sigma_{2m-1}, \sigma_{2m})$.
\end{construction}

We return to (\ref{sset::6-1-1}).
By definition, the signed set (\ref{sset::6-1-1}) is in sijection to
\begin{multline} \label{eq::6-1-1}
	\bigsqcup_{\tau \in \{1,2,\ldots,n+2m\}^{2m} \colon \tau_1 < \tau_3 < \cdots < \tau_{2m-1} \text{ and }\tau_{2i} = \tau_{2i-1}+1 \text{ for $i=1,2,\ldots,m$ }} \bigsqcup_{\sigma \in S_{2m}} \sgn{\sigma}\prod_{i=1}^{2m}\mathcal{P}_{\Gamma}(a_{\tau_{i}}, b_{\sigma_i}) \\
	= \bigsqcup_{\widetilde\tau \in \tbinom{n+2m}{m}} \bigsqcup_{\sigma \in S_{2m}} \sgn{\sigma} \prod_{i=1}^m \left( \mathcal{P}_{\Gamma}(a_{\widetilde\tau_{i}}, b_{\sigma_{2i-1}}) \times \mathcal{P}_{\Gamma}(a_{\widetilde\tau_{i}+1}, b_{\sigma_{2i}}) \right). 
\end{multline}
Using the bijection $\varphi_\text{\ref{bij::TS-ST}}$, the RHS of (\ref{eq::6-1-1}) is in sijection to 
\[
	\bigsqcup_{\widetilde\tau \in T_{n,m}} \bigsqcup_{\sigma \in S_{2m}'} \sgn{\sigma} \prod_{i=1}^m \left( \mathcal{P}_{\Gamma}(a_{\widetilde\tau_{i}}, b_{\sigma_{2i-1}}) \times \mathcal{P}_{\Gamma}(a_{\widetilde\tau_{i}+1}, b_{\sigma_{2i}}) \right). 
\]
Furthermore, this is in sijection to 
\begin{multline} \label{sset::6-1-2}
	\bigsqcup_{\widetilde\tau \in \{1,2,\ldots,n+2m\}^m } \bigsqcup_{\sigma \in S_{2m}'} \sgn{\sigma} \prod_{i=1}^m \left( \mathcal{P}_{\Gamma}(a_{\widetilde\tau_{i}}, b_{\sigma_{2i-1}}) \times \mathcal{P}_{\Gamma}(a_{\widetilde\tau_{i}+1}, b_{\sigma_{2i}}) \right) \\
	=  \bigsqcup_{\sigma \in S_{2m}'} \sgn{\sigma} \prod_{i=1}^m \bigsqcup_{\widetilde\tau_i \in \{1,2,\ldots,n+2m\} } \left( \mathcal{P}_{\Gamma}(a_{\widetilde\tau_{i}}, b_{\sigma_{2i-1}}) \times \mathcal{P}_{\Gamma}(a_{\widetilde\tau_{i}+1}, b_{\sigma_{2i}}) \right)
\end{multline}
since we have a sijection
\begin{align*}
	\bigsqcup_{\widetilde\tau \in \{1,2,\ldots,n+2m\}^m \setminus T_{n,m} } \bigsqcup_{\sigma \in S_{2m}'} \sgn{\sigma} \prod_{i=1}^m \left( \mathcal{P}_{\Gamma}(a_{\widetilde\tau_{i}}, b_{\sigma_{2i-1}}) \times \mathcal{P}_{\Gamma}(a_{\widetilde\tau_{i}+1}, b_{\sigma_{2i}}) \right)
	\sijto (\emptyset,\emptyset),
\end{align*}
where we associate $((((P_{1,1},P_{1,2}),(P_{2,1},P_{2,2}),\ldots,(P_{m,1},P_{m,2})),\sigma),\widetilde\tau)$
with the element obtained by swapping $\sigma_{2j}$ and $P_{j,2}$ with $\sigma_{2k}$ and $P_{k,2}$, respectively, where $(j,k)$ is the index pair such that $\widetilde\tau_j = \widetilde\tau_k$ with minimum $(\widetilde\tau_j,\sigma_{2j-1},\sigma_{2k-1})$, in lexicographical order.
Note that this part is a bit technical. However, we essentially use the fact that when $\tau$ has multiple values, there are no non-intersecting paths (namely, $\mathcal{P}_\Gamma^\text{(NI)}(\mathbf{a}_\tau, \mathbf{b})$ is empty) and therefore lattice paths for such $\tau$ cancel among them, according to the LGV lemma.

Let $b'_i = (i,-i)$. Then, we have
\begin{align*}
	\mathcal{P}_{\Gamma}(a_{\widetilde\tau_{i}}, b_{\sigma_{2i-1}}) \times \mathcal{P}_{\Gamma}(a_{\widetilde\tau_{i}+1}, b_{\sigma_{2i}}) 
	&= \mathcal{P}_{\Gamma}(a_{\widetilde\tau_{i}}, b_{\sigma_{2i-1}}) \times \mathcal{P}_{\Gamma}(b'_{\sigma_{2i}}, a_{\widetilde\tau_{i}+1}) & (\text{reflection}) \\
	&= \mathcal{P}_{\Gamma}(a_{\widetilde\tau_{i}}, b_{\sigma_{2i-1}}) \times \mathcal{P}_{\Gamma}(b'_{\sigma_{2i}-1}, a_{\widetilde\tau_{i}})  & (\text{translation})\\
	&= \mathcal{P}_{\Gamma}(b'_{\sigma_{2i}-1}, a_{\widetilde\tau_{i}}) \times \mathcal{P}_{\Gamma}(a_{\widetilde\tau_{i}}, b_{\sigma_{2i-1}}), &
\end{align*}
by reflecting the paths across the line $x+y=n$ and translating the paths by $(-1,1)$.
For example, please see Figure \ref{fig::5} in Section~8.
Now, since $\widetilde{\tau}_{i}$ runs over all integers in $\{1,2,\ldots,n+2m\}$,
\begin{align*}
\bigsqcup_{\widetilde\tau_i \in \{1,2,\ldots,n+2m\}} \left( \mathcal{P}_{\Gamma}(a_{\widetilde\tau_{i}}, b_{\sigma_{2i-1}}) \times \mathcal{P}_{\Gamma}(a_{\widetilde\tau_{i}+1}, b_{\sigma_{2i}}) \right)
 &= \bigsqcup_{\widetilde\tau_i \in \{1,2,\ldots,n+2m\}} \left( \mathcal{P}_{\Gamma}(b'_{\sigma_{2i}-1}, a_{\widetilde\tau_{i}}) \times \mathcal{P}_{\Gamma}(a_{\widetilde\tau_{i}}, b_{\sigma_{2i-1}}) \right) \\
 &= \mathcal{P}_{\Gamma}(b'_{\sigma_{2i}-1}, b_{\sigma_{2i-1}}) \\
 &= \mathcal{C}(2n,n-1+\sigma_{2i}-\sigma_{2i-1}).
\end{align*}
As a result, we obtain the sijection
\begin{equation}
	\eSPP{n}{m} \sijto \bigsqcup_{\sigma \in S'_{2m}} \sgn{\sigma} \prod_{i=1}^m \mathcal{C}(2n,n-1+\sigma_{2i}-\sigma_{2i-1}). \label{sij::6-1-1}
\end{equation}
By the construction, the statistic $S$ naturally extends to the RHS of (\ref{sij::6-1-1}) as:
\begin{quote}
	For $((P_1,P_2,\ldots,P_m),\sigma) \in (\text{the RHS of (\ref{sij::6-1-1})})$, the value of $S$ is
	\[
		\begin{cases}
			\{ e \in E(P_1) \mid e \subset \{(x,y) \in \mathbb{R}^2 \mid y = n-1+\sigma_2-\sigma_1\} \}, & (\sigma_1 = 1) \\
			\{ e \in E(P_1) \mid e \subset \{(x,y) \in \mathbb{R}^2 \mid x=0\} \}. & (\sigma_2=1)
		\end{cases}
	\]
	Note that, by the definition of $S_{2m}'$, we have $\sigma_1=1$ or $\sigma_2=1$.
\end{quote}
Note also that, except for the step where we apply the LGV lemma, this sijection consists of translations, reflections and concatenations, which essentially do not involve the value of $S$.
In addition, the sijection derived from the LGV lemma does not change the value of $S$ since edges on the line $y=n-1$ can be used only for a path to $b_1$.

\subsection{The Definition of $J_m$}
We introduce yet another signed set of combinatorial objects $J_m \left< x_1,x_2,\ldots,x_{2m} \right>$.

\begin{definition}
Let $m \in \mathbb{Z}_{>0}$ and $x_1,x_2,\ldots,x_{2m} \in \mathbb{Z}$.
Then, $J_m \left< x_1,x_2,\ldots,x_{2m} \right>$ is defined as follows.
First, its support is
\begin{align*}
	\supp{J_m \left< x_1,x_2,\ldots,x_{2m} \right>} &= S_{2m}' \\
	&= \{\sigma \in S_{2m} \mid \min(\sigma_1, \sigma_2) < \min(\sigma_3,\sigma_4) < \cdots < \min(\sigma_{2m-1}, \sigma_{2m}) \},
\end{align*}
and the sign is the same as $\sgn{\sigma}$ for $\sigma \in S_{2m}$, the sign of permutations.
\end{definition}

In addition, we define a statistic $\eta_i$, for $1 \leq i \leq m$, on $J_m$ as:
\[
	\eta_i \colon \sigma \mapsto x_{\sigma_{2i}}-x_{\sigma_{2i-1}}-1.
\]
Then, $\lvert \eta_i \rvert \colon \sigma \mapsto \lvert x_{\sigma_{2i}}-x_{\sigma_{2i-1}}-1 \rvert$ is also a statistic on $J_m$.
Using these statistics, the sijection (\ref{sij::6-1-1}) can be rewritten as follows:
\begin{align}
	\eSPP{n}{m} \sijto& \bigsqcup_{\sigma \in J_m\left<1,2,\ldots,2m\right>} \prod_{i=1}^m \mathcal{C}(2n,n+\eta_i(\sigma)) \notag \\
	=& \bigsqcup_{\sigma \in J_m\left<1,2,\ldots,2m\right>} \prod_{i=1}^m \mathcal{C}(2n,n+\lvert\eta_i\rvert(\sigma)) \label{sij::6-2}
\end{align}
where the second equality means a disjoint union of identities on $\mathcal{C}(2n,\, n+\eta_i(\sigma))$ and/or canonical sijections
\[
	\mathcal{C}(2n, n+w) = \mathcal{C}(2n, n-w),
\]
obtained by transposing the paths, i.e., reflecting the paths with respect to the line $x=y$.
Furthermore, we rotate the paths when $\eta_i < 0$, in order to ensure the desired compatibility with the following extended statistic:
\begin{quote}
For $((P_1,P_2,\ldots,P_m),\sigma) \in \supp{\text{the RHS of (\ref{sij::6-2})}}$, 
we set the value of $S$ to be $$ \{ e \in E(P_1) \mid e \subset \{ (x,y) \in \mathbb{R}^2 \mid y \geq n+\lvert\eta_1\rvert(\alpha) \} \}.$$
  \end{quote}
For $\sigma \in J_m\left<1,2,\ldots,2m\right>$, we have $\eta_1(\sigma) \geq 0$ if $\sigma_1=1$, and $\eta_1(\sigma) < 0$ if $\sigma_2=1$.
Therefore, the sijection (\ref{sij::6-2}) is compatible with $S$. As a result, we obtain a compatible sijection
\[
	(\eSPP{n}{m}, S) \sijto \left( \bigsqcup_{\sigma \in J_m\left<1,2,\ldots,2m\right>} \prod_{i=1}^m \mathcal{C}(2n,n+\lvert\eta_i\rvert(\sigma)),\,\, S  \right).
\]

\section{The sijection between $I_m$ and $J_m$}
In this section, we construct a compatible sijection
\begin{equation} \label{sij::7-main}
	\left(
  \bigsqcup_{\alpha \in I_m\Bigl[\begin{smallmatrix}
    m, m-1, \ldots, 1 \\
    m, m-1, \ldots, 1
  \end{smallmatrix}\Bigr]}
  \prod_{i=1}^m \mathcal{C}(2n,\, n+\lvert \eta_i \rvert(\alpha)), \,\,\widetilde{S}  \right)
  \sijto
  \left( \bigsqcup_{\sigma \in J_m\left<1,2,\ldots,2m\right>} \prod_{i=1}^m \mathcal{C}(2n,n+\lvert\eta_i\rvert(\sigma)),\,\, S  \right),
\end{equation}
which completes the construction of the desired compatible bijection $(\stairPP{n}{m},S) \sijto (\eSPP{n}{m},S)$.

In fact, since both sides are very similar to each other, it suffices to construct a sijection
\[
	\varphi \colon I_m\Bigl[\begin{smallmatrix}
    m, m-1, \ldots, 1 \\
    m, m-1, \ldots, 1
  \end{smallmatrix}\Bigr]
  \sijto J_m\left<1,2,\ldots,2m\right>,
\]
and a family of permutations of $\{1,2,\ldots,m\}$ \begin{equation} \label{sigma7}\{ \sigma_s \}_{s \in \supp{ I_m\Bigl[\begin{smallmatrix}
    m, m-1, \ldots, 1 \\
    m, m-1, \ldots, 1
  \end{smallmatrix}\Bigr]} \sqcup \supp{J_m\left<1,2,\ldots,2m\right>}}\end{equation}
satisfying the following condition:
\begin{quote}
For all $s \in \supp{ I_m\Bigl[\begin{smallmatrix}
    m, m-1, \ldots, 1 \\
    m, m-1, \ldots, 1
  \end{smallmatrix}\Bigr]} \sqcup \supp{J_m\left<1,2,\ldots,2m\right>}$,% $\sigma_s(1) =1$, $\sigma_s \circ \sigma_{\varphi(s)} = \id$ and
  \[
  	\sigma_s(1)=1;\quad \sigma_s \circ \sigma_{\varphi(s)}=\id;
\quad |\eta_i|(s)=|\eta_{\sigma_s(i)}|(\varphi(s)).
  	%\lvert \eta_i \rvert(s) = \lvert \eta_{\sigma_s(i)} \rvert(\varphi(s)).
  \]
\end{quote}
Note that the condition implies that the sijection $\varphi$ is compatible with the statistic $\lvert \eta_1 \rvert$.
From these objects, we obtain the desired sijection as follows (c.f. Definition~\ref{def_sij_DIS}):
\begin{quote}
For $s = ((P_1,P_2,\ldots,P_m),\alpha) \in (\text{the LHS of (\ref{sij::7-main})})$, we associate
\[
	((P_{\sigma_s(1)}, P_{\sigma_s(2)}, \ldots, P_{\sigma_s(m)}), \varphi(s)).
\]
Similarly for the RHS.
\end{quote}
 
We extend this definition of $\lvert \eta_i \rvert $ to several different signed sets as follows:
\begin{itemize}
\item $\lvert \eta_1 \rvert (x) = \lvert x \rvert$ for $x \in \mathbb{Z}$.
\item $\lvert \eta_1 \rvert ((x,\alpha)) = \lvert x \rvert$, and for $i \geq 2$, $\lvert \eta_i \rvert((x,\alpha)) = \lvert \eta_{i-1} \rvert(\alpha)$, for $(x,\alpha) \in \mathbb{Z} \times I_{m-1} [\ast]$.
\item $\lvert \eta_1 \rvert ((x,\sigma)) = \lvert x \rvert $, and for $i \geq 2$, $\lvert \eta_i \rvert((x,\sigma)) = \lvert \eta_{i-1} \rvert(\sigma)$, for $(x,\sigma) \in \mathbb{Z} \times J_{m-1} \left<\ast\right>$.
\item $\lvert \eta_1 \rvert ((\rho,\sigma)) = \lvert \eta_1 \rvert(\rho)$,  and for $i \geq 2$, $\lvert \eta_i \rvert((\rho,\sigma)) = \lvert \eta_{i-1} \rvert(\sigma)$, for $(\rho,\sigma) \in J_{1} \left<\ast\right> \times J_{m-1} \left<\ast\right>$,
\end{itemize}
and similarly for $\eta_1, \eta_2,\ldots, \eta_m$.

\begin{construction} \label{construction::7-1}
Let $m \in \mathbb{Z}_{>0}$ and $x_1,x_2,\ldots, x_{2m}, d \in \mathbb{Z}$. Then, we have a sign-preserving bijection
\[
	J_m\left<x_1,x_2,\ldots,x_{2m}\right> = J_m\left<x_1+d, x_2+d, \ldots, x_{2m}+d\right>,
\]
which is compatible with $\eta_1, \eta_2,\ldots, \eta_m$ and therefore with $\lvert \eta_1 \rvert, \lvert \eta_2 \rvert, \ldots, \lvert \eta_m \rvert$.
In fact, it holds that $$\supp{J_m\left<x_1,x_2,\ldots,x_{2m}\right>} = \supp{J_m\left<x_1+d, x_2+d, \ldots, x_{2m}+d\right>} = S_{2m}',$$
and $\id_{S_{2m}'}$ is a compatible sign-preserving bijection between these two signed sets,
since $$\eta_i (\sigma) =  x_{\sigma_{2i}}-x_{\sigma_{2i-1}}-1 =  (x_{\sigma_{2i}}+d)-(x_{\sigma_{2i-1}}+d)-1.$$
\end{construction}

\begin{construction} \label{construction::7-2}
Let $m \in \mathbb{Z}_{>0}$ and $x_1,x_2,\ldots, x_{2m} \in \mathbb{Z}$.
Then, we have a sign-preserving bijection
\[
	J_m \left<x_1,x_2,\ldots, x_{2m} \right> 
	= \bigsqcup_{i=2}^{2m} (-1)^i J_1 \left< x_1, x_i \right> \times J_{m-1}\left<x_2,x_3,\ldots, x_{i-1}, x_{i+1}, \ldots, x_{2m} \right>,
\]
which is compatible with $\eta_1, \eta_2,\ldots, \eta_m$ and therefore with $\lvert \eta_1 \rvert, \lvert \eta_2 \rvert, \ldots, \lvert \eta_m \rvert$.
With each $$\sigma \in \supp{J_m \left<x_1,x_2,\ldots, x_{2m} \right>} = S_{2m}',$$ we associate $((\rho, \sigma'), i_0)$ where
\begin{gather}
i_0 = \max(\sigma_1, \sigma_2) \\
\rho = \begin{cases} (1,2) & (\sigma_1 = 1), \\ (2,1) & (\sigma_2=1), \end{cases} \\
\sigma'_i = \begin{cases} \sigma_{i+2}-1 & (\sigma_{i+2} < i_0), \\ \sigma_{i+2}-2 & (\sigma_{i+2} > i_0). \end{cases}
\end{gather}
\end{construction}
Here, we write $\bigsqcup_{u=a}^b A_u$ to mean that $\bigsqcup_{u \in \sint{a}{b+1}} A_u$.
\begin{construction} \label{construction::7-3}
Let $x, b \in \mathbb{Z}$. Then, we have a compatible sijection with respect to $\lvert \eta_1 \rvert$,
\begin{align*}
	\bigsqcup_{u=1}^b J_1\left<1,x+2u\right> =& \bigsqcup_{u=1}^b (\{\lvert x+2u-1-1 \rvert \}, \{ \lvert 1-(x+2u)-1 \rvert \}) \\
	=& \bigsqcup_{u=1}^b (\{\lvert x+2u-2 \rvert \}, \{ \lvert x+2u \rvert \}) \\
	\sijto& (\{\lvert x \rvert\}, \{ \lvert x+2b \rvert \}).
\end{align*}
\end{construction}
\begin{construction} \label{construction::7-4}
Let $b_1, b_2, y \in \mathbb{Z}$. Then, we construct a compatible sijection with respect to $\eta_1$,
\begin{align*}
	\bigsqcup_{u_1=1}^{b_1} \bigsqcup_{u_2=1}^{b_2}
	J_{1} \left< y-b_1+2u_1, y-b_2+2u_2 \right> \sijto (\emptyset,\emptyset).
\end{align*}
We associate $(((1,2), u_2), u_1)$ with $(((2,1), b_2+1-u_2), b_1+1-u_1)$.
This is a sign reversing involution on the LHS. Furthermore, we have 
\begin{align*}
	\eta_1 \left((((1,2), u_2), u_1)\right) &= (y-b_1+2u_1) - (y-b_2+2u_2)  \\
	&= b_2-b_1 + 2(u_1-u_2),
\end{align*}
and
\begin{align*}
	\eta_1 \left((((2,1), b_2+1-u_2), b_1+1-u_1)\right) &= (y-b_2+2(b_2+1-u_2)) - (y-b_1+2(b_1+1-u_1)) \\
	&= b_2-b_1 + 2(u_1-u_2).
\end{align*}
Therefore, this sijection is compatible with $\eta_1$.
%By applying Construction~\ref{construction::7-3}, the LHS can be organized as follows:
%\begin{align}
%(\text{LHS})
%=& \bigsqcup_{u_1=1}^{b_1} \bigsqcup_{u_2=1}^{b_2}
%	J_{1} \left< 1, (b_1-2u_1-b_2+1)+2u_2 \right> \notag \\
%\sijto& \bigsqcup_{u_1=1}^{b_1}
%	(\{ \lvert b_1-2u_1-b_2+1 \rvert \}, \{ \lvert b_1-2u_1+b_2+1 \rvert \}). \label{sij::7-4-1}
%\end{align}
%Here, we have
%\[
%	\lvert b_1-2u_1-b_2+1 \rvert = \lvert b_1-2(b_1+1-u_1)+b_2+1 \rvert.
%\]
%Therefore, the RHS of (\ref{sij::7-4-1}) is in sijection to the empty signed set, and this construction is trivially compatible with $\lvert \eta_1 \rvert$.
\end{construction}

\begin{construction} \label{construction::7-5}
Let $m \in \mathbb{Z}_{>0}$ and $k \in \mathbb{Z}_{\geq 0}$ satisfy $m \geq 2$, $k<m$ and $k \equiv m \pmod{2}$.
In addition, let $b_1,b_2,\ldots, b_m, x_1, x_2,\ldots,x_k,y \in \mathbb{Z}$.
Then, we construct a compatible sijection with respect to $\lvert \eta_1 \rvert, \lvert \eta_2 \rvert, \ldots, \lvert \eta_m \rvert$,
\[
	\bigsqcup_{u_1=1}^{b_1} \bigsqcup_{u_2=1}^{b_2} \cdots \bigsqcup_{u_m=1}^{b_m}
	J_{\frac{m+k}{2}} \left< x_1,x_2,\ldots,x_k, y-b_m+2u_m, y-b_{m-1}+2u_{m-1}, \ldots, y-b_1+2u_1 \right> \sijto (\emptyset,\emptyset).
\]
First, we consider the case when $m=2$. Then we have $k=0$ and therefore this case reduces to Construction~\ref{construction::7-4}, which is compatible with $\lvert \eta_1 \rvert$.
For the case $m>2$ we apply Construction~\ref{construction::7-2}. Then, the LHS is in the form of
\[
	\bigsqcup_{i} (-1)^i J_2\left<\ast\right> \times \left( \bigsqcup_{u_\ast} \cdots \bigsqcup_{u_\ast} J_{\frac{m+k-2}{2}}\left<\ast\right> \right),
\]
where $\bigsqcup_{u_\ast} \cdots \bigsqcup_{u_\ast} J_{\frac{m+k-2}{2}}\left<\ast\right>$ match the form of the LHS for the case $(m,k) = (m-2,k)$ or $(m-1,k-1)$.
Therefore, by induction, this is in sijection to 
\[
	\bigsqcup_{i} (-1)^i J_2\left<\ast\right> \times (\emptyset,\emptyset) = (\emptyset,\emptyset).
\]
The compatibilities follow from the construction.
\end{construction}

\begin{construction} \label{construction::7-6}
Let $m \in \mathbb{Z}_{>0}$ and $a_1,a_2,\ldots, a_m, b_1,b_2,\ldots, b_m \in \mathbb{Z}$.
Then, we construct a compatible sijection with respect to $\lvert \eta_1 \rvert, \lvert \eta_2 \rvert, \ldots, \lvert \eta_m \rvert$,
\begin{multline*}
	I_m\Bigl[\begin{smallmatrix}
    a_1,a_2,\ldots,a_m \\
    b_1,b_2,\ldots,b_m
  \end{smallmatrix}\Bigr] \sijto 
  \bigsqcup_{u_1=1}^{b_1} \bigsqcup_{u_2=1}^{b_2} \cdots \bigsqcup_{u_m=1}^{b_m} \\
  	J_m\left< 1,1+a_1-a_2,\ldots,1+a_1-a_m,a_1-b_m+2u_m,a_1-b_{m-1}+2u_{m-1},\ldots,a_1-b_1+2u_1\right>.
\end{multline*}

For the case $m=1$, by Construction~\ref{construction::7-3}, we have
\[
	\bigsqcup_{u_1=1}^{b_1} J_1\left<1,a_1-b_1+2u_1\right> \sijto (\{ \lvert a_1-b_1 \rvert \}, \{ \lvert a_1 + b_1 \rvert \}) = I_1\Bigl[\begin{smallmatrix}
    a_1 \\
    b_1
  \end{smallmatrix}\Bigr].
\]

For the case $m>1$, we define $J_m'\left<\ast\right>$ by 
\[
	J_m'\left<\ast\right>^\pm = \left\{ \sigma \in J_m \left< \ast \right>^\pm \mid \forall i \in \{1,2,\ldots,m\}, \min(\sigma_{2i-1}, \sigma_{2i}) \leq m \right\}.
\]
Note that since $\sigma$ is a permutation, $\sigma \in \supp{J_m' \left< \ast \right>}$ satisfies $\forall i \in \{1,2,\ldots,m\}, \max(\sigma_{2i-1}, \sigma_{2i}) \geq m+1$,
and we have a compatible sijection (cf. Construction \ref{construction::7-4})
\begin{multline*}
  (\text{RHS}) 
  \sijto
  \bigsqcup_{u_1=1}^{b_1} \bigsqcup_{u_2=1}^{b_2} \cdots \bigsqcup_{u_m=1}^{b_m} \\
  	J_m'\left< 1,1+a_1-a_2,\ldots,1+a_1-a_m,a_1-b_m+2u_m,a_1-b_{m-1}+2u_{m-1},\ldots,a_1-b_1+2u_1\right>.
\end{multline*}
Here, we associate $(\sigma, u_1, u_2, \ldots, u_m) \in \supp{ \sqcup \sqcup \cdots \sqcup J_m \left< \ast \right>} \setminus \supp{ \sqcup \sqcup \cdots \sqcup J_m' \left< \ast \right> }$ with
\[
	(\sigma', u_1', u_2', \ldots, u_m'),
\]
where let $j$ be the smallest index with $\min(\sigma_{2j-1} , \sigma_{2j}) \geq m+1$ and 
\begin{align*}
	\sigma'_i &= \begin{cases} \sigma_{2j} & (i = 2j-1), \\ \sigma_{2j-1} & (i = 2j), \\ \sigma_i & (\text{otherwise}), \end{cases} \\ 
	u_i' &= \begin{cases} b_i + 1 - u_ i & (i \in \{m+1-(\sigma_{2j-1}-m), m+1-(\sigma_{2j}-m)\}), \\ u_i & (\text{otherwise}). \end{cases}
\end{align*}

Applying a restriction to Construction~\ref{construction::7-2} to $J_m' \left< \ast \right>$, we have
\begin{multline*}
	(\text{RHS}) \sijto 
	\bigsqcup_{i=1}^{m}(-1)^{2m+1-i} \bigsqcup_{u=1}^{b_i} J_1\left<1,a_1-b_i + 2u\right> 
	\times \bigsqcup_{u_1=1}^{b'_1} \bigsqcup_{u_2=1}^{b'_2} \bigsqcup_{u_{m-1}=1}^{b'_{m-1}} J_{m-1}' \\
		\left<1+a_1-a_2, \ldots, 1+a_1-a_m, a_1-b'_{m-1}+2u_{m-1},a_1-b'_{m-2}+2u_{m-2},a_1-b'_1+2u_1 \right>,
\end{multline*}
where $b'_1,b'_2,\ldots,b'_{m-1}$ is the array obtained by removing $b_i$ from $b_1,b_2,\ldots,b_m$.
By a similar restriction of Construction~\ref{construction::7-1} to $J_m' \left< \ast \right>$ we have
\begin{multline*}
	J_{m-1}' \left<1+a_1-a_2, \ldots, 1+a_1-a_m, a_1-b'_{m-1}+2u_{m-1},a_1-b'_{m-2}+2u_{m-2},a_1-b'_1+2u_1 \right> \\
	\sijto J_{m-1}' \left<1, \ldots, 1+a_2-a_m, a_2-b'_{m-1}+2u_{m-1},a_2-b'_{m-2}+2u_{m-2},a_2-b'_1+2u_1 \right>,
\end{multline*}
and thus, by induction, from the case for $m-1$ and Construction~\ref{construction::7-3} we obtain
\[
	(\text{RHS}) \sijto 
	\bigsqcup_{i=1}^{m}(-1)^{i-1} (\{ \lvert a_1-b_i \rvert \}, \{ \lvert a_1+b_i \rvert \})
	\times I_m\Bigl[\begin{smallmatrix}
    a_2,\ldots,a_m \\
    b'_1,b'_2,\ldots,b'_{m-1}
  \end{smallmatrix}\Bigr],
\]
whose RHS is in sijection to $I_m\Bigl[\begin{smallmatrix}
    a_1,a_2,\ldots,a_m \\
    b_1,b_2,\ldots,b_m
  \end{smallmatrix}\Bigr]$, by definition.
The compatibilities follow from the construction.
\end{construction}

\begin{construction} \label{construction::7-7}
Let $m \in \mathbb{Z}_{>0}$ and $x_1,x_2,\ldots, x_{2m} \in \mathbb{Z}$.
When $x_i = x_j$ for $1 \leq i < j \leq 2m$, we obtain a compatible sijection
\[
	J_m \left< x_1,x_2,\ldots,x_{2m} \right> \sijto (\emptyset,\emptyset),
\]
by associating $\sigma \in S'_{2m} = \supp{J_m\left<\ast\right>}$ with $[(i\,j) \circ \sigma]$,
where $(i\,j)$ is the transposition and $[\tau]$ means the representative of the $S_m$-orbit of $\tau$ under the action described in (\ref{def::action}).
In addition, if $i>1$, then this sijection is compatible with $\eta_1$.
\end{construction}

\begin{construction} \label{construction::7-8}
In the setting of Construction~\ref{construction::7-6}, we set $a_i = b_i = m+1-i$.
Then, the sijection is organized as:
\begin{align*}
	\hspace{1cm} & \hspace{-1cm}I_m\Bigl[\begin{smallmatrix}
    m,m-1,\ldots,1 \\
    m,m-1,\ldots,1
  \end{smallmatrix}\Bigr] \\
  \sijto& 
  \bigsqcup_{u_1=1}^{b_1} \bigsqcup_{u_2=1}^{b_2} \cdots \bigsqcup_{u_m=1}^{b_m} 
  	J_m\left< 1,2,\ldots,m ,a_1-b_m+2u_m,a_1-b_{m-1}+2u_{m-1},a_1-b_1+2u_1\right> \\
  \sijto& 
  \bigsqcup_{u_1=1}^{b_1} \bigsqcup_{u_2=1}^{b_2} \cdots \bigsqcup_{u_m=1}^{b_m} 
  	J_m\left< 1,2,\ldots,m ,m-1+2u_m, m-2+2u_{m-1},m-m+2u_1\right> \\
=& \bigsqcup_{c_1 = m+1} \bigsqcup_{c_2 = m, m+2} \bigsqcup_{c_3 = m-1, m+1, m+3} \cdots \bigsqcup_{c_m = 2,4,\ldots, 2m} J_m\left<1,2,\ldots,m,c_1,c_2,\ldots,c_m\right>.
\end{align*}
Except in the case when $c_i = m+i$ for all $i =1,2,\ldots,m$, the signed set $J_m\left<1,2,\ldots,m,c_1,c_2,\ldots,c_m\right>$ is in sijection to $(\emptyset,\emptyset)$, by Construction~\ref{construction::7-7}. If there are two or more applicable pairs $(i,j)$ at which we can apply $\varphi_\text{\ref{construction::7-7}}$, we choose a minimum pair in lexicographic order.
Therefore, we obtain the desired sijection
\[
	I_m\Bigl[\begin{smallmatrix}
    m,m-1,\ldots,1 \\
    m,m-1,\ldots,1
  \end{smallmatrix}\Bigr] 
  \sijto J_m \left<1,2,\ldots,2m\right>,
\]
together with the family of permutations $\{\sigma_s\}$, which can be recovered from the construction (and, in particular, from the details of Construction~\ref{construction::7-7}).
\end{construction}

\begin{example} \label{example::7-9}
We calculate an element of $J_3 \left< 1,2,\ldots, 6 \right>$ that corresponds to $( (1,2,3), (0,0,0) ) \in$ $I_3\Bigl[\begin{smallmatrix}
    3,2,1 \\
    3,2,1
  \end{smallmatrix}\Bigr] $. This example will be used in the next section, where we explicitly calculate an SPP corresponding to a particular given  QTCPP.
  
Let $\varphi_{\ref{construction::7-6}}$ be the sijection $$I_3\Bigl[\begin{smallmatrix}
    3,2,1 \\
    3,2,1
  \end{smallmatrix}\Bigr]
  \sijto
  \bigsqcup_{u_1=1}^{b_1} \bigsqcup_{u_2=1}^{b_2} \bigsqcup_{u_3=1}^{b_3} 
  	J_m\left< 1,2,3 ,2+2u_3, 1+2u_2,2u_1\right>  $$ used in Construction \ref{construction::7-8},
and let $\varphi_{\ref{construction::7-7}}$ be the sijection $$
  \bigsqcup_{u_1=1}^{b_1} \bigsqcup_{u_2=1}^{b_2} \bigsqcup_{u_3=1}^{b_3} 
  	J_m\left< 1,2,3 ,2+2u_3, 1+2u_2,2u_1\right> 
\sijto J_m \left< 1,2,\ldots, 6 \right>.$$
We set $\varphi_{\ref{construction::7-8}} = \varphi_{\ref{construction::7-7}} \circ \varphi_{\ref{construction::7-6}}$. Then, we have
\begin{multline*}
( (1,2,3), (0,0,0) ) \in I_m\Bigl[\begin{smallmatrix}
    3,2,1 \\
    3,2,1
  \end{smallmatrix}\Bigr]
\overset{(1)}{\underset{\varphi_{\ref{construction::7-6}}}{\mapsto}}
\bigl( (1,6,2,5,3,4),\, u=(1,1,1) \bigr)\\
\overset{(2)}{\underset{\varphi_{\ref{construction::7-7}}}{\mapsto}}
\bigl( (1,6,2,3,5,4),\, u=(1,1,1) \bigr)
\overset{(3)}{\underset{\varphi_{\ref{construction::7-6}}}{\mapsto}}
\bigl( (1,6,2,3,4,5),\, u=(1,2,1) \bigr)
\overset{(4)}{\underset{\varphi_{\ref{construction::7-7}}}{\mapsto}}
\bigl( (1,2,6,3,4,5),\, u=(1,2,1) \bigr) \\
\overset{(5)}{\underset{\varphi_{\ref{construction::7-6}}}{\mapsto}}
\bigl( (1,2,6,3,5,4),\, u=(1,1,1) \bigr)
\overset{(6)}{\underset{\varphi_{\ref{construction::7-7}}}{\mapsto}}
\bigl( (1,2,3,4,6,5),\, u=(1,1,1) \bigr)
\overset{(7)}{\underset{\varphi_{\ref{construction::7-6}}}{\mapsto}}
\bigl( (1,2,3,4,5,6),\, u=(3,2,1) \bigr) \\
\overset{(8)}{\underset{\varphi_{\ref{construction::7-7}}}{\mapsto}}
(1,2,3,4,5,6) \in J_m\left<1,2,\ldots,6\right>.
\end{multline*}
\begin{itemize}
\item[(1)]
The sijection $\varphi_\text{\ref{construction::7-6}}$ sends $$(\sigma, t) \in I_m\Bigl[\begin{smallmatrix} m,m-1,\ldots, 1 \\m,m-1,\ldots,1\end{smallmatrix}\Bigr]$$ to
$$(\sigma', u) \in  \bigsqcup_{u_1=1}^{b_1} \bigsqcup_{u_2=1}^{b_2} \cdots \bigsqcup_{u_m=1}^{b_m} 
  	J_m\left< 1,2,\ldots,m ,a_1-b_m+2u_m,a_1-b_{m-1}+2u_{m-1},a_1-b_1+2u_1\right>,$$
where
\begin{align*} \{\sigma_{2i-1}, \sigma_{2i}\} = \{ i, 2m+1-\sigma_i \},&& \sigma_{2i-1} < \sigma_{2i} \Leftrightarrow t_i = 0. \end{align*}

\item[(2)]
The LHS is in the component $J_3\langle 1,2,3,4,3,2\rangle$, since $u=(1,1,1)$.
Therefore, $\varphi_\text{\ref{construction::7-7}}$ swaps $3$ and $5$.

\item[(3)]
Let $(\sigma, u)$ be the LHS. Since $\min(\sigma_5, \sigma_6) > 3$, $\varphi_\text{\ref{construction::7-6}}$ swaps $\sigma_5$ and $\sigma_6$,
and replaces $u_i$ with $b_i+1-u_i$ for $i = \sigma_5 - 3, \sigma_6-3$.

\item[(4)]
Since $u=(1,2,1)$, the LHS is in $J_3\langle 1,2,3,4,5,2\rangle$. Therefore, $\varphi_\text{\ref{construction::7-7}}$ swaps $2$ and $6$.

\item[(5)]
Similar as (3).

\item[(6)]
Since $u=(1,1,1)$, similar as in (2), we swap $3$ and $5$, and obtain $\sigma=(1,2,6,5,3,4)$.
Because $\min(\sigma_3, \sigma_4) > \min(\sigma_5, \sigma_6)$, we swap these elements and obtain $\sigma=(1,2,3,4,6,5)$.
Due to this operation, the refinement of this sijection is $(\begin{smallmatrix} 1 2 3 \\ 1 3 2 \end{smallmatrix})$.

\item[(7)]
Similar as (3).

\item[(8)]
Since $u=(3,2,1)$, the LHS is in $J_3\langle 1,2,3,4,5,6\rangle$.
\end{itemize}
Since only part (6) involves the refinement, the refinement of this sijection is $(\begin{smallmatrix} 1 2 3 \\ 1 3 2 \end{smallmatrix})$.

\medskip
Similarly, we can calculate $\varphi_{\ref{construction::7-8}}\bigl(((1,3,2),(0,0,0))\bigr)$ as follows:
\begin{multline*}
( (1,3,2), (0,0,0) ) \in I_m\Bigl[\begin{smallmatrix}
    3,2,1 \\
    3,2,1
  \end{smallmatrix}\Bigr] \underset{\varphi_{\ref{construction::7-6}}}{\mapsto} ( (1,6,2,4,3,5), u=(1,1,1) )\\
\underset{\varphi_{\ref{construction::7-7}}}{\mapsto} ( (1,6,2,4,5,3), u=(1,1,1) ) 
\overset{(\ast)}{\underset{\varphi_{\ref{construction::7-6}}}{\mapsto}} ( (1,6,2,4,3,5), u=(1,2,1) )
\overset{(\#)}{\underset{\varphi_{\ref{construction::7-7}}}{\mapsto}} ( (1,2,3,5,6,4), u=(1,2,1) ) \\
\underset{\varphi_{\ref{construction::7-6}}}{\mapsto} ( (1,2,3,5,4,6), u=(3,2,1) ) 
\underset{\varphi_{\ref{construction::7-7}}}{\mapsto} (1,2,3,5,4,6) \in J_m\left<1,2,\ldots,6\right>.
\end{multline*}
At (*), we have $\min(\sigma_{2i-1}, \sigma_{2i}) = i$ for all $i \in \{1,2,3\}$, $\sigma_1 < \sigma_2$ with $u_{\max(\sigma_1,\sigma_2)-3} = 1$, $\sigma_3 < \sigma_4$ with $u_{\max(\sigma_3,\sigma_4)-3} = 1$, and $\sigma_5 > \sigma_6$ but $u_{\max(\sigma_5,\sigma_6)-3} = 1 \ne b_{\max(\sigma_5,\sigma_6)-3}$.
Therefore, we swap $\sigma_5$ and $\sigma_6$, and replace $u_{\max(\sigma_5,\sigma_6)-3}$ with $u_{\max(\sigma_5,\sigma_6)-3} + 1$, according to Construction~\ref{construction::7-3}. The remaining steps are similar to the previous example.
Only (\#) involves the refinement, which is $(\begin{smallmatrix} 1 2 3 \\ 1 3 2 \end{smallmatrix})$.
\end{example}

\section{A Full Example}

In this section, we illustrate the entire construction of the bijection $f \colon \QTCPP{n}{M} \to \SPP{n}{M}$ by explicitly computing
$f\left(\pi := \,\vcenter{\hbox{\young(6553,651,52,5)}}\,\right)$, where $n = 4$ and $M = 6$.
First, we apply $f_{\ref{1-2s::qtc-sta-even}}$. Consequently, we obtain
\begin{align*}
	f_{\ref{1-2s::qtc-sta-even}}(\pi) = \left(\,\vcenter{\hbox{\young(3220,322,21,2)}}\,, (\cdot, 1,1,0)\right) =: (\pi', t),
%	f_{\ref{1-2s::qtc-sta-even}}(\pi) &= (\pi', t), &
%	\pi' &=  \,\vcenter{\hbox{\young(3220,322,21,2)}}\,, &
%	t &= (\cdot, 1,1,0),
\end{align*}
where $\pi' \in \stairPP{4}{3}$, $S(\pi') = \{2,3,4\}$, and $t \colon S(\pi') \to \{0,1\}$.

Next, we apply the sijection $\varphi_{5.1} \colon (\stairPP{4}{3}, S) \sijto (\stairPP{4}{3}, \widetilde{S})$ constructed in Section 5.1.
First, $\pi'$ corresponds to a set of lattice paths:
\[
    \begin{tikzpicture}[scale=0.6, baseline=(current bounding box.center)]
        \fill[red] (1,-1) circle (2pt);
        \node[above] at (1,-1) {$a_1$};
        \fill[red] (2,-2) circle (2pt);
        \node[above] at (2,-2) {$a_2$};
        \fill[red] (3,-3) circle (2pt);
        \node[above] at (3,-3) {$a_3$};
        \fill[red] (4,-4) circle (2pt);
        \node[above] at (4,-4) {$a_4$};
        \fill[red] (-3,-4) circle (2pt);
        \node[below] at (-3,-4) {$b_1$};
        \fill[red] (-1,-5) circle (2pt);
        \node[below] at (-1,-5) {$b_2$};
        \fill[red] (1,-6) circle (2pt);
        \node[below] at (1,-6) {$b_3$};
        \fill[red] (3,-7) circle (2pt);
        \node[below] at (3,-7) {$b_4$};
        \draw[blue, thick](1,-1)--(-1,-1)--(-1,-2)--(-3,-2)--(-3,-4);
        \draw[blue, thick](2,-2)--(2,-3)--(0,-3)--(0,-4)--(-1,-4)--(-1,-5);
        \draw[blue, thick](3,-3)--(3,-4)--(1,-4)--(1,-6);
        \draw[blue, thick](4,-4)--(4,-7)--(3,-7)--(3,-7);
    \end{tikzpicture}.
\]

To apply the LGV lemma, we must fix the order of vertices. We consider the edges to be directed from north-east to south-west and we order the vertices by
\[
	(x_1, y_1) < (x_2, y_2) \Leftrightarrow x_1+y_1 > x_2+y_2 \lor (x_1+y_1=x_2+y_2 \land x_1>x_2).
\]

Then, the lattice paths are transformed as follows (where purple squares mean minimum intersections at which the LGV lemma sijection switch paths,
and red circles, triangles and crosses mean the involved paths are associated with $2,3$ and $4 \in S(\pi')$, respectively):
\begin{figure}[h]
\begin{center}
\includegraphics{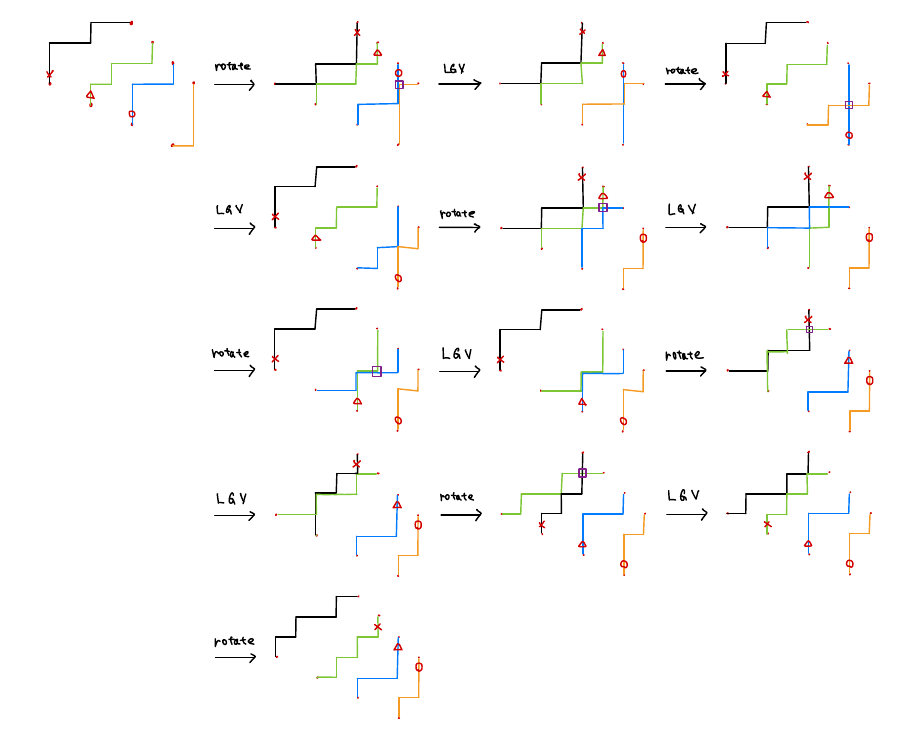}
\end{center}
\vspace{-25pt}
\caption{An example of the use of the double LGV lemma technique.}
\label{fig::4}
\end{figure}

Consequently, we obtain $\varphi_{5.1}(\pi') = \,\vcenter{\hbox{\young(3211,211,20,1)}}\,=: \pi''$, and the associated statistics and refinement are
\begin{align*}
	S(\pi') = \{2,3,4\},&& \widetilde{S}(\pi'') = \{2,3,4\},&& r \colon 2 \mapsto 4, 3 \mapsto 3, 4 \mapsto 2.
\end{align*}
Next, we apply the sijection
\begin{multline*}
	\Phi = \varphi_6 \circ \varphi_7 \circ \varphi_5 \colon
	(\stairPP{4}{3}, \widetilde{S})
	\overset{\varphi_5}{\sijto} \left(
  \bigsqcup_{\alpha \in I_3\Bigl[\begin{smallmatrix}
    3,2,1 \\ 3,2,1
  \end{smallmatrix}\Bigr]}
  \prod_{i=1}^3 \mathcal{C}(2n,\, n+\lvert \eta_i \rvert(\alpha)), \,\,\widetilde{S}  \right) \\
  \overset{\varphi_7}{\sijto}
  \left( \bigsqcup_{\sigma \in J_3\left<1,2,\ldots,6\right>} \prod_{i=1}^3 \mathcal{C}(2n,n+\lvert\eta_i\rvert(\sigma)),\,\, S  \right)
	\overset{\varphi_6}{\sijto} (\eSPP{4}{3}, S),
\end{multline*}
where $\varphi_5$, $\varphi_6$ and $\varphi_7$ are the sijections constructed in Sections 5, 6 and 7, respectively.
First, $\pi''$ is associated with $\alpha = ((1,2,3), (0,0,0)) \in I_3$ and paths
\[
    \begin{tikzpicture}[scale=0.6, baseline=(current bounding box.center)]
        \fill[red] (1,-1) circle (2pt);
        \node[below] at (1,-1) {$a_1$};
        \fill[red] (2,-2) circle (2pt);
        \node[below] at (2,-2) {$a_2$};
        \fill[red] (3,-3) circle (2pt);
        \node[below] at (3,-3) {$a_3$};
        \fill[red] (5,3) circle (2pt);
        \node[above] at (5,3) {$b_1$};
        \fill[red] (6,2) circle (2pt);
        \node[above] at (6,2) {$b_2$};
        \fill[red] (7,1) circle (2pt);
        \node[above] at (7,1) {$b_3$};
        \draw[black, thick](1,-1)--(1,2)--(2,2)--(2,3)--(5,3);
        \draw[orange, thick](2,-2)--(2,-1)--(3,-1)--(3,1)--(4,1)--(4,2)--(6,2);
        \draw[blue, thick](3,-3)--(4,-3)--(4,-1)--(6,-1)--(6,0)--(7,0)--(7,1);
    \end{tikzpicture}.
\]

According to Example \ref{example::7-9}, $\varphi_7$ sends this to $\beta = (1,2,3,4,5,6) \in J_3$ and paths
\[
    \begin{tikzpicture}[scale=0.6, baseline=(current bounding box.center)]
        \fill[red] (1,-1) circle (2pt);
        \node[below] at (1,-1) {$a_1$};
        \fill[red] (2,-2) circle (2pt);
        \node[below] at (2,-2) {$a_2$};
        \fill[red] (3,-3) circle (2pt);
        \node[below] at (3,-3) {$a_3$};
        \fill[red] (5,3) circle (2pt);
        \node[above] at (5,3) {$b_1$};
        \fill[red] (6,2) circle (2pt);
        \node[above] at (6,2) {$b_2$};
        \fill[red] (7,1) circle (2pt);
        \node[above] at (7,1) {$b_3$};
        \draw[black, thick](1,-1)--(1,2)--(2,2)--(2,3)--(5,3);
        \draw[orange, thick](3,-3)--(3,-2)--(4,-2)--(4,0)--(5,0)--(5,1)--(7,1);
        \draw[blue, thick](2,-2)--(3,-2)--(3,0-.05)--(5+.05,0-.05)--(5+.05,1-.05)--(6,1-.05)--(6,2);
    \end{tikzpicture}.
\]

Note that the order of the paths is permuted  by $\sigma = \left(\begin{smallmatrix} 1 2 3 \\ 1 3 2 \end{smallmatrix} \right)$, a permutation associated with $\alpha$ (and $\beta$) defined in \eqref{sigma7}.
By cutting, translating and reflecting, these lattice paths are associated with
\begin{figure}[h]
\begin{center}
\includegraphics[scale=.8]{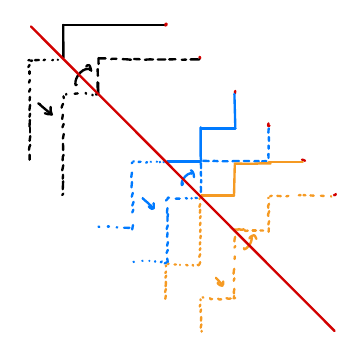}
\vspace{-20pt} 
\end{center}
\caption{Cutting, translating and reflecting lattice paths as explained in Section~5.}
\label{fig::5}
\end{figure}

We consider the edges to be directed from south-west to north-east in the construction of $\varphi_6$, and we order the vertices by
\[
	(x_1, y_1) < (x_2, y_2) \Leftrightarrow x_1+y_1 < x_2+y_2 \lor (x_1+y_1=x_2+y_2 \land x_1<x_2).
\]
Then, the lattice paths is transformed as:
\begin{center}
\includegraphics[scale=.8]{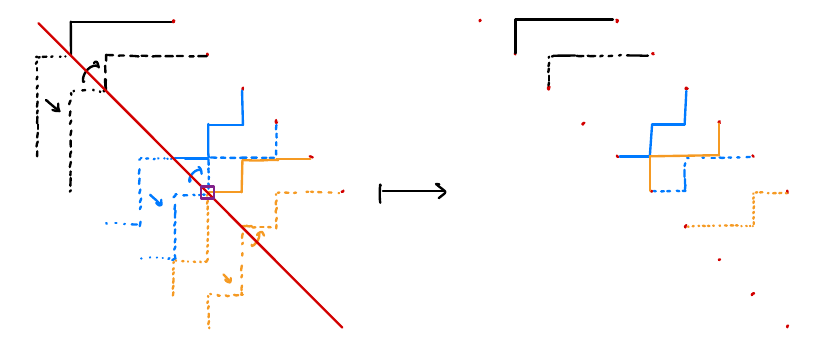}
\end{center}
Note that since the starting points are not distinct in this example, the LGV lemma sijection switches $P_4$ (blue) and $P_5$ (dotted orange).
Since $\beta$ designates the order of the endpoints, the LGV lemma sijection transforms $\beta \to \beta' = (1,2,3,5,4,6)$.
By reflecting, translating and concatenating, the lattice paths are transformed as:
\begin{center}
\includegraphics[scale=.8]{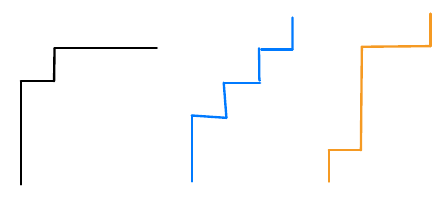}
\end{center}
By applying $\varphi_7$, according to Example \ref{example::7-9}, we obtain $\alpha' = ((1,3,2),(0,0,0))$ and lattice paths:
\begin{center}
\includegraphics[scale=.8]{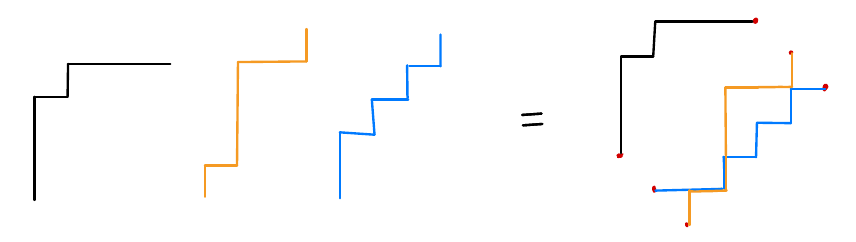}
\end{center}
Subsequently, by applying $\varphi_5$ and $\varphi_7$, this is transformed as follows:
\begin{center}
\includegraphics[scale=.8]{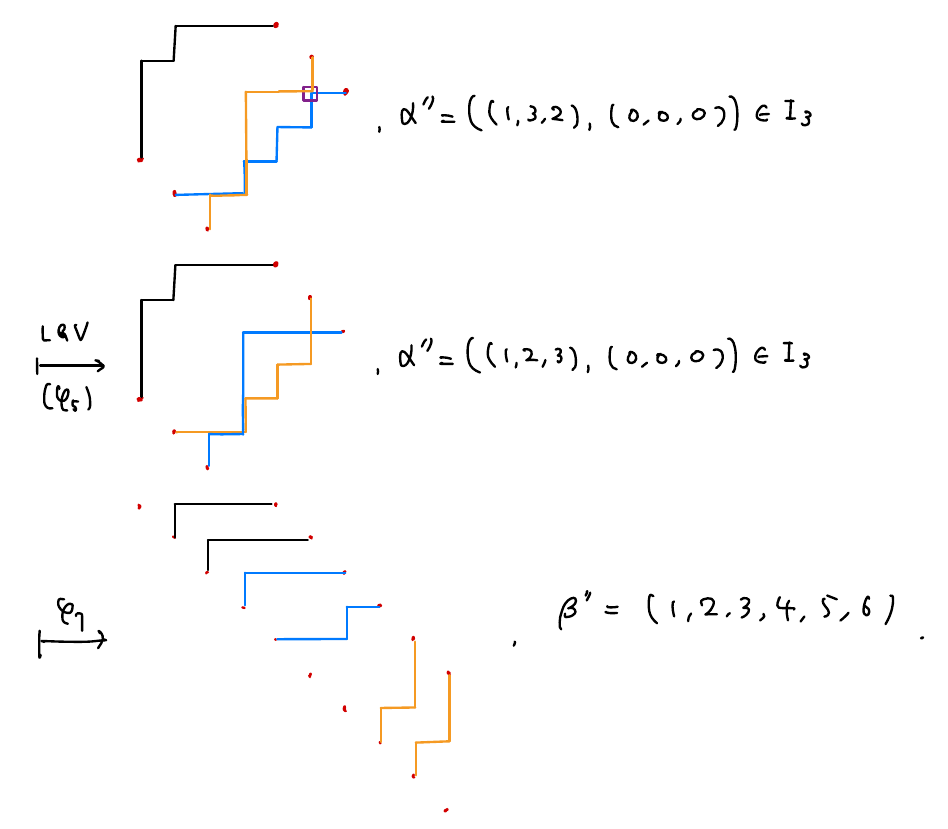}
\end{center}
Therefore, by applying $\varphi_6$, we obtain $\Phi(\pi'') = \,\vcenter{\hbox{\young(6332,:222,::20,:::0)}}\,=: \pi''' \in \eSPP{4}{3}$.
In addition, $\widetilde{S}(\pi'') = S(\pi''')$ and the refinement involving $\Phi$ is the identity map for $\pi''$.
Consequently, the refinement involving the full sijection $\Phi \circ \varphi_{5.1}$ is $\id \circ\, r = r$ for $\pi'$.

Finally, we calculate $f_{\ref{construction::spp-espp}}^{-1}\left( \pi''', t \circ r^{-1} = (\cdot, 0,1,1)\right)$,
where $f_{\ref{construction::spp-espp}}$ is the bijection between $\SPP{n}{2m}$ and $\bigsqcup_{\pi' \in \eSPP{n}{m}} \{0,1\}^{S(\pi')}$ constructed in Construction~\ref{construction::spp-espp}.
The result is 
\begin{align*}
\phi_\text{conj}^{-1}\left( \left(\phi_\text{conj}\left( \,\vcenter{\hbox{\young(6332,:222,::20,:::0)}}\, \right) \leftarrow 3\right) \leftarrow 4 \right)
&= \phi_\text{conj}^{-1}\left( \left(\,\vcenter{\hbox{\young(443111,33,11)}}\, \leftarrow 3\right) \leftarrow 4 \right) \\
&= \phi_\text{conj}^{-1}\left( \,\vcenter{\hbox{\young(443311,331,11)}}\, \leftarrow 4 \right) \\
&= \phi_\text{conj}^{-1}\left( \,\vcenter{\hbox{\young(444311,333,111)}}\, \right) \\
&= \,\vcenter{\hbox{\young(6443,:333,::30,:::0)}}\,.
\end{align*}

Hence, the symmetric plane partition $\,\vcenter{\hbox{\young(6443,4333,4330,3300)}}\,$ corresponds to the quasi transpose complementary plane partition
$\,\vcenter{\hbox{\young(6553,6511,5211,5100)}}\,$, under the constructed sijection.

\section{Conclusion} \label{sec:conclusion}

In this paper we constructed an explicit bijection between $\SPP{n}{M}$ and $\QTCPP{n}{M}$, thus resolving the bijective problem raised by Schreier--Aigner and Proctor. The main idea was to combine the $1:2^{\#S}$ correspondences on both sides with the double--LGV technique, and to organize the resulting structure via the signed index sets $I_m$ and $J_m$. This framework not only recovers the known equinumerosity results of Proctor, but also enhances them with compatibility,
which enabled us to construct the present bijection.

\medskip

\noindent
\textbf{Future directions.}
Several natural problems remain open:
\begin{itemize}
  \item The combinatorial sets $I_m$ and $J_m$ deserve further study. In particular, the structure of $J_m$ strongly suggests a connection with Pfaffian evaluations. We expect applications to classical Pfaffian formulas.
  Moreover, the study of the signed generating functions of $I_m$ and $J_m$ presents another interesting direction.
  \item Our method should extend to other bijection problems, especially those involving reflecting boundaries or parity restrictions. Examples include TSSCPPs, Gog--Magog correspondences, and further symmetry classes.
  \item It would be desirable to simplify the current constructions and discover additional compatibilities.
  For example, if we fully understand the structure of $\{\sigma_s\}_{s \in \supp{I_m} \bigsqcup \supp{J_m}}$ in Section~7, we can construct $m$ (independent) statistics from the construction.
  At least, we should be able to generate the second statistics, since the multi-set of paths are preserved.
\end{itemize}

\medskip

Overall, the framework of signed sets and compatible sijections, together with the new indexing objects $I_m$ and $J_m$, offers a flexible tool that we hope will stimulate further developments in bijective combinatorics and its connections with determinantal and Pfaffian structures.

%\subsection{The Entire Structure of the Sijection between QTCPPs and SPPs}

%example
%\subsection{Further Discussions}

\section*{Acknowledgements}
I would like to thank my supervisor, Prof. Ralph Willox, for helpful discussions and comments on the manuscript.
The author is partially supported by FoPM, WINGS Program, the University of Tokyo and JSPS KAKENHI No.\ 23KJ0795.

\end{document}